%% file: main.tex
\documentclass{amsart}
\usepackage[utf8]{inputenc}

\usepackage{tikzit}
\input{sample.tikzstyles}

\usepackage{amsxtra, amsfonts, amssymb, amstext, amsmath, mathtools}
\usepackage{amsthm}
\usepackage{stmaryrd}
\usepackage{hyperref}
\usepackage{booktabs}
\usepackage{nicefrac}
\usepackage{xspace}
\usepackage{url}
\usepackage{graphics,color}
\usepackage[algo2e,ruled,vlined,linesnumbered]{algorithm2e}
\usepackage[top=3.5cm,bottom=3.5cm,right=3.5cm,left=3.5cm]{geometry}
\usepackage{wrapfig}
\usepackage{algorithm}
\usepackage{algorithmic}
\usepackage{dsfont}
\usepackage{mathrsfs}
\usepackage{comment}
\usepackage{lscape}
\usepackage{float}

\newtheorem{theorem}{Theorem}[section]
\newtheorem{corollary}[theorem]{Corollary}
\newtheorem{lemma}[theorem]{Lemma}
\newtheorem{proposition}[theorem]{Proposition}
\newtheorem{question}[theorem]{Question}

\theoremstyle{remark}
\newtheorem{remark}[theorem]{Remark}

\newtheorem{example}[theorem]{Example}

\theoremstyle{definition}
\newtheorem{definition}[theorem]{Definition}

\newcommand{\Q}{\ensuremath{\mathbb{Q}}}
\newcommand{\N}{\ensuremath{\mathbb{N}}}
\newcommand{\Z}{\ensuremath{\mathbb{Z}}}
\newcommand{\F}{\ensuremath{\mathbb{F}}}
\newcommand{\Or}{\ensuremath{\mathcal{O}}}
\newcommand{\G}{\ensuremath{\mathcal{G}}}

\newcounter{todocnth}

\newcounter{todocntfa}

\newcounter{todocntfr}

\DeclareMathOperator{\en}{End}
\DeclareMathOperator{\dg}{deg}
\DeclareMathOperator{\Ker}{ker}

\DeclareMathOperator{\Cl}{Cl}

\DeclareMathOperator{\ord}{ord}

\DeclareMathOperator{\Gal}{Gal}

\begin{document}
\title{Ordinary isogeny graphs over $\mathbb{F}_p$: the inverse volcano problem}
\author{{H}enry {B}ambury,
{F}rancesco {C}ampagna,
{F}abien {P}azuki}

\address{Henry Bambury. Ecole Polytechnique,
Institut Polytechnique de Paris, Palaiseau, France.} \email{henry.bambury@polytechnique.edu}

\address{Francesco Campagna.
Max-Planck-Institut f\"ur Matematik, Vivatsgasse 7, 53111 Bonn, Germany. }
\email{campagna@mpim-bonn.mpg.de}

\address{Fabien Pazuki. University of Copenhagen, Institute of Mathematics, Universitetsparken 5, 2100 Copenhagen, Denmark, and Universit\'e de Bordeaux, IMB, 351, cours de la Lib\'eration, 33400 Talence, France. }
\email{fpazuki@math.ku.dk}

\date{}
\maketitle


\begin{abstract}
We give a detailed presentation of $\ell$-isogeny graphs associated with ordinary elliptic curves defined over $\mathbb{F}_p$. We then focus on the following inverse problem: given an abstract volcano $V$, do there always exist primes $\ell, p \in \mathbb{N}$ such that the ordinary $\ell$-isogeny graph over $\mathbb{F}_p$ contains $V$ as a connected component? We provide an affirmative answer to this question.
\end{abstract}

{\flushleft
\textbf{Keywords:} Isogenies, modular polynomials.\\
\textbf{Mathematics Subject Classification:} 11G18, 11G20, 14G17, 14K02. }

\section{Introduction}

Given a finite field $\mathbb{F}$ of characteristic $p$ and a prime number $\ell \neq p$, one can consider the so-called \textit{$\ell$-isogeny graph over $\mathbb{F}$}. Roughly speaking, one can construct this graph $\mathcal{G}$ by taking as vertices the elements of $\mathbb{F}$, which in this context are seen as the $j$-invariants of elliptic curves defined over $\mathbb{F}$, and by drawing an oriented edge from one vertex to another if there is a geometric isogeny of degree $\ell$ between the corresponding elliptic curves. The graph $\mathcal{G}$ naturally decomposes into two subgraphs $\mathcal{G}_\text{ord}$ and $\mathcal{G}_\text{ss}$, which are respectively induced by the elements of $\mathbb{F}$ that are $j$-invariants of ordinary and supersingular elliptic curves. Picturally, the structural difference between these two subgraphs is clear.

The supersingular isogeny graph $\mathcal{G}_\text{ss}$ has been first studied in \cite{Pizer_90} as an example of large Ramanujan graph. Due to their complicated structure, supersingular isogeny graphs have found several cryptographic applications, see for instance \cite{hash_09} and \cite{DeFeo_14}. On the other hand, the ordinary isogeny graph $\mathcal{G}_\text{ord}$ has a much more regular structure which has first been studied in the seminal work of Kohel \cite{Koh96} and subsequently by Fouquet and Morain \cite{FouM02}. In particular, in \cite{FouM02} the authors coined the term \textit{isogeny volcano} to denote connected components of $\mathcal{G}_\text{ord}$. This terminology is justified by the fact that the connected components of ordinary isogeny graphs often appear as a cycle, the \textit{crater}, whose vertices are roots of isomorphic trees, the \textit{lava flows}. Because of this satisfyingly regular structure, ordinary isogeny graphs have attracted a lot of attention, finding applications both in computational number theory (see for instance \cite{Sut12,Sut13}) and in  cryptography (see for instance \cite{Mir07}).

We focus here on ordinary isogeny graphs over $\mathbb{F}_p$ and the text consists of two parts. The first part is essentially a review of known results on the structure of $\mathcal{G}_\text{ord}$. We gather in a unique place the terminology that appeared in different works and provide details when needed. 
Within what we call the \textit{volcano park}, we thus identify the \textit{cordilleras} in Section \ref{subsec:cordilleras}, the \textit{belts} in Section \ref{subsec:belts}, and finally the \textit{volcanoes} in Section \ref{subsec:volcanoes}. We also treat in full detail the pathological cases corresponding to the $j$-invariants $0$ and $1728$ (see for instance Propositions \ref{prop:zero} and \ref{prop:1728}). The geological lexicon we use is sometimes new and sometimes borrowed from previous works; for instance, the term ``cordillera" first appeared in \cite{Mir07}.

In the second part of this manuscript we solve the following inverse problem: suppose we are given an abstract volcano graph $V$, \textit{i.e.} a graph that looks like a genuine volcano (see Definition \ref{volcano:graph}); does there then exist a pair of distinct primes $\ell$ and $p$ such that the ordinary $\ell$-isogeny graph over $\mathbb{F}_p$ has $V$ as a connected component? If the target abstract volcano is given only as a crater with no lava flows, the answer to this question is easier since we have a lot of freedom in the choice of $\ell$ and $p$. We prove the following result (see again Definition \ref{volcano:graph} for terminology):

\begin{theorem}
\label{thm:abstract_crater_depth0_intro}
Let $V$ be an abstract volcano of depth $0$. Then there exists infinitely many distinct primes $p, \ell \in \mathbb{Z}$ such that $V$ is a connected component of the $\ell$-isogeny graph over $\mathbb{F}_p$.
\end{theorem}

On the other hand, if $V$ has lava flows then the inverse volcano problem becomes more difficult, since the theory of isogeny volcanoes now fixes $\ell$ uniquely (we speak of \textit{$\ell$-volcano} in this case) and we thus only have freedom on the choice of $p$. We prove nonetheless:

\begin{theorem} \label{thm:inverse_volcano_intro}
Let $\ell \in \mathbb{N}$ be a prime number and let $V$ be an abstract $\ell$-volcano of depth $d>0$. Then there exists infinitely many primes $p\in \mathbb{Z}$ such that $V$ is a connected component of the $\ell$-isogeny graph over $\mathbb{F}_p$.
\end{theorem}

The proofs of Theorems \ref{thm:abstract_crater_depth0_intro} and \ref{thm:inverse_volcano_intro} feature a study of elements in the class group of imaginary quadratic fields: in a nutshell, volcanoes with crater size $n$ exist because one can find ideal classes in well-chosen imaginary quadratic fields with order $n$. To study these questions on orders of elements in class groups, we are lead to study and explicitly solve some diophantine equations. Using variations of arguments of Nagell, Mahler, Pell, one is able to prove the following key step:

\begin{theorem} \label{thm:prime_given_order_intro}
The following properties hold.
\begin{enumerate}
\item Let $n \neq 4$ be a positive integer and let $K=\mathbb{Q}(\sqrt{1-2^{n+2}})$. Then in $\mathcal{O}_K$ the prime $2$ splits into two prime ideals whose corresponding classes in $\mathrm{Cl}(\mathcal{O}_K)$ have order $n$.
\item Let $K=\mathbb{Q}(\sqrt{-39})$. Then in $\mathcal{O}_K$ the prime $2$ splits into two prime ideals whose corresponding classes in $\mathrm{Cl}(\mathcal{O}_K)$ have order $4$.
\item Let $\ell \in \mathbb{Z}$ be an odd prime and let $n \in \mathbb{Z}_{>0}$. Define $K_1:=\mathbb{Q}(\sqrt{1-\ell^{n}})$ and $K_2:=\mathbb{Q}(\sqrt{1-4\ell^{n}})$. Then either in $\mathcal{O}_{K_1}$ or in $\mathcal{O}_{K_2}$ the prime $\ell$ splits into two prime ideals whose corresponding classes in $\mathrm{Cl}(\mathcal{O}_{K_i})$ have order $n$.
\end{enumerate}
\end{theorem}

To obtain a proof of Theorem \ref{thm:prime_given_order_intro}, we prove Propositions \ref{prop:prime_given_order_2},  \ref{prop:prime_given_order_l}, and \ref{prop:prime_given_order_l_part_2}. We are in fact more precise and are able to decide in the third item which of $K_1$ or $K_2$ is the correct field to consider, for a given pair $(\ell,n)$.

In the final section, we discuss two follow-up projects: the first one concerns the inverse volcano problem over more general finite fields. We prove in Proposition \ref{non_inverse_over_F_p^2} the existence of abstract 2-volcanoes such that for any prime $p\neq 2$, these are not connected components of ordinary isogeny graphs over $\mathbb{F}_{p^2}$. This triggers natural questions: how many exceptions can occur? Infinitely many or not? Do they come up with a specific shape? We plan to come back to these questions in future work.

The second project focuses on solving the inverse volcano problem with algorithmic efficiency: we provide an explicit solution to the inverse volcano problem over $\mathbb{F}_p$, but when working with concrete examples, it appears that one is often able to find number fields with smaller discriminants than the ones in our families. Is it possible to describe a solution with minimal discriminant? What is the smallest prime $p$ that realises an abstract volcano as an isogeny volcano over $\mathbb{F}_p$?

We also present a detailed study of the ordinary $3$-isogeny graph over $\mathbb{F}_{1009}$ in an appendix, including a complete picture of the corresponding volcano park \footnote{All the heavy computations in this paper have been performed using the SageMath software~\cite{sagemath}}. We encourage the reader to take a look at this volcano park before starting, it is very beautiful!

\section*{Acknowledgements}
The authors thank the IRN GandA (CNRS) and IRN MaDeF (CNRS) for support. HB thanks the University of Copenhagen for its hospitality. FP is supported by ANR-17-CE40-0012 Flair, FC and FP are both supported by ANR-20-CE40-0003 Jinvariant. FC is grateful to Max Planck Institute for Mathematics in Bonn for its hospitality and financial support. The three authors warmly thank Teresa Sorbera for helping in organising the workshop in R\o m\o, where the whole project started.

\section{Preliminaries}

We gather general results on isogenies of elliptic curves and modular polynomials in this first section.

\subsection{Isogenies of elliptic curves}

Let $k$ be either a finite field or a number field (which we suppose to be embedded in $\mathbb{C}$). By $\ell$-isogeny between two elliptic curves over $k$ we mean an isogeny of degree $\ell$ defined over $\overline{k}$. An elliptic curve $E/k$ has complex multiplication by a ring $R$ strictly containing $\mathbb{Z}$ if $\mathrm{End}_{\overline{k}}(E) \cong R$. An elliptic curve over a finite field is ordinary if $R$ is an imaginary quadratic order. It is supersingular if $R$ is an order in a quaternion algebra.

\begin{proposition}\label{isogeny:type}
Let $\varphi : E_1 \rightarrow E_2$ be an $\ell$-isogeny of ordinary elliptic curves over $k$ with geometric endomorphism rings $\Or_1$ and $\Or_2$. If $\mathrm{char}(k) \neq \ell$ then $[\Or_1:\Or_2]\in\{1/\ell,1,\ell\}$.
\end{proposition}
\begin{proof}
See~\cite[Proposition 21]{Koh96} or~\cite[\textsection 2.7]{Sut13}.
\end{proof}

\begin{definition}
Using the notations of Proposition~\ref{isogeny:type}, we say that $\varphi$ is \emph{horizontal} if $\Or_1=\Or_2$. Otherwise $\varphi$ is \emph{vertical}; \emph{descending} if $[\Or_1:\Or_2]=\ell$, \emph{ascending} if $[\Or_2:\Or_1]=\ell$. 
\end{definition}

\begin{remark}
The dual of a horizontal isogeny is horizontal, and the dual of a vertical isogeny is vertical, ascending if the initial isogeny was descending, and vice versa.
\end{remark}

\begin{definition}
Let $\mathcal{O}$ be an imaginary quadratic order and let $E/k$ be an elliptic curve with $\mathrm{End}_{\overline{k}}(E)\cong \mathcal{O}$. For every $\Or$-ideal $\mathfrak{a}$ the $\mathfrak{a}$\emph{-torsion subgroup} of $E$ is defined as 
\[
E[\mathfrak{a}] := \bigcap_{\alpha\in \mathfrak{a}} \Ker{\alpha} \subseteq E(\overline{k}).
\]
\end{definition}

The $\mathfrak{a}$-torsion subgroups correspond to isogenies $\varphi_{\mathfrak{a}}:E\rightarrow E/E[\mathfrak{a}]$ with kernel $E[\mathfrak{a}]$ and degree $\dg(\varphi_{\mathfrak{a}})=N(\mathfrak{a})$. Here, $E/E[\mathfrak{a}]$ denotes the quotient elliptic curve of $E$ by the subgroup $E[\mathfrak{a}]$. If $\mathfrak{a}$ is an invertible ideal of $\en_{\overline{k}}(E)$, then $\en_{\overline{k}}(E/E[\mathfrak{a}]) \cong \en_{\overline{k}}(E)$ (see \cite[Proposition 3.9 and Theorem 4.5]{Wat69}).

We thus obtain an action of the group of invertible $\mathcal{O}$-ideals on the set of $\overline{k}$-isomorphism classes of elliptic curves with complex multiplication by $\mathcal{O}$, given by
\[
\mathfrak{a} \ast E := E/E[\mathfrak{a}].
\]
This action factors via a faithful and transitive action of the class group $\mathrm{Cl}(\mathcal{O})$ of the order $\mathcal{O}$. If $k$ is a number field, we can write $E(\mathbb{C}) \cong \mathbb{C}/\Lambda$ where $\Lambda \subseteq K:=\mathrm{Frac}(\mathcal{O})$ is an invertible ideal, and then one has
\[
(E/E[\mathfrak{a}])(\mathbb{C}) \cong \mathbb{C}/\mathfrak{a}^{-1}\Lambda.
\]
We then see that the above action corresponds to the classical one over complex numbers described for instance in \cite[II, Proposition 1.2]{Silverman_advanced} for maximal orders. We summarise this discussion in the following theorem.

\begin{theorem} \label{thm:main_theorem_CM}
Let $K$ be an imaginary quadratic field, and let $\mathcal{O}$ be an order in $K$. Then, for every ideal class $[\mathfrak{a}] \in \mathrm{Cl}(\mathcal{O})$ and every elliptic curve $E/\overline{k}$ with $\mathrm{End}_{\overline{k}}(E)\cong \mathcal{O}$ the association
\[
\mathfrak{a} \ast E := E/E[\mathfrak{a}]
\]
defines a faithful and transitive action of the class group on the set of $\overline{k}$-isomorphism classes of elliptic curves with complex multiplication by $\mathcal{O}$.
\end{theorem}

\begin{proof}
If $k$ is a number field the statement is proved in~\cite[Chapters 8 and 10]{Lan87}. If $k$ is a finite field, the statement is proved in \cite[Theorem 4.5]{Wat69}.
\end{proof}

\begin{definition}\label{def:equivalent}
Let $E$ and $E'$ be elliptic curves over $k$. We say that two $\ell$-isogenies $\varphi, \psi: E \to E'$ are equivalent if $\ker \varphi = \ker \psi$. This is the same as requiring that there exists $\alpha \in \mathrm{Aut}_{\overline{k}}(E')$ such that $\alpha \circ \varphi=\psi$.
\end{definition} 
 
 \begin{lemma} \label{lem:conductor_not_divisible_by_p}
 Let $k$ be a finite field of characteristic $p$ and let $E/k$ be an elliptic curve with geometric endomorphism ring isomorphic to an order $\mathcal{O}$ in an imaginary quadratic field $K$. Denote by $f$ the conductor of $\mathcal{O}$. Then $p$ does not divide $f$.
 \end{lemma}
 
 \begin{proof}
By Deuring's lifting theorem \cite[Chapter 13, Theorem 14]{Lan87} there exists an elliptic curve $E'$ over $\overline{\mathbb{Q}}$ and a prime ideal $\mathfrak{p} \subseteq \overline{\mathbb{Q}}$ such that $E'$ has complex multiplication by an order $\mathcal{O}' \supseteq \mathcal{O}$, it has good reduction at $\mathfrak{p}$ and $E' \text{ mod } \mathfrak{p}$ is isomorphic to $E$. By \cite[II, Proposition 4.4]{Silverman_advanced} we must in fact have $\mathcal{O}'=\mathcal{O}$ and by \cite[Chapter 13, Theorem 12]{Lan87} the prime $p$ does not divide $f$.
 \end{proof}
 
 \begin{corollary} \label{cor:everything_is_here}
 Let $k$ be a finite field of characteristic $p$ and let $j(E) \in k$ be the $j$-invariant an elliptic curve $E/k$ with complex multiplication by an order $\mathcal{O}$ in the imaginary quadratic field $K$. If $h(\mathcal{O})$ denotes the class number of $\mathcal{O}$, then there are exactly $h(\mathcal{O})$ distinct elements of $k$ that are the $j$-invariants of elliptic curves with complex multiplication by $\mathcal{O}$.
 \end{corollary}

\begin{proof}
By Deuring's lifting theorem \cite[Chapter 13, Theorem 14]{Lan87} and \cite[II, Proposition 4.4]{Silverman_advanced} there exists a prime $\mathcal{P} \subseteq \overline{\mathbb{Q}}$ lying above $p$ and an element $j' \in \overline{\mathbb{Q}}$ such that $j'$ is the $j$-invariant of an elliptic curve with complex multiplication by $\mathcal{O}$ and $j' \equiv j \text{ mod } \mathcal{P}$. Denoting by $D$ the discriminant of the order $\mathcal{O}$, we then have that the Hilbert class polynomial $H_D(X)$ (see \cite[pag. 285]{Cox97}) has a root modulo $p$, namely the reduction of $j'$ modulo $\mathcal{P}$. 
Let now $H_\mathcal{O}$ be the compositum over $\mathbb{Q}$ of $\mathbb{Q}(j')$ and $K$. It is well-known (see for instance \cite[Lemma 9.3]{Cox97}) that $H_\mathcal{O}$ is a Galois extension of $\mathbb{Q}$ which can only be ramified at the rational primes dividing $D$. In particular, using Lemma \ref{lem:conductor_not_divisible_by_p} and the fact that the prime $p$ is split in $K$ by \cite[Chapter 13, Theorem 12]{Lan87}, we see that the prime $p$ does not ramify in $\mathbb{Q} \subseteq H_\mathcal{O}$. Moreover, by assumption the prime $\mathfrak{p}:= \mathcal{P} \cap \mathbb{Q}(j')$ has residue field $k_\mathfrak{p}$ contained in $k$. Let $\mathfrak{P}:=\mathcal{P} \cap H_\mathcal{O}$ and denote by $D(\mathfrak{P}/\mathfrak{p}) \subseteq \mathrm{Gal}(H_\mathcal{O}/\mathbb{Q}(j'))$ the decomposition group of $\mathfrak{P}$ over $\mathfrak{p}$. For every $\sigma \in D(\mathfrak{P}/\mathfrak{p})$, the restriction $\sigma|_K$ belongs to the decomposition group of $\mathfrak{P} \cap K$ over $p$, which is trivial. Since $H_\mathcal{O}$ is the compositum of $K$ and $\mathbb{Q}(j')$, we deduce that $D(\mathfrak{P}/\mathfrak{p})$ is also trivial, so that in particular the residue degree $f(\mathfrak{P}/\mathfrak{p})=1$. Because the extension $\mathbb{Q} \subseteq H_\mathcal{O}$ is Galois, we see that all the primes in $H_\mathcal{O}$ lying above $p$ have the same residue field isomorphic to $k_\mathfrak{p}$. 

Hence, all the roots $j \in H_\mathcal{O}$ of $H_D(X)$, which are all the possible singular invariants of elliptic curves over $\overline{\mathbb{Q}}$ with complex multiplication by $\mathcal{O}$, satisfy $j \text{ mod } \mathcal{P} \in k_\mathfrak{p} \subseteq k$. By \cite[Chapter 13, Theorems 12 and 13]{Lan87} and Lemma \ref{lem:conductor_not_divisible_by_p}, this shows that there are at least $\mathrm{deg}(H_D(X))=h(\mathcal{O})$ distinct elements of $k$ that are the $j$-invariants of elliptic curves with complex multiplication by $\mathcal{O}$. By Deuring's lifting theorem there cannot be more, and this concludes the proof.
\end{proof}

\subsection{Modular polynomials}

Let $\ell \in \mathbb{N}$ be a prime. The \textit{modular polynomial of level $\ell$} is the polynomial $\Phi_\ell(X,Y) \in \mathbb{Z}[X,Y]$ characterised by the following property: for every $j \in \overline{\mathbb{Q}}$ which is the $j$-invariant of an elliptic curve $E$, the polynomial $\Phi_\ell(j,Y)$ factors as
\begin{equation} \label{eq:factorization_modular_polynomial}
    \Phi_\ell(j,Y)=\prod_{C} \left(Y-j(E/C) \right)
\end{equation}
where $C$ varies among the $\ell+1$ cyclic subgroups of order $\ell$ of $E(\overline{\mathbb{Q}})$, and $E/C$ denotes the quotient of the curve $E$ by the subgroup $C$. One can prove that $\Phi_\ell(j,Y)$ is symmetric in its variables and has degree $\ell + 1$. For more information concerning modular polynomials see \cite[Chapter 11, paragraph C]{Cox97}.

Since the polynomial $\Phi_\ell(X,Y)$ has integer coefficients, we can reduce it modulo a prime $p$, obtaining a polynomial $\widetilde{\Phi}_\ell(X,Y) \in \mathbb{F}_p[X,Y]$. It is not difficult to show that, if $p\neq \ell$, then $\widetilde{\Phi}_\ell$ satisfies the analogous property \eqref{eq:factorization_modular_polynomial} with $\overline{\mathbb{Q}}$ replaced by $\overline{\mathbb{F}}_p$. Indeed, if $\widetilde{j} \in \overline{\mathbb{F}}_p$ is the $j$-invariant of an elliptic curve $\widetilde{E}/\overline{\mathbb{F}}_p$, choose a prime $\mathfrak{p} \subseteq \overline{\mathbb{Q}}$ lying above $p$ and $j \in \overline{\mathbb{Q}}$ such that $j$ is the singular invariant of an elliptic curve $E$ reducing to $\widetilde{E}$ modulo $\mathfrak{p}$. Then, by reducing the factorization \eqref{eq:factorization_modular_polynomial} modulo $\mathfrak{p}$, we see that the roots of $\widetilde{\Phi}_\ell (\widetilde{j},Y)$ are precisely the reductions $\widetilde{j(E/C)}$ for all subgroups $C$ of order $\ell$ in $E(\overline{\mathbb{Q}})$. By Vélu's formulas \cite{Velu_1971} we have $\widetilde{j(E/C)}=j(\widetilde{E}/\widetilde{C})$, where $\widetilde{C}=C \text{ mod } \mathfrak{p}$, and, since $p\neq \ell$, \cite[VII, Proposition 3.1]{Sil09} implies that the subgroups $\widetilde{C}$ range among all possible cyclic subgroups of order $\ell$ of $\widetilde{E}(\overline{\mathbb{F}}_p)$. 

From now on, we fix $k$ to be either $\mathbb{Q}$ or $\mathbb{F}_p$ for some prime $p \neq \ell$, and we consider $\Phi_\ell$ as a polynomial defined over $k$. The affine curve $\mathcal{C} \subseteq \mathbb{A}_k^2$ described by $\Phi_\ell$ is a singular model for the modular curve $Y_0(\ell)/k$. This in particular means that every smooth point $Q\in \mathcal{C}(\overline{k})$ corresponds to an isomorphism class of pairs $(E,C)$, where $E$ is an elliptic curve over $\overline{k}$ and $C$ is a cyclic subgroup of order $\ell$ of $E$. If $\Phi_\ell(j,j')=0$, then the smooth point $(j,j') \in \mathcal{C}(\overline{k})$ represents the equivalence class of a pair $(E,C)$ with $j(E)=j$ and $j(E/C)=j'$. More generally, the roots of $\Phi_\ell(j,Y)$, counted with multiplicities, are in bijection with the isomorphism classes of pairs $(E,C)$ where $j(E)=j$ and the subgroup $C\subseteq E(\overline{k})$ is cyclic of order $\ell$. The presence of multiple roots of $\Phi_\ell(j,Y)$ comes from the existence of distinct cyclic subgroups $C, C'$ in $E$ such that $j(E/C)=j(E/C')$. This certainly occurs if there exists an automorphism $\alpha\in{\mathrm{Aut}_{\overline{k}}(E)}$ such that $\alpha(C)=C'$, which can be the case when $j=0$ or $j=1728$.

\begin{example}
This example is taken from \cite[page 238]{Sch95}. Consider the elliptic curve $E:y^2=x^3-1$ defined over $\mathbb{F}_7$. This is an ordinary elliptic curve with $j(E)=0$, whose CM order is generated over the integers by the automorphism $[\zeta_3]: (x,y) \mapsto (2x,y)$. The four cyclic subgroups of order $3$ in $E(\overline{\mathbb{F}}_7)$ are given by
\[
C_0=\langle (0,i) \rangle, \hspace{0.2cm} C_1=\langle (\sqrt[3]{4}, 2i) \rangle, \hspace{0.2cm} C_2=\langle (2 \sqrt[3]{4}, 2i) \rangle, \hspace{0.2cm} C_3=\langle (4 \sqrt[3]{4}, 2i) \rangle
\]
where $i, \sqrt[3]{4} \in \overline{\mathbb{F}}_7$ are respectively a fixed square root of $-1$ and cube root of $4$. Note that the automorphism $[\zeta_3]$ acts transitively on the subgroups $C_i$ with $i=1,2,3$. In accordance with what was mentioned above, the polynomial $\Phi_3(0,Y)=Y(Y-3)^3$ has a triple root, corresponding to the fact that $j(E/C_i)=3$ for all $i=1,2,3$.
\end{example}

\section{Ordinary isogeny graphs over \texorpdfstring{$\mathbb{F}_p$}{Fp}}
\label{sec:volc}
Let $p\geq5$ be a prime number and denote by $\mathbb{F}_p$ the prime field of characteristic $p$, with algebraic closure $\overline{\mathbb{F}}_p$. The condition on $p$ is not restrictive for our study of the inverse volcano problem and, besides, it allows us to simplify the exposition of the theory in this section (cfr. for instance the use of Waterhouse's \cite[Theorem 4.1]{Wat69} in the proof of Lemma \ref{lem:cordillera_empty}).
We want to define and study ordinary isogeny graphs over $\mathbb{F}_p$. All the constructions appearing in this section can be performed, \textit{mutatis mutandis}, over $\mathbb{F}_{p^s}$ for any $s>1$. We decided to focus only on prime fields, over which we will formulate and solve the inverse volcano problem.
 
 Let $\mathcal{V}$ be the subset of $j \in \mathbb{F}_p$ such that the elliptic curves $E/\overline{\mathbb{F}}_p$ with $j(E)=j$ are ordinary.
For any $j \in \mathcal{V}$, let $\mathcal{E}_j$ be the set of $\mathbb{F}_{p^2}$-isomorphism classes of elliptic curves $E/\mathbb{F}_p$ with $j(E)=j$. In other words, two elliptic curves $E,E'$ over $\mathbb{F}_p$ represent the same class in $\mathcal{E}_j$ if and only if $j(E)=j(E')=j$ and $E$ is a quadratic twist of $E'$. If $j\neq 0,1728$, the set $\mathcal{E}_j$ has cardinality $1$ (see \cite[Proposition 5.4]{Sil09}), and we fix a representative $E_j/\mathbb{F}_p$. If $j=0$ then $\# \mathcal{E}_0=3$ and we fix $\{E_0^{(1)},E_0^{(2)}, E_0^{(3)} \}$ to be three representatives of the different isomorphism classes. Similarly, if $j=1728$ then $\# \mathcal{E}_{1728}=2$ and we fix $\{E_{1728}^{(1)},E_{1728}^{(2)}\}$ to be two representatives of the different isomorphism classes. In these two special cases, we generically use $E_0$ (resp. $E_{1728}$) to denote any of the three elliptic curves $\{E_0^{(1)},E_0^{(2)}, E_0^{(3)} \}$ (resp. $\{E_{1728}^{(1)},E_{1728}^{(2)}\}$).

\begin{definition} For a prime $\ell\neq p$, the ordinary $\ell$\emph{-isogeny graph} over $\mathbb{F}_p$ is the directed graph $\G_{\ell}(\mathbb{F}_p)$ whose vertex set equals $\mathcal{V}$ and such that, for every $j,j' \in \mathcal{V}$, there are $m$ directed edges from $j$ to $j'$ if and only if $j'$ is a root of $\Phi_\ell(j,Y)$ with multiplicity $m$.
\end{definition}

\begin{remark} \label{rmk:how_to_draw}
The directed edges from one node $j_1$ to another node $j_2$ in $\G_{\ell}(\mathbb{F}_p)$ represent the non-equivalent classes of cyclic isogenies of degree $\ell$ between $E_{j_1}$ and $E_{j_2}$. If there is a directed edge between $j_1$ and $j_2$ with $j_1\neq j_2$, and if $j_1, j_2\notin\{0,1728\}$, then there is a unique edge from $j_2$ to $j_1$ corresponding to the dual isogeny. By convention, in this situation we only represent one undirected edge in our figures. The case $j_1=j_2$ can be more subtle, as an elliptic curve may have a self $\ell$-isogeny $\varphi$ with dual $\widehat{\varphi}$ satisfying $\ker{\varphi}\neq \ker{\widehat{\varphi}}$. In this case, the graphic representations display both $\varphi$ and $\widehat{\varphi}$. 
\end{remark}

We begin our systematic study of the structure of $\mathcal{G}_\ell(\mathbb{F}_p)$. Moving from general to specific, we divide the isogeny graph into progressively smaller subgraphs: cordilleras, belts and volcanoes. We follow the geological terminology from~\cite{FouM02, Mir07}.

Before leaving on a mountain hike, we fix some notations: for each $j \in \mathcal{V}$, the elliptic curve $E_j$ has complex multiplication by an order $\mathcal{O}_j:=\en_{\overline{\mathbb{F}}_p}(E_j)$ in an imaginary quadratic field $K$. We write $D(\mathcal{O})$ for the discriminant of a general order $\mathcal{O}$. For imaginary quadratic orders, it is a negative integer congruent to $0$ or $1$ mod $4$.

\subsection{Cordilleras} \label{subsec:cordilleras}

For every $j \in \mathcal{V}$ let us denote by $\mathrm{Tr}:\en_{\overline{\mathbb{F}}_p}(E_j) \otimes_\mathbb{Z} \mathbb{Q} \to \mathbb{Q}$ the trace map and by
$\pi_j \in \en_{\overline{\mathbb{F}}_p}(E_j)$ the Frobenius endomorphism of the elliptic curve $E_j$. One knows that $\mathrm{Tr}(\pi_j)=p+1-\#E_j(\F_p)$ and that $\mathbb{Z}[\pi_j]$ is isomorphic to the imaginary quadratic order of discriminant $D(\mathbb{Z}[\pi_j])=\mathrm{Tr}(\pi_j)^2-4p$.

\begin{definition}
Let $t\in \Z$. The $t$\emph{-cordillera} in $\G_{\ell}(\F_p)$ is the subgraph of $\G_{\ell}(\F_p)$ induced by the subset of vertices
\[
\mathcal{V}_t:=\{j\in \mathbb{F}_p: \mathrm{Tr}(\pi_j)=\pm t\} \subseteq \mathcal{V}.
\]
\end{definition}

By definition, if $j,j' \in \mathcal{V}_t$ we have $D(\mathbb{Z}[\pi_j])=D(\mathbb{Z}[\pi_{j'}])$, so that both $\en_{\overline{\mathbb{F}}_p}(E_j)$ and $\en_{\overline{\mathbb{F}}_p}(E_{j'})$ contain an order isomorphic to the imaginary quadratic order of discriminant $t^2-4p$. In particular, $E_j$ and $E_{j'}$ have complex multiplication by an order inside the \textit{same} imaginary quadratic field $K$. We call $K$ the \textit{field associated with the $t$-cordillera} and, if $D(\mathcal{O}_K)<-4$, we say that the cordillera is \textit{regular}. 
Cordilleras span the whole isogeny graph. Note that their vertices are defined independently of the isogeny degree $\ell$. See Appendix~\ref{app:tables} for an example of how changing $\ell$ changes the edges in the same $t$-cordillera.

\begin{lemma} \label{lem:cordillera_empty}
For $t\in \mathbb{Z}_{\neq 0}$ the following holds:
\begin{enumerate}
    \item The set $\mathcal{V}_t$ is non-empty if and only if $-2\sqrt{p} \leq t \leq 2\sqrt{p}$;
    \item If $\mathcal{V}_t$ is non-empty then for every order $\mathcal{O}$ containing the order of discriminant $t^2-4p$ there exists an elliptic curve $E/\mathbb{F}_p$ such that $\en_{\overline{\mathbb{F}}_p}(E) \cong \mathcal{O}$ and $j(E)$ is in $\mathcal{V}_t$;
    \item If $j(E)\in{\mathcal{V}_t}$ and if $\en_{\overline{\mathbb{F}}_p}(E) \cong \mathcal{O}$, then all the $h(\mathcal{O})$ $j$-invariants corresponding to curves with CM by $\mathcal{O}$ are also in $\mathcal{V}_t$. 
\end{enumerate}
\end{lemma}

\begin{proof}
The first part is given in \cite[Theorem 4.1]{Wat69}. We now prove the second part: let $j(E)$ be an element in $\mathcal{V}_t$. Let $\mathcal{O}$ be the endomorphism ring of $E$. It contains the order $\mathbb{Z}[\pi_E]$ generated by the Frobenius $\pi_E$, which has discriminant $t^2-4p$. Each order containing an order isomorphic to $\mathbb{Z}[\pi_E]$ can be realised as the endomorphism ring of an elliptic curve defined over $\mathbb{F}_p$ and $\mathbb{F}_p$-isogenous to $E$, see \cite[Theorem 4.2, point (2)]{Wat69}. In particular, as these curves are all $\mathbb{F}_p$-isogenous to $E$, they have the same trace $t$ and this proves the second part. For the third part, by Corollary \ref{cor:everything_is_here}, all the $j$-invariants corresponding to curves with CM by $\mathcal{O}$ are rational over $\mathbb{F}_p$. If $j(E)=0$ or $j(E)=1728$, the class number of $\mathcal{O}$ is one and the result is immediate. If $j(E)\notin\{0,1728\}$, consider a $j$-invariant $j(E') \in \mathbb{F}_p$ of an elliptic curve $E'$ with complex multiplication by $\mathcal{O}$. Since $E$ and $E'$ have complex multiplication by the same order, they are geometrically isogenous. Let $k/\mathbb{F}_p$ be a degree $r$ field extension over which this isogeny is defined. The Frobenius endomorphisms $\pi_E$ and $\pi_{E'}$ of the two elliptic curves, seen as elements of the same imaginary quadratic field $K=\mathrm{Frac}(\mathcal{O})$, satisfy $\pi_E^r=\pi_{E'}^r$ or $\pi_E^r=\overline{\pi}_{E'}^r$ by Tate's isogeny theorem \cite[Theorem 1 (c4)]{Tate_1966}. This means that there exists a root of unity $\zeta \in K$ such that
\[
\pi_E=\zeta \pi_{E'} \hspace{0.2cm} \text{ or } \hspace{0.2cm} \pi_E= \zeta \overline{\pi}_{E'}.
\]
If $\zeta=\pm 1$, it follows immediately that $j(E') \in \mathcal{V}_t$. Otherwise, $\zeta \in \{\pm i, \pm \zeta_3, \pm \zeta_3^2 \}$, where $i$ is a primitive fourth root of unity and $\zeta_3$ is a primitive third root of unity. In this case, one easily deduces that $j(E)=j(E') \in \{0, 1728\}$, a contradiction. There are $h(\mathcal{O})$ possibilities for $j(E')$, and that gives the claim. 
\end{proof}

Fix a nonzero integer $t \in (-2\sqrt{p}, 2\sqrt{p})$ and let $K$ be the field associated with the $t$-cordillera. All the $j$-invariants of elliptic curves with complex multiplication by the maximal order in $K$ belong to $\mathcal{V}_t$ by Lemma \ref{lem:cordillera_empty} (3). Fix such a singular invariant $j$ and let $E_j$ be a corresponding elliptic curve over $\mathbb{F}_p$. If $D(\mathcal{O}_K)<-4$, then $E_j$ is determined only up to quadratic twisting, and the Frobenius endomorphism of any of its quadratic twists has, up to sign, the same trace. This in particular implies, using Lemma \ref{lem:cordillera_empty} (2) that any imaginary quadratic number field $K$ with discriminant $D(\mathcal{O}_K)<-4$ can be associated \textit{to at most one cordillera} (with positive trace). One can formulate this statement in terms of norm equations as follows. 

\begin{lemma}\label{normeq:regular}
Let $K$ be an imaginary quadratic field of discriminant $D(\mathcal{O}_K)<-4$. Then the equation
\begin{equation}\label{cox:bigD}
4p = t^2-v^2D(\mathcal{O}_K)    
\end{equation}
has at most one solution $(t,v)\in \N^2$ with $p \nmid t$.
\end{lemma}

If $\mathrm{disc}(K) \in \{-3,-4\}$ the situation is more delicate because the vertices $j=0$ and $j=1728$ are represented by 3 and 2 curves respectively, and the squared traces of the Frobenius endomorphisms of these curves may be different. These vertices may thus belong to different cordilleras at the same time. This is indeed always the case, as the following lemma shows.

\begin{lemma} \label{lem:normeq_bad_cases}
The following holds:
\begin{enumerate}
    \item[(i)] The equation 
\begin{equation}
    4p = t^2+3v^2
\end{equation}
has no integer solutions if $p\equiv 2 \text{ mod } 3$, and exactly three solutions $(t,v)$ with $t>0$ and $v>0$ if $p\equiv 1 \text{ mod } 3$. 

\item[(ii)] The equation 
\begin{equation}\label{eq:4}
    4p = t^2+4v^2
\end{equation}
has no solutions if $p\equiv3 \text{ mod } 4$, and exactly two solutions $(x,y)$ with $x>0$ and $y>0$ if $p\equiv1 \text{ mod } 4$. 
\end{enumerate}
\end{lemma}

\begin{proof}
Both statements reduce to solving a norm equation in the ring $\mathbb{Z}[\frac{1+\sqrt{-3}}{2}]$ and $\mathbb{Z}[\sqrt{-1}]$ respectively.
\end{proof}

In particular, Lemma \ref{lem:normeq_bad_cases} shows that if $j=0$ is an ordinary invariant then $K=\mathbb{Q}(\sqrt{-3})$ is associated with three $t$-cordilleras, and if $j=1728$ is an ordinary invariant then $K=\mathbb{Q}(\sqrt{-1})$ is associated with two $t$-cordilleras ($t$ positive).

\subsection{Belts} \label{subsec:belts}


In this section, we work inside a fixed non-empty $t$-cordillera $\mathcal{C}_t$ with vertex set $\mathcal{V}_t$. Let $K$ be the imaginary quadratic number field associated with this cordillera and denote by $\Or_K$ its ring of integers. Choose $\pi \in K$ to be any element such that 
\[
\pi^2-t\pi+p=0
\]
so that, if $v:=[\Or_K:\Z[\pi]]$ denotes the conductor of the order $\mathbb{Z}[\pi]$ in $K$, then
\[
4p-t^2=-v^2D(\mathcal{O}_K)
\]
where $D(\mathcal{O}_K)$ is the discriminant of $K$. Denoting by $v_\ell(\cdot)$ the usual $\ell$-adic valuation on $\mathbb{Q}$, we let $d:=v_{\ell}(v)$ and $v':= v\cdot \ell^{-d}$. 
\begin{definition}
An order $\mathcal{O}$ such that $\Z[\pi]\subset\Or\subset\Or_K$ is said to be $\ell$\emph{-saturated} if  $v_{\ell}([\Or_K:\Or])=d$. It is said to be $\ell$\emph{-dry} if $v_{\ell}([\Or_K:\Or])=0$. 
\end{definition}
For every order $\Z[\pi]\subset\Or\subset\Or_K$ we define its \textit{saturation} $\widetilde{\mathcal{O}}$ as the unique order in $K$ such that
\[
[\mathcal{O}_K : \widetilde{\mathcal{O}}]=[\mathcal{O}_K : \mathcal{O}] \cdot \ell^{d-v_\ell([\mathcal{O}_K : \mathcal{O}])}.
\]
Note that $\widetilde{\mathcal{O}}$ still contains $\mathbb{Z}[\pi]$. We are now ready to state the main definition of this section.
\begin{definition}
For a positive integer $m\mid v'$, the $t$-\emph{belt} of index  $m$, denoted by $\mathcal{B}_{t,m}$, is the subgraph of $\mathcal{G}_{\ell}(\F_p)$ induced by the subset
\[\mathcal{V}_{t,m} = \{j \in \mathcal{V}_t: \widetilde{\en_{\overline{\mathbb{F}}_p}(E_j)} \cong \mathbb{Z}+ m\ell^{d} \mathcal{O}_K \} \subseteq \mathcal{V}_t.\]
\end{definition}

\begin{lemma} The following holds:
\begin{enumerate}
    \item The cut-set determined by the vertex partition $(\mathcal{V}_{t,m}, \mathcal{V} \setminus \mathcal{V}_{t,m})$ is empty.
    \item Different $t$-belts form disjoint subgraphs of $\mathcal{G}_\ell(\mathbb{F}_p)$.
\end{enumerate}
\end{lemma}
\begin{proof}
Both points follow immediately from Proposition \ref{isogeny:type} and the fact that the endomorphism ring of every elliptic $E$ with $j(E) \in \mathcal{V}_t$ must contain a subring isomorphic to $\mathbb{Z}[\pi]$.
\end{proof}
Note that two belts $\mathcal{B}_{t,m_1}$ and $\mathcal{B}_{t,m_2}$ for $m_1\neq m_2$ cannot contain $j$-invariants of elliptic curves that have isomorphic endomorphism ring. Moreover, a belt does not have to be connected.


\subsection{Isogeny volcanoes} \label{subsec:volcanoes}

Isogeny volcanoes are the connected components of $\mathcal{G}_\ell(\mathbb{F}_p)$. The vertices of these graphs come endowed with a natural ``stratification" by level, in the sense of the following definition.

\begin{definition}
An ordinary elliptic curve $E/\F_p$, its $j$-invariant $j(E)$, and its associated order $\en(E)=\Or=\Z+f\Or_K$ are said to lie at \emph{level} $n$ in the volcano containing $j(E)$ if $v_{\ell}(f)=n$. 
\end{definition}

The subgraph induced by the subset of level-$n$ vertices in a volcano $V$ is denoted by $V_n$. The graph $V_0$ is called the \textit{crater} of $V$ while we say that the vertices in $V\setminus V_0$ are on the \textit{lava flows} of the volcano $V$. The \textit{depth} of $V$ is the maximal level on which vertices $j \in V$ lie.

\begin{lemma}\label{lem:invertible}
Let $k$ be a finite field of characteristic $p$ and $\ell \neq p$ a prime number. Let $E$ and $E'$ be two elliptic curves over $k$ with complex multiplication by an order $\mathcal{O}$ of conductor $f$ in an imaginary quadratic field $K$.
\begin{enumerate}
    \item If $\mathfrak{L} \subseteq \mathcal{O}$ is an invertible ideal of norm $\ell$, then $E/E[\mathfrak{L}]$ has complex multiplication by $\mathcal{O}$ and the quotient map $E \to E/E[\mathfrak{L}]$ has degree $\ell$;
    \item If $\varphi: E \to E'$ is an isogeny of degree $\ell$, then there exists an invertible ideal $\mathfrak{L} \subseteq \mathcal{O}$ of norm $\ell$ such that $E' \cong E/E[\mathfrak{L}]$;
    \item If $\mathfrak{L} \subseteq \mathcal{O}$ is a non-invertible ideal of norm $\ell$, then $\ell \mid f$ and $E/E[\mathfrak{L}]$ has complex multiplication by the unique quadratic order containing $\mathcal{O}$ with index $\ell$.
\end{enumerate}
\end{lemma}

\begin{proof}
For an ideal $I \subseteq \mathcal{O}$ we denote by $\mathcal{O}(I):=\{x \in K: xI \subseteq I \}$ the order associated with $I$. We always have $\mathcal{O} \subseteq \mathcal{O}(I)$ and $\mathcal{O}(I)$ is by definition the smallest quadratic order $\mathcal{O}' \supseteq \mathcal{O}$ such that $I\mathcal{O}'$ is invertible in $\mathcal{O}'$.

Part (1) follows from \cite[Proposition 3.9]{Wat69} and the fact that $\# E[\mathfrak{L}]= \# (\mathcal{O}/\mathfrak{L})$, since $\ell \neq p$.

We now prove part (2). Consider the ideal
\[
\mathfrak{L}:=\{ f \in \mathrm{End}_{\overline{k}}(E): f(\ker \varphi)=0 \} \subseteq \mathrm{End}_{\overline{k}}(E) = \mathcal{O}.
\]

We clearly have $\ker \varphi \subseteq E[\mathfrak{L}]$. On the other hand, it is not difficult to see, using a finite fields analogue of \cite[II, Corollary 1.1.1]{Silverman_advanced} and the fact that $E$ and $E'$ have complex multiplication by the same order, that $\ker \varphi$ is a cyclic $\mathcal{O}$-module which is then isomorphic to $\mathcal{O}/\mathrm{Ann}_\mathcal{O}(\ker \varphi) = \mathcal{O}/\mathfrak{L}$, where $\mathrm{Ann}_\mathcal{O}(\ker \varphi)$ denotes the annihilator ideal of $\ker \varphi$. Since $\ell \neq p$ we have 
\[
\# E[\mathfrak{L}] = \# \mathcal{O}/\mathfrak{L} = \# \ker \varphi = \ell
\]
and we then obtain $\ker \varphi = E[\mathfrak{L}]$. Hence $E' \cong E/\ker \varphi = E/E[\mathfrak{L}]$. Using \cite[Proposition 3.9]{Wat69} (which can be applied thanks to \cite[Theorem 4.5]{Wat69}) we see that $\mathcal{O}(\mathfrak{L})=\mathcal{O}$ and so $\mathfrak{L}$ is invertible in $\mathcal{O}$.

We finally prove part (3). Note that the existence of a non-invertible ideal of norm $\ell$ implies that the conductor of $\mathcal{O}$ is divisible by $\ell$. Using~\cite[Proposition 3.9]{Wat69} again, we need to prove that $\mathcal{O}(\mathfrak{L})$ is equal to the unique order $\mathcal{O}'$ containing $\mathcal{O}$ with index $\ell$. Set $\mathcal{O}'=\mathbb{Z}[\omega]$. Then we have $\mathcal{O}=\mathbb{Z}[\ell \omega]$ and $\mathfrak{L}=(\ell, \ell \omega)$. Hence $\mathfrak{L} \mathcal{O}'$ is principal, generated by $\ell$, and in particular invertible. We deduce that $\mathcal{O}(\mathfrak{L})=\mathcal{O}'$ and the proof is concluded.
\end{proof}

\begin{remark}\label{ideal:norm}
If $\Or$ is an imaginary quadratic order of discriminant $D_{\Or}$ and $\ell$ is a prime, then there are exactly $1+\left(\frac{D_{\Or}}{\ell}\right)$ prime ideals in $\mathcal{O}$ with norm $\ell$. If $\ell | [\Or_K:\Or]$, the unique prime ideal of norm $\ell$ is not invertible. In all other cases, the primes of norm $\ell$ are always invertible.
\end{remark}

From a volcanic perspective, Lemma \ref{lem:invertible} and Remark~\ref{ideal:norm} translate in the following way.

\begin{corollary} \label{cor:horizon}
Horizontal isogenies in $\G_{\ell}(\F_p)$ can only occur between vertices at level $0$ in the volcanoes belonging to the same belt. More precisely, for every non-zero $t\in\mathbb{Z}$, for every $t$-cordillera, if $K$ is the field associated with the $t$-cordillera, then for every $m$, for every belt $\mathcal{B}_{t,m}$, for every volcano $V$ in $\mathcal{B}_{t,m}$, there are exactly
\[
1+\left(\frac{D(\mathcal{O}_K)}{\ell}\right) = \left\{\begin{array}{ll}
        0 & \text{if } \ell \text{ is inert in } K,\\
        1 & \text{if }  \ell \text{ is ramified in } K,\\
        2 & \text{if }  \ell \text{ splits in } K,\end{array}\right.
\]
distinct edges of $\G_{\ell}(\F_p)$ from $j\in{V_0}$ to other vertices in $V_0$.
\end{corollary}

We will now describe the structure of an isogeny volcano by analysing 
crater and lava flow separately. We start by analysing crater structure, then focus on a single belt (that is, looking at a fixed endomorphism structure), and finally we let the lava flow down, looking at class numbers at every level of the volcanoes. In other words, there is a simultaneous eruption in all volcanoes belonging to the same belt, and class number arithmetic constrains the shape to which lava solidifies!

\subsubsection{The crater}

The following proposition presents the possible craters that can occur.

\begin{proposition}\label{horizontal}
Let $\mathcal{B}_{t,m}$ be a belt in the isogeny graph $\mathcal{G}_\ell(\mathbb{F}_p)$ and let $\mathcal{C}_{t,m}$ be the subgraph of $\mathcal{B}_{t,m}$ induced by the vertices at level $0$. Let $\mathcal{O}$ be the CM order associated with any vertex in $\mathcal{C}_{t,m}$ and let $\mathfrak{L} \subseteq \mathcal{O}$ be a prime ideal dividing $\ell$. Then each connected component of $\mathcal{C}_{t,m}$ is a crater of a volcano in $\mathcal{B}_{t,m}$ and consists of:
\begin{enumerate}
    \item A single vertex if $\ell$ is inert in $K$;
    \item A single vertex with a self-loop if $\ell$ is ramified in $K$ and $\mathfrak{L}$ is principal;
    \item A single vertex with two self-loops if $\ell$ splits in $K$ and $\mathfrak{L}$ is principal;
    \item Two vertices connected with a single edge if $\ell$ ramifies in $K$ and $\mathfrak{L}$ is not principal;
    \item Two vertices connected with a double edge if $\ell$ splits in $K$ and $\mathfrak{L}$ has order 2 in the class group $\mathrm{Cl}(\mathcal{O})$;
    \item More generally, a cycle of size the order of $[\mathfrak{L}]$ in the class group $\Cl(\Or)$ if $\ell$ splits in $K$ and we are not in any of the previous cases.
\end{enumerate}
\end{proposition}
\begin{proof}
Fix $j(E) \in \mathcal{C}_{t,m}$. We want to analyse the connected component of $\mathcal{C}_{t,m}$ containing $j(E)$. Using Lemma \ref{lem:invertible}, there is an edge from $j(E)$ to $j(E/E[\mathfrak{L}])$. As the action described in Theorem \ref{thm:main_theorem_CM} is faithful and transitive, by iteration we obtain a cycle of length the order of the class of $\mathfrak{L}$ in $\mathrm{Cl}(\mathcal{O})$. 

A case by case analysis concludes the proof, for instance:
In the case where $\ell$ is ramified in $K$, and the unique ideal $\mathfrak{L}$ lying above $\ell$ is principal, then $\mathfrak{L}=\overline{\mathfrak{L}}$, and $\mathfrak{L}$,  $\overline{\mathfrak{L}}$ correspond to mutually dual isogenies with the same kernel. This situation is represented by a unique self-loop. In the case where $\ell$ splits into principal ideals $\mathfrak{L}, \overline{\mathfrak{L}}$ in $K$, and the ideal $\mathfrak{L}$ lying above $\ell$ is principal, then $\mathfrak{L}\neq \overline{\mathfrak{L}}$. These ideals correspond to two self-isogenies with distinct kernels. This situation is represented by two self-loops in the graph (cfr. Remark \ref{rmk:how_to_draw}).
\end{proof}

\begin{remark}
Case $(6)$ in the above proposition in fact includes also case $(5)$. However, we decided to separate the two cases in order to underline the difference between case $(4)$ and case $(5)$.
\end{remark}

See Figure~\ref{fig:craters} for a view of all possible subgraphs at level $0$.

\begin{figure}[h]
    \centering
    \scalebox{2}{\input{volcanoes/l=3/craters.tikz}}
    \caption{All possible craters.}
    \label{fig:craters}
\end{figure}
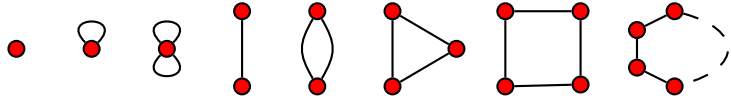

\subsubsection{The lava flows}

We now explain what the other levels of a volcano look like.

\begin{lemma}\label{l+1}
Let $E/k$ be an elliptic curve. The number of non-equivalent $\ell$-isogenies from $E$ to any other curve with $k$-rational $j$-invariant is $0,1,2$ or $\ell+1$.
\end{lemma}
\begin{proof}
The kernel of an isogeny from $E$ must be a subgroup of $E[\ell]\cong \Z/\ell\Z\times\Z/\ell\Z$. For it to be an $\ell$-isogeny the subgroup must have order $\ell$. In order for the target curve to be $k$-rational, the kernel of the isogeny must be invariant under the action of the Galois group $G=\mathrm{Gal}(\overline{k}/k)$. There are $\ell+1$ subgroups of order $\ell$ in $E[\ell]$, that can be seen as lines in an $\F_{\ell}$-vector space. $G$ acts linearly, and if it fixes three or more such lines, then it acts as an homothety and hence fixes every subgroup. 
\end{proof}

\begin{proposition}\label{descending}
Let $V$ be a volcano of depth $d$ in the isogeny graph $\mathcal{G}_\ell(\mathbb{F}_p)$. Assume that the crater $V_0$ contains $j$-invariants corresponding to elliptic curves with complex multiplication by an order $\mathcal{O}$ of discriminant $D(\mathcal{O})$  with $D(\mathcal{O}) \neq -3,-4$. Then:
\begin{itemize}
    \item For every $j \in V_0$, the elliptic curve $E_j$ has, up to equivalence, $0$ descending isogenies if $d=0$ and $\left( \ell - \left( \frac{D}{\ell}\right) \right)$ descending isogenies otherwise; 
    
    \item For every $n>0$ and every $j \in V_n$, the elliptic curve $E_j$ has, up to equivalence, $\ell$ descending isogenies if $n<d$ and $0$ descending isogenies otherwise;
    
    \item For every $n>0$ and every $j \in V_n$, the elliptic curve $E_j$ has, up to equivalence, exactly $1$ ascending isogeny.
\end{itemize}
\end{proposition}

\begin{proof}
The proof of this proposition is essentially contained in the second part of the proof of \cite[Lemma 6]{Sut13}. We sketch the argument here.

Let $\mathcal{B}_{t,m}$ be the belt containing the volcano $V$. We determine the vertical isogenies between the elliptic curves corresponding to the vertices in $\mathcal{B}_{t,m}$ by induction on the level $n$ of the vertices themselves.

If $n=0$, all the elliptic curves corresponding to the vertices at level $0$ of $\mathcal{B}_{t,m}$ have, by Corollary \ref{cor:horizon}, exactly $1+\left(\frac{D(\mathcal{O}_K)}{\ell}\right)$ horizontal isogenies. Note that, since the order of discriminant $D$ must be $\ell$-dry by the level-$0$ assumption, we have $\left(\frac{D(\mathcal{O}_K)}{\ell}\right)=\left(\frac{D(\mathcal{O})}{\ell}\right)$. The maximal possible number of non-equivalent isogenies from every $E_j$ is $\ell+1$, so the maximal possible number of non-equivalent descending isogenies from every $E_j$ is 
\[
(\ell+1) - \left( 1+ \left(\frac{D(\mathcal{O})}{\ell}\right) \right)=\ell - \left(\frac{D(\mathcal{O})}{\ell}\right).
\]
If $d=0$, there are none. If $d>0$, then by Corollary \ref{cor:everything_is_here} there are exactly $h(\mathcal{O}')$ vertices in $\mathcal{B}_{t,m}$ at level $1$, where $\mathcal{O}'$ is the unique order contained in $\mathcal{O}$ with index $\ell$. Using \cite[Corollary 7.28]{Cox97} and the fact that $D\neq -3,-4$ we have
\[
h(\mathcal{O}')=h(\mathcal{O}) \left( \ell - \left(\frac{D(\mathcal{O})}{\ell}\right) \right).
\]
By Lemma \ref{lem:invertible}, each of the $h(\mathcal{O})\left( \ell - \left(\frac{D(\mathcal{O})}{\ell}\right) \right)$ elliptic curves with CM by $\mathcal{O}'$ admit a vertical isogeny. Since, by duality and the fact that $D(\mathcal{O}) \neq -3,-4$, the number of ascending isogenies from level $1$ to level $0$ is the same as the number of descending isogenies from level $0$ to level $1$, a counting argument shows that every elliptic curve $E_j$ at level $0$ has, up to equivalence, $\left( \ell - \left( \frac{D(\mathcal{O})}{\ell}\right) \right)$ descending isogenies and every elliptic curve $E_j$ at level $1$ has, up to equivalence, precisely $1$ vertical isogeny.

If $n=1$, there are no horizontal isogenies by Corollary \ref{cor:horizon}. Hence, all elliptic curves corresponding to the vertices at level $1$ of $\mathcal{B}_{t,m}$ have one ascending isogeny (by the previous discussion) and $(\ell + 1) -1 =\ell$ descending isogenies if $d>1$, no descending isogenies otherwise. Suppose then that $d>1$. Then by Corollary \ref{cor:everything_is_here} there are exactly $h(\mathcal{O}'')$ vertices in $\mathcal{B}_{t,m}$ at level $2$, where $\mathcal{O}''$ is the unique order contained in $\mathcal{O}'$ with index $\ell$. Using \cite[Corollary 7.28]{Cox97} we see that
\[
h(\mathcal{O}'')= \ell h(\mathcal{O}').
\]
By Lemma \ref{lem:invertible}, each of the $\ell h(\mathcal{O}')$ elliptic curves with CM by $\mathcal{O}''$ admits a vertical isogeny. As above we conclude that every elliptic curve $E_j$ at level $1$ has, up to equivalence, $\ell$ descending isogenies and every elliptic curve $E_j$ at level $2$ has, up to equivalence, precisely $1$ vertical isogeny. An easy induction now leads to the claim. 
\end{proof}

It is now time to deal with the pathological volcanoes containing $0$ and $1728$. We warm up with the following lemma.

\begin{lemma}\label{lem:annul}
Suppose that $0 \in \mathbb{F}_p$ or $1728\in\mathbb{F}_p$  is the $j$-invariant of an ordinary elliptic curve $E$ and let $\ell \neq p$ be a prime. Suppose that there exists a cyclic subgroup $H \subseteq E(\overline{\mathbb{F}}_p) $ of order $\ell$ which is fixed by all the automorphisms of $E$. Then there exists an ideal $\mathfrak{L} \subseteq \mathrm{End}_{\overline{\mathbb{F}}_p}(E)$ of norm $\ell$ such that $H=E[\mathfrak{L}]$.
\end{lemma}

\begin{proof}
We give the details for $0$, the proof is similar for $1728$. There is an isomorphism $\mathcal{O}:= \mathrm{End}_{\overline{\mathbb{F}}_p}(E) \cong  \mathbb{Z}[\zeta_6]$ where $\zeta_6 \in \overline{\mathbb{Q}}$ is a primitive $6$-th root of unity. In particular, $\mathrm{Aut}_{\overline{\mathbb{F}}_p}(E) \cong  \langle \zeta_6 \rangle$. Hence, if $H$ is fixed by all the automorphisms of $E$, then it is in fact fixed by all its endomorphisms. We deduce that $H$ is a cyclic $\mathcal{O}$-module of order $\ell$. We then have an $\mathcal{O}$-module isomorphism
\[
\mathcal{O}/\mathrm{Ann}_\mathcal{O}(H) \cong H
\]
where $\mathrm{Ann}_\mathcal{O}(H)$ is the annihilator of $H$ in $\mathcal{O}$. The isomorphism above shows that $\mathfrak{L}:=\mathrm{Ann}_\mathcal{O}(H)$ is an ideal of $\mathcal{O}$ of norm $\ell$ and $H \subseteq E[\mathfrak{L}]$. Since these two groups have the same cardinality, we deduce that $H = E[\mathfrak{L}]$, as wanted.
\end{proof}

We now describe the directed neighbourhood of the vertex $0$.

\begin{proposition}\label{prop:zero}
Suppose that $0 \in \mathbb{F}_p$ is the $j$-invariant of an ordinary elliptic curve $E$ and let $\ell \neq p$ be a prime. In the volcano containing $0$, the directed subgraph induced by the neighbours of $0$ at level zero and level 1 can be described as:
\begin{itemize}
    \item If $\ell=3$: the vertex $0$ has one self-loop, three descending isogenies towards a unique vertex $j$ at level 1, and there is one unique isogeny from $j$ to $0$;
    \item If $\ell\equiv 1 \text{ mod } 3$: the vertex $0$ has two self-loops. Level 1 has either zero or $(\ell-1)/3$ vertices. In the latter case, each of these vertices receives three descending isogenies from $0$ and sends one ascending isogeny to $0$;
    \item If $\ell\equiv 2 \text{ mod } 3$, the vertex $0$ has no self-loop. Level 1 has either zero or $(\ell+1)/3$ vertices. In the latter case, each of these vertices receives three descending isogenies from $0$, and sends one ascending isogeny to $0$.
\end{itemize}
\end{proposition}

\begin{proof}
The structure of the crater is already described in Proposition \ref{horizontal}. If level $1$ is empty we are done. We then assume that there is at least one $j$-invariant at level $1$.

Let $\mu_6$ be the group of $\overline{\mathbb{F}}_p$-automorphisms of $E$. It acts on the set $S$ of subgroups of order $\ell$ in $E$. The orbits of this action can be described by means of Lemma \ref{lem:annul}: if $H \in S$ is of the form $H=E[\mathfrak{L}]$ for some ideal $\mathfrak{L} \subseteq \mathbb{Z}[\zeta_3]$ of norm $\ell$ then its orbit is a singleton, otherwise the orbit of $H$ contains three elements. Each singleton orbit gives rise to a self-loop around $0$ thanks to Proposition \ref{horizontal}, while each orbit of cardinality 3 gives rise to three directed edges (descending isogenies) from $0$ towards the same $j$-invariant in level $1$ by Lemma \ref{lem:invertible}. Note also that there is a directed edge from each $j$-invariant at level $1$ to $0$ by Lemma \ref{lem:invertible}. Dualising this isogeny, we see that there is at least one, hence three, directed edges from $0$ to each $j$-invariant at level $1$. Let $\mathcal{O}$ be the order corresponding to the vertices at level $1$. The number of vertices on this level is given by the class number $h(\mathcal{O})$ (see Corollary \ref{cor:everything_is_here}), which in turn is equal to
\[
h(\mathcal{O})=\frac{1}{3} \left( \ell - \left(\frac{-3}{\ell}\right) \right)
\] 
by \cite[Corollary 7.28]{Cox97}. On the other hand, the total number of oriented edges from level $0$ to level $1$ is precisely $\ell-\left(\frac{-3}{\ell} \right)$. This, together with the discussion above, implies that for each vertex $j$ at level $1$ there are exactly three directed edges from $0$ to $j$ in the isogeny graph. By counting, one also easily sees that there is exactly one ascending isogeny from each vertex at level $1$ to $0$.

\end{proof}

\begin{figure}[h]
    \centering
    \scalebox{1.5}{\input{fig/example0_3.tikz}}
    \scalebox{1.5}{\input{fig/example0_1.tikz}} 
    \scalebox{1.5}{\input{fig/example0_2.tikz}}
    \caption{All possible neighbourhoods of depth 1 around $0$. Numbers next to directed edges represent multiplicities.}
    \label{fig:neighbours0}
\end{figure}
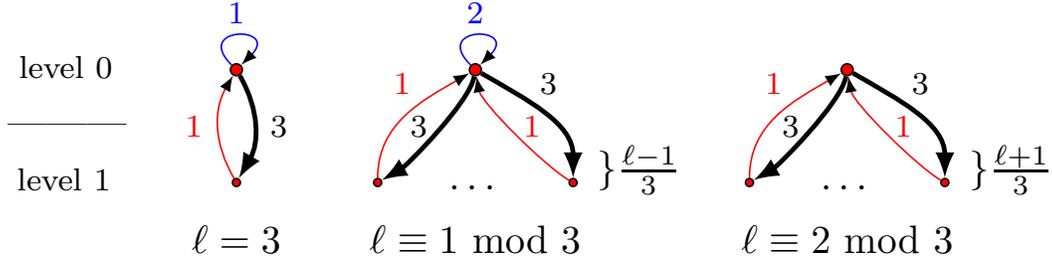

\begin{proposition}\label{prop:1728}
Suppose that $1728 \in \mathbb{F}_p$ is the $j$-invariant of an ordinary elliptic curve $E$ and let $\ell \neq p$ be a prime. In the volcano containing $0$, the directed subgraph induced by the neighbours of $0$ at level zero and level 1 can be described as:
\begin{itemize}
    \item If $\ell=2$: the vertex $1728$ has one self-loop, two descending isogenies towards a unique vertex $j$ at level 1, and there is one unique isogeny from $j$ to $1728$;
    \item If $\ell\equiv 1 \text{ mod } 4$: the vertex $1728$ has two self-loops. Level 1 has either zero or $(\ell-1)/2$ vertices. In the latter case, each of these vertices receives two descending isogenies from $1728$ and sends one ascending isogeny to $1728$;
    \item If $\ell\equiv 3 \text{ mod } 4$, the vertex $1728$ has no self-loop. Level 1 has either zero or $(\ell+1)/2$ vertices. In the latter case, each of these vertices receives two descending isogenies from $1728$, and sends one ascending isogeny to $1728$.
\end{itemize}
\end{proposition}

\begin{proof}
The proof goes along the same lines as the proof of Proposition \ref{prop:zero}.
\end{proof}

\begin{figure}[h]
    \centering
    \scalebox{1.5}{\input{fig/exemple1728_2.tikz}}
    \scalebox{1.5}{\input{fig/example1728_1.tikz}} 
    \scalebox{1.5}{\input{fig/exemple1728_3.tikz}}
    \caption{All possible neighbourhoods of depth 1 around $1728$. Numbers next to directed edges represent multiplicities.}
    \label{fig:neighbours1728}
\end{figure}
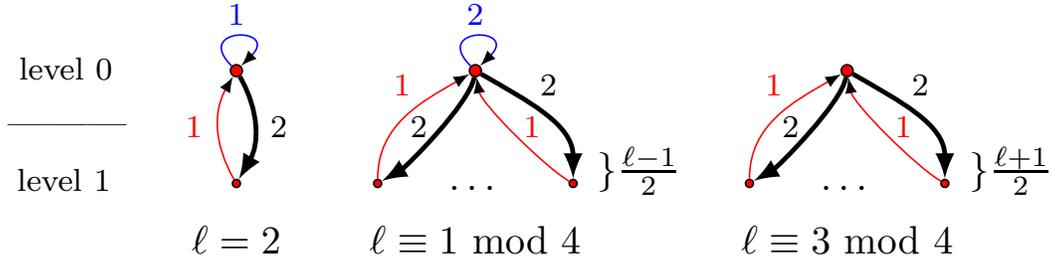

\subsection{Mapping the territory}

In this paragraph, we count the number of volcanic structures.

\begin{lemma}
For $t\in\Z_{>0}$, the number of non-empty $t$-cordilleras in $\mathcal{G}_\ell(\mathbb{F}_p)$ is $\lfloor 2\sqrt{p} \rfloor$.
\end{lemma}
\begin{proof}
The number of cordilleras is the number of possible traces up to sign. By Waterhouse \cite[Theorem 4.1 page 536]{Wat69}, each trace in the Hasse interval $[-2\sqrt{p},2\sqrt{p}]$ is attained over $\mathbb{F}_p$. This gives $\lfloor 2\sqrt{p} \rfloor$ possibilities for non-zero traces up to sign. It is independent of $\ell$.
\end{proof}

\begin{lemma}
Let $t\in{[1,2\sqrt{p}]}$ be an integer. Let $K$ be the associated field $\mathbb{Q}(\sqrt{t^2-4p})$. Let $D(\mathcal{O}_K)$ denote its discriminant. Let $v=\sqrt{(t^2-4p)/D(\mathcal{O}_K)}$.
Let $d=v_\ell(v)$. The number of belts in the $t$-cordillera is $\omega(v\ell^{-d})$, where $\omega(k)$ is the number of positive divisors of the integer $k$.
\end{lemma}
\begin{proof}
  For every positive divisor $m$ of $v\ell^{-d}$, we have the belt $B_{t,m}$ induced by the following set of vertices:
$$\mathcal{V}_{t,m}=\{j\in{\mathcal{V}_t}\,\vert\, \widetilde{\mathrm{End}(E_j)}\cong\mathbb{Z}+m\ell^d \mathcal{O}_K\}.$$
\end{proof}

\begin{lemma}
Let $\mathcal{B}_{t,m}$ be a belt associated with a trace $t\neq0$ and integer $m$. Let $K$ be the associated field. Let $\mathcal{O}=\mathbb{Z}+m \mathcal{O}_K$.
\begin{itemize}
    \item[(i)] If $\ell$ is inert, the number of volcanoes in $\mathcal{B}_{t,m}$ is $h(\mathcal{O})$;
    \item[(r)] If $\ell$ is ramified, if $\ell \mathcal{O}=\mathcal{L}^2$ with $\mathcal{L}$ principal, then the number of volcanoes in $\mathcal{B}_{t,m}$ is $h(\mathcal{O})$.
    If $\ell \mathcal{O}=\mathcal{L}^2$ with $\mathcal{L}$ not principal, then the number of volcanoes in $\mathcal{B}_{t,m}$ is $h(\mathcal{O})/2$;
    \item[(s)] If $\ell$ is split and $\ell \mathcal{O}=\mathfrak{L} \overline{\mathfrak{L}}$, let $r$ be the order of $\mathcal{L}$ in $\Cl(\mathcal{O})$, then the number of volcanoes in $\mathcal{B}_{t,m}$ is $h(\mathcal{O})/r$.
\end{itemize}
\end{lemma}

\begin{proof}
The number of volcanoes is the number of craters, and the craters are isomorphic within a belt. The total number of vertices of depth 0 in $\mathcal{B}_{t,m}$ is $h(\mathcal{O})$. The size of each crater is given by Proposition \ref{horizontal}. 
\end{proof}

\begin{remark}
All volcanoes on the same cordillera have the same depth. All volcanoes on the same belt have the exact same shape. The total number of vertices at level $0$ in a belt with endomorphism order $\mathcal{O}$ is $h(\mathcal{O})$.
\end{remark}

\begin{lemma}\label{count}
Let $V$ be a volcano in the belt $\mathcal{B}_{t,m}$, where $t$ corresponds to a cordillera of depth $d$. Let $K$ be the associated field. Let $\mathcal{O}=\mathbb{Z}+m \mathcal{O}_K$. The number of vertices in the volcano $V$ is

$$c+\frac{2c}{\# \mathcal{O}^\times} \left(\left(\ell-\left(\frac{D(\mathcal{O})}{\ell}\right) \right)\frac{\ell^d-1}{\ell-1}\right).$$
\end{lemma}

\begin{proof}
Use Proposition \ref{descending} in the regular case and Propositions \ref{prop:zero} and \ref{prop:1728} in the non-regular case. A formula for $c$ is given in Proposition \ref{horizontal}.
\end{proof}

\begin{remark}
There are exactly $p$ $j$-invariants in $\F_p$. Therefore, by taking into account supersingular curves, we can partition $p$ and obtain a mass formula. A full study in the case $p=1009$ and $\ell=3$ is enclosed in the appendix \ref{app:1009}.
\end{remark}

\section{The Inverse Volcano Problem}
\label{sec:inv}

Inspired by the structure of the connected components of $\mathcal{G}_\ell(\mathbb{F}_p)$, we give the following definition.

\begin{definition}\label{volcano:graph}
An \emph{abstract volcano} $V=(\mathcal{V}, \mathcal{E})$ of \emph{depth} $d\ge0$ is a connected undirected graph together with a distinguished subset $\mathcal{V}_0 \subseteq \mathcal{V}$ such that the subgraph $V_0$ induced by $\mathcal{V}_0$ is one of the graphs described in Proposition \ref{horizontal}. Moreover, $d>0$ if and only if $\mathcal{V} \setminus \mathcal{V}_0 \neq \emptyset$, in which case there exists a partition
\[
\mathcal{V} \setminus \mathcal{V}_0 = \bigcup_{i=1}^d \mathcal{V}_i
\]
and a prime number $\ell \in \mathbb{Z}$ such that, denoting by $V_i$ the subgraph induced by $\mathcal{V}_i$, the following holds:
\begin{enumerate}
    \item All vertices in $\mathcal{V}_0 \cup ... \cup \mathcal{V}_{d-1}$ have degree $\ell+1$ and all vertices in $\mathcal{V}_d$ have degree $1$;
    \item If $v \in \mathcal{V}_r$ and $v' \in \mathcal{V}_k$ are connected by an edge, then $|r-k| \leq 1$;
    \item For all $0<r\leq d$, the graph $V_r$ is totally disconnected;
    \item For $0<r \leq d$, each vertex in $V_r$ has exactly one edge to a vertex in $V_{r-1}$;
    \end{enumerate}
We call $V_0$ the \textit{crater} of $V$ and we say that the vertices in $\mathcal{V}_r$ lie at \textit{level} $r$.
\end{definition}
If the depth $d$ of an abstract volcano $\mathcal{V}$ is strictly positive, then the prime $\ell$ appearing in Definition \ref{volcano:graph} is uniquely determined by condition (1) above. In this case, we will also speak of $\mathcal{V}$ as of an (abstract) $\ell$-volcano.

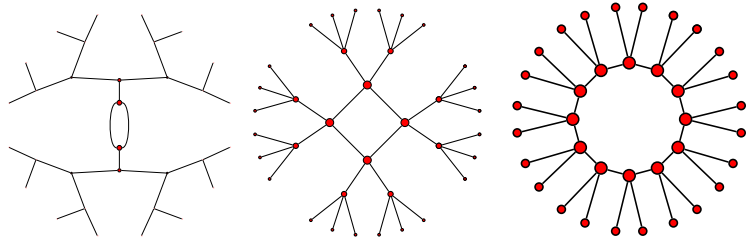
\begin{figure}[h]
    \centering
    \scalebox{0.6}{\input{volcanoes/l=2/1009/30.tikz}}
    \scalebox{1}{\input{volcanoes/l=3/3121/783.tikz}} 
    \scalebox{1.5}{\input{volcanoes/l=3/3121/1.tikz}}
    \caption{Three abstract volcanoes.}
    \label{fig:volcanoes}
\end{figure}

It follows from the discussion in Section \ref{sec:volc} that a connected component of $\mathcal{G}_\ell(\mathbb{F}_p)$ not containing $0$ or $1728$ is an abstract volcano in the sense above. One may wonder whether the converse is also true. More precisely, we pose the following question.

\begin{question}[Inverse Volcano Problem over $\mathbb{F}_p$] \label{inverse volcano problem}
Let $V$ be an abstract volcano as in Definition~\ref{volcano:graph}. Can we find primes $p, \ell\in \mathbb{Z}$ with $p\neq \ell$ such that $V$ is a connected component in the isogeny graph $\G_{\ell}(\F_p)$?
\end{question}  

Roughly speaking, one can study this question by dividing it into two subcases:
\begin{enumerate}
    \item The volcano $V$ has depth $d=0$ i.e. $V=V_0$;
    \item The volcano $V$ has depth $d>0$ i.e. $V \neq V_0$.
\end{enumerate}
These two cases are fundamentally different in nature. Indeed, in the first case we can try to find (and we will find) the volcano $V$ as connected component of $\mathcal{G}_{\ell}(\mathbb{F}_p)$ where both $\ell$ and $p$ are allowed to vary. However, in the second case the prime $\ell$ is fixed, since it must be one less than the degree of any vertex in $V_0$. Hence, in this second setting we have less freedom of choice and the inverse volcano problem becomes more difficult.

To study Question \ref{inverse volcano problem}, it is useful to introduce a couple of definitions. First of all note that, letting $V_0$ be one of the graphs described in Proposition \ref{horizontal}, $\ell$ a prime number and $d\in \mathbb{Z}_{>0}$, there exists a unique abstract $\ell$-volcano $V=V(V_0,\ell, d)$ with crater $V_0$ and depth $d$. We call $V$ the \textit{volcano induced by the triple $(V_0,\ell,d)$}. We also give the following definition of abstract crater associated with a prime ideal.

\begin{definition} \label{def:induced_crater}
Let $\mathcal{O}$ be an imaginary quadratic order and $\ell \in \mathbb{Z}_{>0}$ a prime number not dividing the conductor of $\mathcal{O}$. Choose a prime ideal $\mathfrak{L} \subseteq \mathcal{O}$ lying above $\ell$. The \textit{abstract crater $V_0$ associated with $\mathfrak{L}$} is the graph consisting of: 
\begin{itemize}
 \item A single vertex if $\ell$ is inert in $\mathcal{O}$;
    \item A single vertex with a self-loop if $\ell$ is ramified in $\mathcal{O}$ and $\mathfrak{L}$ is principal;
    \item A single vertex with two self-loops if $\ell$ splits in $\mathcal{O}$ and $\mathfrak{L}$ is principal;
    \item Two vertices connected with a single edge if $\ell$ ramifies in $\mathcal{O}$ and $\mathfrak{L}$ is not principal;
    \item A cycle of size the order of $[\mathfrak{L}]$ in the class group $\Cl(\Or)$ if $\ell$ splits in $\mathcal{O}$ and we are not in any of the previous cases.
\end{itemize}
\end{definition}
The relationship between Definition~\ref{def:induced_crater}, Proposition~\ref{horizontal} and Question~\ref{inverse volcano problem} is clear: if we want to realise an $\ell$-volcano $V$ as the connected component of $\mathcal{G}_\ell(\mathbb{F}_p)$ for some prime $p$, it seems natural to begin by realising its crater $V_0$. By Proposition \ref{horizontal}, a good start would be to find an imaginary quadratic field $K$ whose ring of integers contains an ideal $\mathfrak{L} \mid \ell$ such that $V_0$ is equal to the abstract crater associated with $\mathfrak{L}$ (here we are evidently hoping to view the vertices of $V_0$ as $j$-invariants of CM elliptic curves with complex multiplication by the \textit{maximal order} in $K$).

In fact, we will now prove that, in order to solve the inverse volcano problem, realising $V_0$ as a crater is the key step.

\begin{proposition} \label{prop:abstract_crater_realizable}
Let $\mathcal{O}$ be an order of discriminant $D(\mathcal{O})<-4$ in an imaginary quadratic field $K$ and let $\ell \in \mathbb{Z}_{>0}$ be a prime not dividing the conductor of $\mathcal{O}$. Let $\mathfrak{L} \subseteq \mathcal{O}$ be a prime lying above $\ell$ and let $V_0$ be the abstract crater associated with $\mathfrak{L}$. Then for every $d>0$ there exist infinitely many primes $p=p(d) \in \mathbb{Z}_{>0}$ such that $\mathcal{G}_\ell(\mathbb{F}_p)$ contains the volcano induced by $(V_0, \ell, d)$ as connected component. Moreover, if $\ell\neq 2$ the same is true with $d=0$.
\end{proposition}

\begin{proof}
Given a positive integer $d$, our goal is to find infinitely many primes $p \in \mathbb{Z}_{>0}$ and pairs $(t,v) \in \mathbb{Z}^2$ such that $t\neq 0$, the prime power $\ell^d$ divides exactly $v$ and $4p=t^2-v^2 D(\mathcal{O})$. Indeed, suppose this is the case. Then of course we must have $-2\sqrt{p} \leq t \leq 2\sqrt{p}$ so by Lemma \ref{lem:cordillera_empty} the $t$-cordillera in $\mathcal{G}_\ell(\mathbb{F}_p)$ is non-empty. Since the order $\mathcal{O}$ certainly contains the order of discriminant $v^2 D(\mathcal{O})$ and $\ell$ does not divide the conductor of $\mathcal{O}$, by Proposition \ref{descending} combined again with Lemma \ref{lem:cordillera_empty} the $t$-cordillera possesses a connected component isomorphic to the volcano induced by $(V_0, \ell, d)$.

To find primes $p$ and corresponding couples $(t,v)$ as above, we proceed as follows. For every $k \in \mathbb{N}$, denote by $H_k$ the ring class field of the order $\mathbb{Z}[\ell^k \sqrt{D(\mathcal{O})}] \subseteq \mathcal{O}$. We have that $H_k \subsetneq H_{k+1}$ for all $k \in \mathbb{N}$. To see this, one can simply use \cite[Corollary 7.28]{Cox97} to compute $[H_{k+1}:H_k]$: if $\ell \neq 2$ one has $[H_{k+1}:H_k] \in \{\ell-1, \ell, \ell+1 \}$ for all $k\geq 0$, while in the case $\ell=2$ one has $[H_{k+1}:H_k]=2$ for all $k\geq 0$ since $\mathbb{Z}[2^k \sqrt{D(\mathcal{O})}]$ has always even discriminant. 

In particular $H_d \subsetneq H_{d+1}$, so by the Chebotar\"ev density theorem there are infinitely many primes $p\nmid 2\ell D(\mathcal{O})$ that split completely in $H_d$ but do not split completely in $H_{d+1}$. Using \cite[Theorem 9.4]{Cox97}, we see that there exist $x,y \in \mathbb{Z}$ such that $p=x^2-D(\mathcal{O})\ell^{2d}y^2$. Moreover, we also have $\ell \nmid y$, since otherwise there would exist $\tilde{y} \in \mathbb{Z}$ such that $p=x^2-D(\mathcal{O})\ell^{2d+2}\tilde{y}^2$ and then $p$ would split completely (again by \cite[Theorem 9.4]{Cox97}) in $H_{d+1}$, contradicting our assumptions. Now such a $p$ satisfies the norm equation
\[
4p=t^2-v^2D(\mathcal{O})
\]
where $t=2x$ and $v=2\ell^d y$. We certainly have $t\neq 0$ since $p$ is split in $K$ by our choices. Moreover, if $\ell \neq 2$, the power $\ell^d$ divides exactly $v$ so in this case the theorem is proved.

If $\ell=2$ then the same arguments work by considering primes splitting completely in $H_{d-1}$ but not in $H_d$ (here we use the fact that $d>0$). This concludes the proof.
\end{proof}

\begin{remark}
In the above proof we have chosen a down-to-earth approach, proving the existence of the volcano induced by $(V_0,\ell,d)$ in $\mathcal{G}_\ell(\mathbb{F}_p)$ by solving an equation of the form $4p=t^2-v^2 \ell^{2d} D(\mathcal{O})$ with $t,v \in \mathbb{Z}$. We could have been more sophisticated and argued as follows: if one manages to find a prime that splits completely in the ring class field of $\mathbb{Z}+\ell^d \mathcal{O}$ but not in the ring class field of $\mathbb{Z} + \ell^{d+1} \mathcal{O}$, then the residue field at $p$ will certainly contain all $j$-invariants of elliptic curves with CM by $\mathbb{Z}+\ell^d \mathcal{O}$ but no $j$-invariant of elliptic curves with CM by $\mathbb{Z} + \ell^{d+1} \mathcal{O}$ (cfr. Corollary \ref{cor:everything_is_here}) and this would ensure the existence of the desired volcano. By the Chebotar\"ev density theorem, this can be achieved if and only if the two aforementioned ring class fields are distinct, which certainly happens if $\ell \neq 2$ or $\ell=2$ and $d>0$.

Ultimately, the difference in the two proofs lies in the following fact: in the first proof we have solved the equation
\begin{equation} \label{eq:4p}
    4p=t^2-v^2 \ell^{2d} D(\mathcal{O})
\end{equation}
by first solving the auxiliary equation 
\begin{equation} \label{eq:p}
    p=x^2-y^2 \ell^{2d} D(\mathcal{O})
\end{equation}
and then multiplying by $2$ the found $x,y$. However, \eqref{eq:4p} may have a solution even if \eqref{eq:p} does not. For instance, the prime $3$ is certainly not of the form $x^2+11 y^2$, but we have $4\cdot 3=1^2+11\cdot 1^2$. The proof appearing in this remark directly solves equation \eqref{eq:4p} without expressing $p$ itself in the form $t^2-v^2 \ell^{2d} D(\mathcal{O})$. This may be useful in view of explicit computations, since the smallest prime $p$ solving \eqref{eq:4p} is smaller or equal than the smallest prime solving \eqref{eq:p}.
\end{remark}

\subsection*{Depth \texorpdfstring{$d=0$}{d=0}}
Let us now analyze more closely the first instance of the inverse volcano problem, that is, the case when our given volcano $V$ coincides with its crater $V_0$. We will now answer Question \ref{inverse volcano problem} in this case.

\begin{proof}[Proof of Theorem \ref{thm:abstract_crater_depth0_intro}]
We refer to the possible shapes of $V=V_0$ as described in Proposition \ref{horizontal}. For each of the cases appearing in the proposition, we first want to find an imaginary quadratic field $K$ and a prime $\mathfrak{L} \subseteq \mathcal{O}_K$ such that $V$ is the abstract crater associated with $\mathfrak{L}$. In cases $(1)-(5)$ it is easy to find such a field $K$ of discriminant $D(\mathcal{O}_K)<-4$ and a prime ideal $\mathfrak{L}$ with odd residue characteristic.

To deal with the case where $V$ is a cycle of length $n\geq 3$ we appeal to~\cite[Theorem 2]{Yamamoto}, which ensures the existence of an imaginary quadratic field $K$ of discriminant $<-4$ whose class group contains an element of order $n$. Since by~\cite[Theorem 9.12]{Cox97} every ideal class contains infinitely many prime ideals, we deduce that there exists a prime $\ell>2$ that splits in $\mathcal{O}_K$ into two prime ideals of order $n$ in the class group.

Applying Proposition~\ref{prop:abstract_crater_realizable} now allows us to conclude.
\end{proof}

\subsection*{Depth \texorpdfstring{$d>0$}{d>0}} 
We now turn to the second, more difficult instance of the inverse volcano problem \textit{i.e.} the case when $V$ has depth $d>0$.
Given an $\ell$-volcano $V$ of depth $d>0$, realising its crater $V_0$ as an abstract crater amounts to finding an imaginary quadratic order of conductor coprime to $\ell$ where there exists a prime ideal $\mathfrak{L} \supseteq \ell$ satisfying the condition in Proposition \ref{horizontal} corresponding to $V_0$. In order to apply Proposition \ref{prop:abstract_crater_realizable} we may also want the discriminant of the order to be smaller than $-4$. So let us fix $\ell$ prime and see if we manage to find, for each of the six conditions expressed in Proposition \ref{horizontal}, an imaginary quadratic order that realises that condition:
\begin{enumerate}
    \item By Dirichlet's theorem on primes in arithmetic progression, there are infinitely many imaginary quadratic fields where $\ell$ is inert;
    \item If $\ell \neq 3$, then $\ell$ ramifies in a principal ideal in the ring of integers of $\mathbb{Q}(\sqrt{-\ell})$ and the latter has discriminant $<-4$. If $\ell=3$ the same is true if we consider the order $\mathbb{Z}[\sqrt{-3}]$ of conductor $2$ in $\mathbb{Q}(\sqrt{-3})$;
    \item In the imaginary quadratic field $K=\mathbb{Q}(\sqrt{1-4\ell})$ the integral element $\alpha=\frac{1+\sqrt{1-4\ell}}{2}$ has norm $\ell$. We deduce that $\ell$ splits in $\mathcal{O}_K$ into the two principal prime ideals generated by $\alpha$ and $\overline{\alpha}$;
    \item If $q$ is a prime that is sufficiently large with respect to $\ell$ then in the ring of integers of $K=\mathbb{Q}(\sqrt{-\ell q})$ the prime $\ell$ is ramified into a non-principal ideal. This follows from the fact that every element $\alpha \in \mathcal{O}_K \setminus \mathbb{Z}$ has norm $N_{K/\mathbb{Q}}(\alpha) \geq (1+\ell q)/4$.
\end{enumerate}

Conditions (5) and (6) in Proposition \ref{horizontal} are more obscure, as they require to construct an imaginary quadratic field $K$ where $\ell$ splits into two prime ideals whose class in the ideal class group of $K$ has prescribed order $n$. We now prove that it is always possible to find such a field. Our construction is inspired by the techniques used by Nagell in \cite{Nagell_1955} and very much depends on whether $\ell=2$ or $\ell>2$. We begin by treating the first case.

\begin{proposition}
\label{prop:prime_given_order_2}
Let $n \neq 4$ be a positive integer and let $K=\mathbb{Q}(\sqrt{1-2^{n+2}})$. Then in $\mathcal{O}_K$ the prime $2$ splits into two prime ideals whose corresponding classes in $\mathrm{Cl}(\mathcal{O}_K)$ have order $n$.
\end{proposition}

\begin{proof}
Some preliminary remarks: write $\sqrt{1-2^{n+2}}= x \cdot \sqrt{-A}$ with $x,A \in \mathbb{Z}$ and $A>0$ squarefree. In particular, we can write \begin{equation} \label{eq:diophantine equation for 2}
    Ax^2+1=2^{n+2}.
\end{equation}
Moreover, since $n\geq 1$ we know that $x$ is odd and we thus also have
\begin{equation} \label{eq:1 mod 8}
    1 \equiv 1-2^{n+2} \equiv x^2 \cdot (-A) \equiv -A \text{ mod } 8
\end{equation}
so that $D(\mathcal{O}_K)=-A$ and $2 \mathcal{O}_K$ splits into two distinct conjugate prime ideals $\mathfrak{p}_2$ and $\overline{\mathfrak{p}}_2$.

Consider now the two conjugate principal ideals $\mathfrak{a}:=\left(\frac{1+x\sqrt{-A}}{2}\right)$ and $\overline{\mathfrak{a}}$. We have
\[
\mathfrak{a} \cdot \overline{\mathfrak{a}} = N_{K/\mathbb{Q}} \left(\frac{1+x\sqrt{-A}}{2} \right)=\frac{Ax^2+1}{4}=2^{n}
\]
hence these two ideals can be divisible only by $\mathfrak{p}_2$ and $\overline{\mathfrak{p}}_2$. Moreover, the ideals $\mathfrak{a}$ and $\overline{\mathfrak{a}}$ are coprime, as one can see by adding their generators. Hence, we can assume without loss of generality that the prime ideal factorization of $\mathfrak{a}$ is
\[
\mathfrak{a}=\left( \frac{1+x\sqrt{-A}}{2} \right)=\mathfrak{p}_2^n
\]
and in particular, the class of $\mathfrak{p}_2$ in $\mathrm{Cl}(\mathcal{O}_K)$ has order dividing $n$. Our goal for the rest of this proof is to prove that this order is precisely $n$. 

Assume by contradiction that this is not the case. Then there exists a prime $q$ and $\tilde{u}, \tilde{v} \in \mathbb{Z}$ such that the following equality of ideals holds:
\begin{equation} \label{eq:equality of ideals ell=2}
    \left(\frac{\tilde{u}+\tilde{v} \sqrt{-A}}{2}\right)^q = \left( \frac{1+x\sqrt{-A}}{2} \right).
\end{equation}
We now distinguish two cases.
\vspace{0.3cm}

\textbf{First case: $q$ odd.} We begin by noting that $A \neq 3$ since by \eqref{eq:1 mod 8} we have $-A \equiv 1 \text{ mod } 8$. This implies, using that fact that $q$ is odd, that all the units in $\mathcal{O}_K$ are $q$-th powers. Hence, there exist $u,v \in \mathbb{Z}$ with $u\equiv v \text{ mod } 2$ such that the following equality \textit{of elements of $K$} holds:
\begin{equation} \label{eq:equality_of_elements}
    \left ( \frac{u+v \sqrt{-A}}{2}\right )^q =  \frac{1+x\sqrt{-A}}{2}.
\end{equation}
Expanding the left-hand side and collecting the rational terms together, we reach the equality
\begin{equation} \label{eq:binomial_equality}
    u^q-\binom{q}{2}u^{q-2}v^2 A+ ... + \binom{q}{q-1} u v^{q-1} (-A)^{\frac{q-1}{2}} = 2^{q-1}.
\end{equation}
Clearly, \eqref{eq:binomial_equality} implies that $u$ divides $2^{q-1}$. Suppose initially that $u$ is even \textit{i.e.} $u=2c$ for some $c \in \mathbb{Z}$. Since $u\equiv v \text{ mod } 2$ we can also write $v=2d$ for some $d \in \mathbb{Z}$. Plugging $2c$ and $2d$ in place of $u$ and $v$ in equation \eqref{eq:binomial_equality}, one obtains that $2^q$ divides $2^{q-1}$, a contradiction.

Hence, $u$ must be odd, \textit{i.e.} $u = \pm 1$. Reducing \eqref{eq:binomial_equality} modulo $q$ we obtain
\[
(\pm 1)^q \equiv 2^{q-1} \equiv 1 \text{ mod } q,
\]
so, since $q$ is odd, we in fact have $u=1$. In particular, equality \eqref{eq:equality_of_elements} now reads
\[
\left( \frac{1+v \sqrt{-A}}{2}\right)^q =  \frac{1+x\sqrt{-A}}{2}
\]
and taking norms we obtain
\[
\left(\frac{1+v^2 A}{4} \right)^q=\frac{Ax^2+1}{4}.
\]
After some manipulations, this can be rewritten as follows:
\[
A^2 x^2 - (A+v^2A^2)\left[\left(\frac{1+v^2A}{4}\right)^{\frac{q-1}{2}} \right]^2 = -A
\]
so we reach an equation of the form
\begin{equation} \label{eq:Pell}
    U^2-D V^2=-A
\end{equation}
where $U=Ax$, $D=A+v^2 A^2$ and $V=[(1+v^2A)/4]^{(q-1)/2}$. We now wish to apply Mahler's theorem~\cite[Theorem 16]{Nagell_1955}, since all the prime factors of $V$ divide $D$. Let us verify that the hypotheses of the theorem are satisfied:
\begin{itemize}
    \item We certainly have $A \neq 1$ since $-A \equiv 1 \text{ mod } 8$. This also implies that $A \neq D$. Moreover $A$ is squarefree by hypothesis; 
    \item By the previous bullet point we have $D=A(1+v^2A)>1$. Moreover, $D$ is not a perfect square since $v \neq 0$ (because it is odd), so all the primes dividing $A$ cannot divide also $1+v^2 A$;
\end{itemize}
Hence, \cite[Theorem 16]{Nagell_1955} implies that the only positive solutions to \eqref{eq:Pell} are given by the fundamental solution $U,V=U_1,V_1$ (that is, the solution with $U,V>0$ such that $|U+\sqrt{D}V|$ is the smallest possible in absolute value) and by 
\[
U=\frac{U_1^3+3U_1 V_1^2 D}{A}, \hspace{0.5cm} V=\frac{3U_1^2 V_1+DV_1^3}{A}.
\]
In our case, the fundamental solution is given by $U=vA$ and $V=1$. This latter equality yields in particular $v^2 A=3$, hence $A=3$, which is not possible.

By looking at the $V$ of the non-fundamental solution, we find
\[
\left(\frac{1+v^2A}{4}\right)^{\frac{q-1}{2}}=V=\frac{3v^2 A^2 + D}{A}=1+4v^2 A=4(1+v^2 A)-3.
\]
In order to conclude, note that since $v$ is odd, we have 
\[
1+v^2A \equiv 1+1\cdot 7 \equiv 0 \text{ mod } 8
\]
so in particular the left-hand side of the above equality is even. However, the right-hand side of the equality is odd, contradiction. This concludes the proof in this case.
\vspace{0.3cm}

\textbf{Second case: $q=2$.} By \eqref{eq:1 mod 8} we have $A\neq 1,3$ and so $\mathcal{O}_K^\times =\{\pm1\}$. Hence the equality of ideals \eqref{eq:equality of ideals ell=2} yields an equality of elements of $K$
\[
\left(\frac{u+v\sqrt{-A}}{2}\right)^2=\pm \frac{1+x\sqrt{-A}}{2}
\]
where $u,v\in\Z$ are such that $u\equiv v \text{ mod } 2$ (here we have simply set $u=\tilde{u}$ and $v=\tilde{v}$).
Expanding and looking at rational/irrational parts gives
\[
\begin{cases}
u^2-v^2A  &=\pm 2 \\
2uv &=\pm 2x.
\end{cases}
\]
By substituting $v=\pm x/u$ in the first equation and expanding we get
\[
u^4\mp 2u^2-Ax^2=0,
\]
and solving this quadratic equation gives, after using \eqref{eq:diophantine equation for 2}
\[
u^2=2^{n/2+1}\pm1.
\]
Looking modulo $4$ one sees that $u^2=2^{n/2+1}-1$ cannot hold, as $n\ge2$. Hence we must have $u^2=2^{n/2+1}+1$, or otherwise written $(u-1)(u+1)=2^{n/2+1}$. This implies yields $u=\pm3$ and $n=4$. However, this case is excluded by our assumptions and the theorem is proved.
\end{proof}

One can directly verify that in $\mathbb{Q}(\sqrt{-39})$ the prime $2$ splits into two prime ideals having order $4$ in the class group. This observation and the previous proposition imply that for every $n \in \mathbb{Z}_{>0}$ there exists an imaginary quadratic field $K$ where $2$ splits into two prime ideals having order $n$ in $\mathrm{Cl}(\mathcal{O}_K)$.

We now treat the case when $\ell$ is odd. We will prove that for every $n \in \mathbb{Z}_{>0}$ at least one among the imaginary quadratic fields $\mathbb{Q}(\sqrt{1-\ell^n})$ and $\mathbb{Q}(\sqrt{1-4\ell^n})$, call it $K$, has the property that the prime $\ell$ splits in $\mathcal{O}_K$ into two prime ideals having order $n$ in $\mathrm{Cl}(\mathcal{O}_K)$. Let us begin by studying these fields separately.

\begin{proposition} \label{prop:prime_given_order_l}
Let $\ell \in \mathbb{N}$ be an odd prime and let $n \in \mathbb{Z}_{>0}$. Define $K:=\mathbb{Q}(\sqrt{1-\ell^{n}})$. Suppose that:
\begin{enumerate}
    \item Either $n \geq 3$ is odd and $(\ell, n) \neq (3,5)$;
    \item Or $n$ is even and neither $\frac{\ell^{n/2}+1}{2}$ nor $\frac{\ell^{n/2}-1}{2}$ is a square;
\end{enumerate}
Then in $\mathcal{O}_K$ the prime $\ell$ splits into two prime ideals whose corresponding classes in $\mathrm{Cl}(\mathcal{O}_K)$ have order $n$.
\end{proposition}

\begin{proof}
We proceed as in the proof of Proposition \ref{prop:prime_given_order_2}. Write $\sqrt{1-\ell^{n}}= x \cdot \sqrt{-A}$ with $x,A \in \mathbb{Z}$ and $A>0$ squarefree, so that we have
\begin{equation} \label{eq:diophantine equation for l}
    Ax^2+1 = \ell^n.
\end{equation}
This equation implies in particular that $-A$ is a square modulo $\ell$, so that $\ell \mathcal{O}_K=\mathfrak{p}_\ell \overline{\mathfrak{p}}_\ell$ with $\mathfrak{p}_\ell, \overline{\mathfrak{p}}_\ell \subseteq \mathcal{O}_K$ distinct prime ideals.
Consider the two conjugate principal ideals $\mathfrak{a}:=(1+x\sqrt{-A})$ and $\overline{\mathfrak{a}}$. We have
\[
\mathfrak{a} \cdot \overline{\mathfrak{a}} = N_{K/\mathbb{Q}}(1+x\sqrt{-A})=\ell^n.
\]
Since $\ell$ is odd, the ideals $\mathfrak{a}$ and $\overline{\mathfrak{a}}$ are coprime, so we have, without loss of generality, $\mathfrak{a}=\mathfrak{p}_\ell^n$. In particular, the class of $\mathfrak{p}_\ell$ in $\mathrm{Cl}(\mathcal{O}_K)$ has order dividing $n$. If this order is not precisely $n$, then there exists a prime $q$ and $u,v \in \mathbb{Z}$ with $u\equiv v \text{ mod } 2$ such that the following equality of ideals holds:
\begin{equation} \label{eq:equality of ideals l not 2}
    (1+x\sqrt{-A})=\left( \frac{u+v\sqrt{-A}}{2} \right)^q.
\end{equation}
Suppose first that $q$ is odd. Then, the proof of \cite[Theorem 25]{Nagell_1955} shows that we must have $x=11$, $A=2$, $\ell=3$ and $q=5$. From \eqref{eq:diophantine equation for l}, we deduce that $q=n=5$, which contradicts assumption $(1)$.

Hence, we must have $q=2$ and, in particular, $n$ is even. Reducing equation \eqref{eq:diophantine equation for l} modulo $4$ and using that $A$ is squarefree, we see that $x$ is even, say $x=2y$ for $y\in\mathbb{Z}$. Now we can write
\begin{equation}\label{eq:lm}
    Ay^2=\left(\frac{\ell^{n/2}-1}{2}\right)\left(\frac{\ell^{n/2}+1}{2}\right).
\end{equation}
The factors on the right hand side are consecutive, hence coprime, integers. By assumption $(2)$ neither of them can be a square, and we deduce that $A$ must be divisible by at least two different primes. In particular, $A>3$ and $\mathcal{O}_K^\times = \{\pm1\}$.

Now from \eqref{eq:equality of ideals l not 2} we get the following equality of elements of $K$:
\[
\left(\frac{u+v\sqrt{-A}}{2}\right)^2=\pm(1+x\sqrt{-A})
\]
Expanding and looking at rational/irrational parts gives
\[
\begin{cases}
u^2-v^2A  &=\pm4 \\
2uv &=\pm4x.
\end{cases}
\]
From the second equation we get $u\equiv v\equiv x \equiv 0 \text{ mod } 2$. So writing $(u',v',y)=(\frac{u}{2},\frac{v}{2},\frac{x}{2})$ and proceeding as in the second case of Proposition \ref{prop:prime_given_order_2}, we get
\[
u^2=\frac{\pm 1 + \sqrt{1+4y^2A}}{2}=\frac{\pm 1 + \ell^{n/2}}{2},
\]
which is excluded by our hypotheses. This concludes the proof.
\end{proof}

\begin{proposition}\label{prop:prime_given_order_l_part_2}
Let $\ell \in \mathbb{N}$ be an odd prime, and $n$ an even positive integer. Define $K:=\Q(\sqrt{1-4\ell^n})$. If $\ell^{n/2}$ is not the sum two consecutive squares, then $\ell$ splits in $\mathcal{O}_K$ into two prime ideals whose classes in $\Cl(K)$ have order $n$.
\end{proposition}
\begin{proof}
Write $\sqrt{1-4\ell^n}=x\cdot\sqrt{-A}$ with $x,A\in \mathbb{Z}$ and $A>0$ squarefree. In particular,
\begin{equation}\label{eq:4lm}
Ax^2+1=4\ell^n.
\end{equation}
We have that $x$ is odd as it divides $4\ell^n-1$. Thus $A\equiv 3 \text{ mod }8$ and so $A\neq1$. Suppose we have $A=3$. Then \eqref{eq:4lm} becomes
\[
3x^2=(2\ell^{n/2}-1)(2\ell^{n/2}+1).
\]
Both factors on the right hand side are consecutive odd integers, so they are coprime. Hence one of them must be a perfect square and the other three times a perfect square. Suppose that $2\ell^{n/2}-1$ is a square. Since it is odd, we would have $2\ell^{n/2}-1=(2j+1)^2$ which is equivalent to $\ell^{n/2}=j^2+(j+1)^2$, contradicting our assumptions. This means that there exists $k\in\Z_{>0}$ such that $2\ell^{n/2} +1 = k^2$. However, reducing this equality modulo $4$ shows that this cannot happen either and we conclude that $A\neq3$. In particular, $\mathcal{O}_K^\times = \{\pm 1 \}$.

Equality \eqref{eq:4lm} implies that $\ell$ splits in $\mathcal{O}_K$ into distinct conjugate prime ideals $\mathfrak{p}_{\ell}$ and $\overline{\mathfrak{p}}_\ell$. Now consider the conjugate principal ideals $\mathfrak{a}=\left(\frac{1+x\sqrt{-A}}{2}\right)$ and $\overline{\mathfrak{a}}$. We have
\[
\mathfrak{a}\overline{\mathfrak{a}}=N_{K/\Q}\left(\frac{1+x\sqrt{-A}}{2}\right)=\frac{1+Ax^2}{4}=\ell^n,
\]
and $\mathfrak{a}+\overline{\mathfrak{a}}=\mathcal{O}_K$, so we can assume without loss of generality that $\mathfrak{a}=\mathfrak{p}_\ell^n$. 
Once again the class of $\mathfrak{p}_\ell$ in $\Cl(K)$ has order dividing $n$ and we want to prove this order is exactly $n$. Assume by contradiction it is not the case. Then there exists a prime divisor $q$ of $m$ and $u,v\in \mathbb{Z}$ with $u\equiv v \text{ mod } 2$ such that the following equality of ideals holds: 
\[
\left(\frac{u+v\sqrt{-A}}{2}\right)^q=\left(\frac{1+x\sqrt{-A}}{2}\right).
\]
If $q$ is odd we use the same argument as in the first part of the proof of Proposition~\ref{prop:prime_given_order_2}, using that $A\neq 1,3$ to rule out this case. The argument goes through unchanged up until the final part, when we reach the equality
\[
\left(\frac{1+v^2A}{4}\right)^{\frac{q-1}{2}}=4(1+v^2 A)-3
\]
with $v \in \mathbb{Z}$ odd. In Proposition~\ref{prop:prime_given_order_2} here we concluded by using the fact that $A\equiv 7 \text{ mod } 8$ in that setting. In the current setting however, the congruence $A \equiv 3 \text{ mod } 8$ does not yield any contradiction. Instead to conclude one can notice that after setting $z=\frac{1+v^2A}{4}$, the above equation becomes
\[
z^{\frac{q-1}{2}}=16z-3
\]
which does not have any integral solution.

Hence, we can assume that $q=2$.
Again using the fact that $\mathcal{O}_K^\times =\{\pm1\}$, we have the following equality of elements of $K$:
\[
\left(\frac{u+v\sqrt{-A}}{2}\right)^2=\pm\frac{1+x\sqrt{-A}}{2}
\]
where $u,v\in\Z$ are such that $u\equiv v \text{ mod } 2$. With the usual arguments we arrive at
\[
u^2=\frac{\pm2+\sqrt{4+4x^2A}}{2}=\pm1+\sqrt{1+x^2A}=\pm1+2\ell^{m/2}.
\]
As we showed above this is impossible and the proof is concluded.
\end{proof}

One can directly verify that in $\mathbb{Q}(\sqrt{-971})=\mathbb{Q}(\sqrt{1-4\cdot 3^5})$ the prime $3$ splits into two prime ideals of order $5$ in the class group. This observation and the two previous propositions imply that for every $n \in \mathbb{Z}_{>0}$ and every odd prime $\ell$ there exists an imaginary quadratic field $K$ where $\ell$ splits into two prime ideals having order $n$ in $\mathrm{Cl}(\mathcal{O}_K)$.

We are now ready to prove Theorem \ref{thm:prime_given_order_intro}.

\begin{proof}[Proof of Theorem \ref{thm:prime_given_order_intro}]
By Propositions \ref{prop:prime_given_order_2}, \ref{prop:prime_given_order_l} and \ref{prop:prime_given_order_l_part_2} and the comments in between, it suffices to show that for every even $n \in \mathbb{N}$ the two sets
\[
E_1(n) : =\left \{\ell>2 \text{ prime}: \frac{\ell^{n/2}-1}{2}=j^2\text{ or }\frac{\ell^{n/2}+1}{2}=j^2 \text{ for some } j\in \mathbb{N} \right \},
\]
\[
E_2(n) := \left \{\ell>2 \text{ prime}: \ell^{n/2}=j^2+(j+1)^2 \text{ for some } j\in \mathbb{N} \right \}
\]
have empty intersection. Let $\ell\in E_1(n)\cap E_2(n)$. Then there exists $j\in\N$ such that $\ell^{n/2}=j^2+(j+1)^2$. We have
\[
\frac{\ell^{n/2}-1}{2}=j^2+j,
\]
and
\[
\frac{\ell^{n/2}+1}{2}=j^2+j+1.
\]
Suppose there exists $k\in\N$ such that $k^2=\frac{\ell^{n/2}-1}{2}$. Then $(k-j)(k+j)=j$, so $k+j$ divides $j$, which is impossible as $j,k>0$. Hence, since $\ell\in E_1(n)$, there exists $k\in\N$ such that $k^2=\frac{\ell^{n/2}+1}{2}$, that is, $(k-j)(k+j)=j+1$. Now we must have $k+j$ divides $j+1$ yielding $k=1$ and $j=0$. This is impossible, and the corollary follows.
\end{proof}

Theorem \ref{thm:inverse_volcano_intro} now follows by combining Proposition \ref{prop:abstract_crater_realizable} and Theorem \ref{thm:prime_given_order_intro}. The inverse volcano problem over $\mathbb{F}_p$ is solved.

\section{New questions}

In this final section, we discuss two follow-up projects: first, the inverse volcano problem over more general finite fields. Second, the question of solving the inverse volcano problem with algorithmic efficiency.

\subsection{The inverse volcano problem over \texorpdfstring{$\mathbb{F}_{p^s}$ with $s>1$}{Fpk with k>1}}

The inverse volcano problem over $\mathbb{F}_{p^s}$ with $s>1$ does not always have a solution. An example where $s=2$ is provided by the next proposition.

\begin{proposition}\label{non_inverse_over_F_p^2}
Let $V_0$ be a cycle of length $2$ and let $V$ be the abstract volcano induced by $(V_0,2,1)$, as in Figure~\ref{fig:volcano_counter_example}. For every prime $p \neq 2$, the volcano $V$ is not a connected component of $\mathcal{G}_2(\mathbb{F}_{p^2})$.
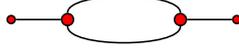
\begin{figure}[h!]
    \centering
    \scalebox{1.5}{\input{volcanoes/l=2/square/ex.tikz}} 
    \caption{The abstract volcano induced by $(V_0,2,1)$.}
    \label{fig:volcano_counter_example}
\end{figure}
\end{proposition}

\begin{proof}

Let us first notice that the ring of integers $\mathcal{O}_K$ of $K:=\mathbb{Q}(\sqrt{-15})$ is the only imaginary quadratic order where $2$ splits into two ideals having order $2$ in its class group (for example this can be seen using Lemma \ref{lem:bound_order_classgroup} proved below). Hence, if $V$ were an isogeny volcano in characteristic $p$, the elliptic curves corresponding to the vertices on its crater would have necessarily complex multiplication by $\mathcal{O}_K$ and $p$ would be split in it. Let $H$ be the ring class field relative to the order $\mathcal{O}:= \mathbb{Z}[2\sqrt{-15}]$. The natural exact sequence (see \cite[Chapter I, Proposition 12.9]{Neukirch})
\[
 1 \to \left( \mathcal{O}_K/4\mathcal{O}_K\right)^\times / \{\pm 1\} \to \mathrm{Cl}(\mathcal{O}) \to \mathrm{Cl}(\mathcal{O}_K) \to 1
\]
splits, so we have $\Gal(H/K) \cong (\mathbb{Z}/2\mathbb{Z})^2$. In particular, by \cite[Lemma 9.3]{Cox97}
\[
\Gal(H/\mathbb{Q}) \cong (\mathbb{Z}/2\mathbb{Z})^2 \rtimes \mathbb{Z}/2\mathbb{Z} \cong (\mathbb{Z}/2\mathbb{Z})^3
\]
and this implies that every prime $\mathfrak{p} \subseteq \mathcal{O}_H$ has residue degree bounded by $2$. Hence, for every prime $p \neq 2$ and split in $\mathcal{O}_K$ the field $\mathbb{F}_{p^2}$ contains the $j$-invariants of elliptic curves with complex multiplication by $\mathcal{O}$. The connected component of $\mathcal{G}_2(\mathbb{F}_{p^2})$ containing them must necessarily be a volcano of depth $\geq 2$ whose crater vertices correspond to all the elliptic curves with complex multiplication by $\mathcal{O}_K$. This proves the proposition.
\end{proof}

This fact triggers many questions: what are the obstructions to a possible solution? Are there infinitely many abstract volcanoes that are not connected components of ordinary isogeny graphs over $\mathbb{F}_{p^s}$, for fixed $s>1$? If there are only finitely many counter-examples, how many of them? This will be the subject of further research.

\subsection{Algorithmic corner}

For any given abstract volcano $V$ it is possible to explicitly find some ordinary isogeny graph that contains $V$ as a connected component, as our proofs show. However, computationally speaking, our theorems are far from optimal. For instance, by Theorem \ref{thm:prime_given_order_intro} the prime $3$ certainly splits into two prime ideals having order $5$ in the ideal class group of $\mathbb{Q}(\sqrt{-971})$. However, the same is true for the field $\mathbb{Q}(\sqrt{-47})$, whose discriminant is more than $20$ times smaller in absolute value. In this respect, the following easy lemma can be a useful tool for computational purposes.

\begin{lemma}\label{lem:bound_order_classgroup}
Let $\ell$ be a prime and $n\in \mathbb{N}$ a fixed integer. Suppose that $\mathcal{O}$ is an imaginary quadratic order of conductor coprime to $\ell$ where $\ell$ splits into two prime ideals having order $n$ in $\mathrm{Cl}(\mathcal{O})$. Then $|D(\mathcal{O})| \leq 4\ell^n-1$.
\end{lemma}

\begin{proof}
Set $K:=\mathrm{Frac}(\mathcal{O})$ and let $\mathfrak{L}$ and $\overline{\mathfrak{L}}$ be the two distinct prime ideals lying above $\ell$. Then by assumption there exists $\alpha \in \mathcal{O}$ such that $\mathfrak{L}^n=(\alpha)$. We have that $\alpha \not \in \mathbb{Z}$ since otherwise $n$ would be even and $\alpha=\pm \ell^{n/2}$, implying that $\mathfrak{L}=\overline{\mathfrak{L}}$. The lemma now follows from the fact that every element $\beta \in \mathcal{O} \setminus \mathbb{Z}$ satisfies $N_{K/\mathbb{Q}}(\beta) \geq \frac{1+|D(\mathcal{O})|}{4}$.
\end{proof}

Applying the above lemma sometimes leads to easier (\textit{i.e.} with associated imaginary quadratic field that has smaller discriminant) solutions than the ones provided by Theorem \ref{thm:prime_given_order_intro}. 

\begin{example}

Let $V_0$ be a cycle of length $6$ and let $V$ be the volcano induced by $(V_0,2,1)$, as in Figure~\ref{fig:volcano_ex}.

\begin{figure}[h!]
    \centering
    \scalebox{1.5}{\input{volcanoes/l=3/103/8.tikz}} 
    \caption{The abstract volcano $V$.}
    \label{fig:volcano_ex}
\end{figure}
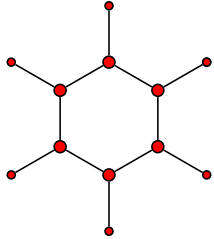

We claim that $V$ is realisable as a connected component of $\mathcal{G}_2(\mathbb{F}_{103})$ (more precisely, as a connected component of the $8$-cordillera in this graph). Indeed, using Lemma \ref{lem:bound_order_classgroup} one finds that $K=\mathbb{Q}(\sqrt{-87})$ is an imaginary quadratic field where $\ell=2$ splits into two ideals having order $6$ in $\mathrm{Cl}(\mathcal{O}_K)$. The prime $p=103$ splits completely in the ring class field of $\mathbb{Z}[\sqrt{-87}]$ but does not split completely in the ring class field of  $\mathbb{Z}[2\sqrt{-87}]$. We have
\[
4\cdot 103 = 8^2+87\cdot 2^2
\]
and the results explained in this paper now imply the result.
\end{example}

We encourage the reader to work out that the three abstract volcanoes appearing in Figure \ref{fig:volcanoes} are connected components, respectively, of the isogeny graphs $\mathcal{G}_2(\mathbb{F}_{1009})$, $\mathcal{G}_3(\mathbb{F}_{1303})$ and $\mathcal{G}_3(\mathbb{F}_{997})$.

\newpage

\appendix

\section{Full study for \texorpdfstring{$p=1009$ and $\ell=3$}{p=1009 and l=3}}
\label{app:1009}
In this section, we present the $3$-isogeny graph $\G_{3}(\F_{1009})$, and count the number of vertices in each connected components. This count is essentially another take on the Hurwitz class number formula (see~\cite{Cox97}, formula (14.21))
\[
p = \frac{1}{2}\sum_{0<|a|<2\sqrt{p}}H(a^2-4p).
\]
From Lemma~\ref{lem:cordillera_empty}, we expect traces of maximal absolute value $\lfloor2\sqrt{1009}\rfloor=63$.

Note that $1009\equiv1$ mod $12$, therefore $j=0$ and $j=719$ are $j$-invariants of ordinary elliptic curves (where $1728\equiv 719$ mod $1009$).

\subsection{Supersingulars}

We can check that the set of supersingular $j$-invariants is $$\{149,155,157,529,602,605,838,890,897,905\},$$ of cardinality $10$.

\subsection{More automorphisms than usual: \texorpdfstring{$j=0$}{j=0}}
As $1009\equiv1$ mod $3$, $0$ is ordinary, and we can compute the three traces in its cordillera: $(t_1,t_2,t_3)=(19,43,62)$. The orders $\Z[\pi_{t_i}]$ for $i\in\{t_1,t_2,t_3\}$ have pairwise coprime conductors $f_1,f_2,f_3$. If two of the conductors had a common prime factor $q$, then there would exist a $q$-isogeny connecting curves with different traces, which is impossible if neither of the curves have $j$-invariant $0$ or $1728$. We obtain $f_1=5\cdot7$, $f_2=3^3$ and $f_3=2^3$. We deduce that $0$ must connect to points corresponding to curves of trace $t_2=43$, as $\ell|f_2$. We expect volcanoes in the cordillera to have depth $0$ if the trace of their curves is $t_1$ or $t_3$, and depth $3$ if it is $t_2$ or the curve has $j$-invariant $0$. How $0$ connects to the rest of the cordillera has been described in Proposition~\ref{prop:zero}. This behaviour can be observed in Figure~\ref{fig:1009:0} with vertex $j=0$ pictured in blue. 

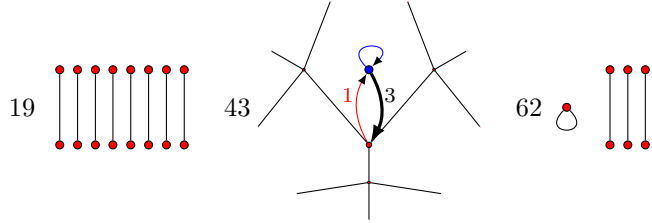
\begin{figure}[h]
    \centering
    \scalebox{1}{\input{volcanoes/l=3/1009/19.tikz}}
    \scalebox{1}{\input{volcanoes/l=3/1009/43.tikz}} 
    \scalebox{1}{\input{volcanoes/l=3/1009/62.tikz}}
    \caption{The $19/43/62$-cordillera in $\G_3(\F_{1009})$.}
    \label{fig:1009:0}
\end{figure}

Using Lemma~\ref{count}, the number of vertices in the volcano containing $0$ is given by \[1+\frac{2\cdot 1}{6}\times\left(3-\left(\frac{D(\mathcal{O}_K)}{3}\right)\right)\frac{3^{3}-1}{3-1}=1+\frac{27-1}{3-1}=\color{red}14\color{black}\] as $K=\Q(\sqrt{-3})$.  We can compute the number of vertices in the cordilleras associated with $K$ that are not connected to $j=0$, belt by belt. The conductors of $\Z[\pi_{t_i}]$ stripped of multiples of $3$ are $5\cdot7$ for $t_1=19$, $1$ for $t_2=43$ and $2^3$ for $t_3=62$, therefore the number of vertices for trace $19$ is \[h(5^2\cdot(-3))+h(7^2\cdot(-3))+h(35^2\cdot(-3))=2+2+12=\color{red}16\color{black}\] and \[h(2^2\cdot(-3))+h(4^2\cdot(-3))+h(8^2\cdot(-3))=1+2+4=\color{red}7\color{black}\] for trace $62$. There are no additional belts for trace $43$ because the $f_2$ is a power of $\ell$.

\subsection{More automorphisms than usual: \texorpdfstring{$j=1728$}{j=1728}}

The same study can be done for $1728$ (or rather $719$ as we are in $\F_{1009}$), $1728\equiv1$ mod $4$ so we know that two traces are associated with $719$ in this case: $(t_4,t_5)=(30,56)$. The conductors are coprime $f_4=2^2\cdot7$ and $f_5=3\cdot5$, hence according to Proposition~\ref{prop:1728}, $719$ in blue connects to $\frac{3+1}{2}$ curves of trace $t_5$. Figure~\ref{fig:1009:1728} shows the non-regular subgraph for $1728$.

\begin{figure}[h]
    \centering
    \scalebox{1}{\input{volcanoes/l=3/1009/30.tikz}}
    \scalebox{1}{\input{volcanoes/l=3/1009/56.tikz}} 
    \caption{The $30$- and $56$-cordilleras in $\G_3(\F_{1009})$.}
    \label{fig:1009:1728}
\end{figure}
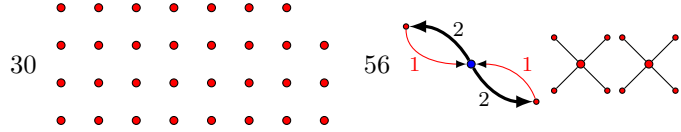

As $K=\Q(\sqrt{-4})$, we can compute the number of points in the belt where the vertex $1728$ attaches: 
\[
1+\frac{2\cdot1}{4}\left(3-\left(\frac{D(\mathcal{O}_K)}{3}\right)\right)\frac{3^1-1}{3-1}=1+\frac{4}{2}=\color{red}3\color{black}.
\]
Counting the other vertices gives
\[
\sum_{1<m|28}h(-4m^2)=\color{red}31\color{black}
\]
for trace $30$ and 
\[
h(-4\cdot5^2)\left(1+(3-(-1))\left(\frac{3^1-1}{3-1}\right)\right)=\color{red}10\color{black}
\]
for trace $56$. 

\subsection{Regular case}

In this section we present an extensive classification of all volcanoes appearing as regular connected components in $\G_3(\F_{1009})$, and we count their number of vertices. We go through the graph belt by belt and expose the results in four tables.

The first table, Figure~\ref{graph:1009:1} lists volcanoes that are single vertices. We can note that they arise in cordilleras whose traces are multiples of $3$. Indeed if the trace is a multiple of $3$, $v$ in the equation $4p=t^2-D(\mathcal{O}_K)v^2$ can not be divisible by three, and the Kronecker symbol $(D(\mathcal{O}_K)/3)$ is always $-1$, meaning that $3$ is inert in $K$. This is also the case for the $30$-cordillera, already discussed in the non-regular case (see Figure~\ref{fig:1009:1728}).

Secondly in Figure~\ref{graph:1009:2} we look at cordilleras with volcanoes that are just two vertices connected by a single edge. These are frequent and can be listed in the same way. Another example is the $19$-cordillera shown in Figure~\ref{fig:1009:0}, with $16$ vertices and $8$ volcanoes.

\begin{figure}[H]
\centering
\begin{tabular}{c|c|c|c|l}
    Cordillera & Associated & Belt & Vertex & Volcano\\
    (Trace) & field & (Order) & count & shape
    \\\hline
     $3$ & $\Q(\sqrt{-4027})$ & $\Z[\sqrt{-4027}]$ & $h(-4027)=\color{red}9\color{black}$ & \scalebox{1}{\input{volcanoes/point.tikz}}$\times 9$\\\hline
     $6$ & $\Q(\sqrt{-40})$ & $\Z[\sqrt{-40}]$ & $h(-40)=\color{red}2\color{black}$ & \scalebox{1}{\input{volcanoes/point.tikz}}$\times 2$\\  &   & $\Z[2\sqrt{-40}]$ & $h(-160)=\color{red}4\color{black}$ & \scalebox{1}{\input{volcanoes/point.tikz}}$\times 4$\\  &   & $\Z[5\sqrt{-40}]$ & $h(-1000)=\color{red}10\color{black}$ & \scalebox{1}{\input{volcanoes/point.tikz}}$\times 10$\\
      &   & $\Z[10\sqrt{-40}]$ & $h(-4000)=\color{red}20\color{black}$ & \scalebox{1}{\input{volcanoes/point.tikz}}$\times 20$\\\hline
     $9$ & $\Q(\sqrt{-3955})$ & $\Z[\sqrt{-3955}]$ & $h(-3955)=\color{red}12\color{black}$ & \scalebox{1}{\input{volcanoes/point.tikz}}$\times 12$\\\hline
     $12$ & $\Q(\sqrt{-3892})$ & $\Z[\sqrt{-3892}]$ & $h(-3892)=\color{red}12\color{black}$ & \scalebox{1}{\input{volcanoes/point.tikz}}$\times 12$\\\hline
     $15$ & $\Q(\sqrt{-3811})$ & $\Z[\sqrt{-3811}]$ & $h(-3811)=\color{red}10\color{black}$ & \scalebox{1}{\input{volcanoes/point.tikz}}$\times 10$\\\hline
     $18$ & $\Q(\sqrt{-232})$ & $\Z[\sqrt{-232}]$ & $h(-232)=\color{red}2\color{black}$ & \scalebox{1}{\input{volcanoes/point.tikz}}$\times 2$\\
      &   & $\Z[2\sqrt{-232}]$ & $h(-928)=\color{red}4\color{black}$ & \scalebox{1}{\input{volcanoes/point.tikz}}$\times 4$\\
      &   & $\Z[4\sqrt{-232}]$ & $h(-3712)=\color{red}8\color{black}$ & \scalebox{1}{\input{volcanoes/point.tikz}}$\times 8$\\\hline
     $21$ & $\Q(\sqrt{-3595})$ & $\Z[\sqrt{-3595}]$ & $h(-3595)=\color{red}8\color{black}$ & \scalebox{1}{\input{volcanoes/point.tikz}}$\times 8$\\\hline
     $24$ & $\Q(\sqrt{-3460})$ & $\Z[\sqrt{-3460}]$ & $h(-3460)=\color{red}16\color{black}$ & \scalebox{1}{\input{volcanoes/point.tikz}}$\times 16$\\\hline
     $27$ & $\Q(\sqrt{-3307})$ & $\Z[\sqrt{-3307}]$ & $h(-3307)=\color{red}9\color{black}$ & \scalebox{1}{\input{volcanoes/point.tikz}}$\times 9$\\\hline
     $33$ & $\Q(\sqrt{-2947})$ & $\Z[\sqrt{-2947}]$ & $h(-2947)=\color{red}8\color{black}$ & \scalebox{1}{\input{volcanoes/point.tikz}}$\times 8$\\\hline
     $36$ & $\Q(\sqrt{-2740})$ & $\Z[\sqrt{-2740}]$ & $h(-2740)=\color{red}12\color{black}$ & \scalebox{1}{\input{volcanoes/point.tikz}}$\times 12$\\\hline
     $39$ & $\Q(\sqrt{-2515})$ & $\Z[\sqrt{-2515}]$ & $h(-2515)=\color{red}6\color{black}$ & \scalebox{1}{\input{volcanoes/point.tikz}}$\times 6$\\\hline
     $42$ & $\Q(\sqrt{-568})$ & $\Z[\sqrt{-568}]$ & $h(-568)=\color{red}4\color{black}$ & \scalebox{1}{\input{volcanoes/point.tikz}}$\times 4$\\ & & $\Z[2\sqrt{-568}]$ & $h(-2272)=\color{red}8\color{black}$ & \scalebox{1}{\input{volcanoes/point.tikz}}$\times 8$\\\hline
     $45$ & $\Q(\sqrt{-2011})$ & $\Z[\sqrt{-2011}]$ & $h(-2011)=\color{red}7\color{black}$ & \scalebox{1}{\input{volcanoes/point.tikz}}$\times 7$\\\hline
     $48$ & $\Q(\sqrt{-1732})$ & $\Z[\sqrt{-1732}]$ & $h(-1732)=\color{red}12\color{black}$ & \scalebox{1}{\input{volcanoes/point.tikz}}$\times 12$\\\hline
     $51$ & $\Q(\sqrt{-1435})$ & $\Z[\sqrt{-1435}]$ & $h(-1435)=\color{red}4\color{black}$ & \scalebox{1}{\input{volcanoes/point.tikz}}$\times 4$\\\hline
     $54$ & $\Q(\sqrt{-280})$ & $\Z[\sqrt{-280}]$ & $h(-280)=\color{red}4\color{black}$ & \scalebox{1}{\input{volcanoes/point.tikz}}$\times 4$\\ &  & $\Z[2\sqrt{-280}]$ & $h(-1120)=\color{red}8\color{black}$ & \scalebox{1}{\input{volcanoes/point.tikz}}$\times 8$\\\hline
     $57$ & $\Q(\sqrt{-787})$ & $\Z[\sqrt{-787}]$ & $h(-787)=\color{red}5\color{black}$ & \scalebox{1}{\input{volcanoes/point.tikz}}$\times 5$\\\hline
     $60$ & $\Q(\sqrt{-436})$ & $\Z[\sqrt{-436}]$ & $h(-436)=\color{red}6\color{black}$ & \scalebox{1}{\input{volcanoes/point.tikz}}$\times 6$\\\hline
     $63$ & $\Q(\sqrt{-67})$ & $\Z[\sqrt{-67}]$ & $h(-67)=\color{red}1\color{black}$ & \scalebox{1}{\input{volcanoes/point.tikz}} 
\end{tabular}
\caption{Distribution of single-point volcanoes in $\G_3(\F_{1009})$.}
\label{graph:1009:1}
\end{figure}

\begin{figure}[H]
\centering
\begin{tabular}{c|c|c|c|l}
    Cordillera & Associated & Belt & Vertex & Volcano\\
    (Trace) & field & (Order) & count & shape\\\hline
     $1$ & $\Q(\sqrt{-4035})$ & $\Z[\sqrt{-4035}]$ & $h(-4035)=\color{red}12\color{black}$ & \scalebox{1}{\input{volcanoes/line.tikz}}$\times 6$\\\hline
     $4$ & $\Q(\sqrt{-4020})$ & $\Z[\sqrt{-4020}]$ & $h(-4020)=\color{red}16\color{black}$ & \scalebox{1}{\input{volcanoes/line.tikz}}$\times 8$\\\hline
     $5$ & $\Q(\sqrt{-4011})$ & $\Z[\sqrt{-4011}]$ & $h(-4011)=\color{red}20\color{black}$ & \scalebox{1}{\input{volcanoes/line.tikz}}$\times 10$\\\hline
     $8$ & $\Q(\sqrt{-3972})$ & $\Z[\sqrt{-3972}]$ & $h(-3972)=\color{red}12\color{black}$ & \scalebox{1}{\input{volcanoes/line.tikz}}$\times 6$\\\hline
     $10$ & $\Q(\sqrt{-984})$ & $\Z[\sqrt{-984}]$ & $h(-984)=\color{red}12\color{black}$ & \scalebox{1}{\input{volcanoes/line.tikz}}$\times 6$\\ & & $\Z[2\sqrt{-984}]$ & $h(-3936)=\color{red}24\color{black}$ & \scalebox{1}{\input{volcanoes/line.tikz}}$\times 12$\\\hline
     $13$ & $\Q(\sqrt{-3867})$ & $\Z[\sqrt{-3867}]$ & $h(-3867)=\color{red}14\color{black}$ & \scalebox{1}{\input{volcanoes/line.tikz}}$\times 7$\\\hline
     $14$ & $\Q(\sqrt{-15})$ & $\Z[\sqrt{-15}]$ & $h(-15)=\color{red}2\color{black}$ & \scalebox{1}{\input{volcanoes/line.tikz}} \\ & & $\Z[2\sqrt{-15}]$ & $h(-60)=\color{red}2\color{black}$ & \scalebox{1}{\input{volcanoes/line.tikz}} \\ & & $\Z[4\sqrt{-15}]$ & $h(-240)=\color{red}4\color{black}$ & \scalebox{1}{\input{volcanoes/line.tikz}}$\times 2$\\ & & $\Z[8\sqrt{-15}]$ & $h(-960)=\color{red}8\color{black}$ & \scalebox{1}{\input{volcanoes/line.tikz}}$\times 4$\\ & & $\Z[16\sqrt{-15}]$ & $h(-3840)=\color{red}16\color{black}$ & \scalebox{1}{\input{volcanoes/line.tikz}}$\times 8$\\\hline
     $17$ & $\Q(\sqrt{-3747})$ & $\Z[\sqrt{-3747}]$ & $h(-3747)=\color{red}12\color{black}$ & \scalebox{1}{\input{volcanoes/line.tikz}}$\times 6$\\\hline
     $22$ & $\Q(\sqrt{-888})$ & $\Z[\sqrt{-888}]$ & $h(-888)=\color{red}12\color{black}$ & \scalebox{1}{\input{volcanoes/line.tikz}}$\times 6$\\& & $\Z[2\sqrt{-888}]$ & $h(-3552)=\color{red}24\color{black}$ & \scalebox{1}{\input{volcanoes/line.tikz}}$\times 12$\\\hline
     $23$ & $\Q(\sqrt{-3507})$ & $\Z[\sqrt{-3507}]$ & $h(-3507)=\color{red}8\color{black}$ & \scalebox{1}{\input{volcanoes/line.tikz}}$\times 4$\\\hline
     $26$ & $\Q(\sqrt{-840})$ & $\Z[\sqrt{-840}]$ & $h(-840)=\color{red}8\color{black}$ & \scalebox{1}{\input{volcanoes/line.tikz}}$\times 4$\\ & & $\Z[2\sqrt{-840}]$ & $h(-3360)=\color{red}16\color{black}$ & \scalebox{1}{\input{volcanoes/line.tikz}}$\times 8$\\\hline
     $28$ & $\Q(\sqrt{-3252})$ & $\Z[\sqrt{-3252}]$ & $h(-3252)=\color{red}12\color{black}$ & \scalebox{1}{\input{volcanoes/line.tikz}}$\times 6$\\\hline
     $31$ & $\Q(\sqrt{-123})$ & $\Z[\sqrt{-123}]$ & $h(-123)=\color{red}2\color{black}$ & \scalebox{1}{\input{volcanoes/line.tikz}} \\ & & $\Z[5\sqrt{-123}]$ & $h(-3075)=\color{red}12\color{black}$ & \scalebox{1}{\input{volcanoes/line.tikz}}$\times 6$\\\hline
     $32$ & $\Q(\sqrt{-3012})$ & $\Z[\sqrt{-3012}]$ & $h(-3012)=\color{red}12\color{black}$ & \scalebox{1}{\input{volcanoes/line.tikz}}$\times 6$\\\hline
     $35$ & $\Q(\sqrt{-2811})$ & $\Z[\sqrt{-2811}]$ & $h(-2811)=\color{red}16\color{black}$ & \scalebox{1}{\input{volcanoes/line.tikz}}$\times 8$\\\hline
     $37$ & $\Q(\sqrt{-2667})$ & $\Z[\sqrt{-2667}]$ & $h(-2667)=\color{red}8\color{black}$ & \scalebox{1}{\input{volcanoes/line.tikz}}$\times 4$\\\hline
     $40$ & $\Q(\sqrt{-2436})$ & $\Z[\sqrt{-2436}]$ & $h(-2436)=\color{red}16\color{black}$ & \scalebox{1}{\input{volcanoes/line.tikz}}$\times 8$\\\hline
     $41$ & $\Q(\sqrt{-2355})$ & $\Z[\sqrt{-2355}]$ & $h(-2355)=\color{red}12\color{black}$ & \scalebox{1}{\input{volcanoes/line.tikz}}$\times 6$\\\hline
     $44$ & $\Q(\sqrt{-84})$ & $\Z[\sqrt{-84}]$ & $h(-84)=\color{red}4\color{black}$ & \scalebox{1}{\input{volcanoes/line.tikz}}$\times 2$\\ & & $\Z[5\sqrt{-84}]$ & $h(-2100)=\color{red}16\color{black}$ & \scalebox{1}{\input{volcanoes/line.tikz}}$\times 8$\\\hline
     $46$ & $\Q(\sqrt{-120})$ & $\Z[\sqrt{-120}]$ & $h(-120)=\color{red}4\color{black}$ & \scalebox{1}{\input{volcanoes/line.tikz}}$\times 2$\\ & & $\Z[2\sqrt{-120}]$ & $h(-480)=\color{red}8\color{black}$ & \scalebox{1}{\input{volcanoes/line.tikz}}$\times 4$\\ & & $\Z[4\sqrt{-120}]$ & $h(-1920)=\color{red}16\color{black}$ & \scalebox{1}{\input{volcanoes/line.tikz}}$\times 8$\\\hline
     $49$ & $\Q(\sqrt{-1635})$ & $\Z[\sqrt{-1635}]$ & $h(-1635)=\color{red}8\color{black}$ & \scalebox{1}{\input{volcanoes/line.tikz}}$\times 4$\\\hline
     $50$ & $\Q(\sqrt{-24})$ & $\Z[\sqrt{-24}]$ & $h(-24)=\color{red}2\color{black}$ & \scalebox{1}{\input{volcanoes/line.tikz}} \\
     & & $\Z[2\sqrt{-24}]$ & $h(-96)=\color{red}4\color{black}$ & \scalebox{1}{\input{volcanoes/line.tikz}}$\times 2$\\  & & $\Z[4\sqrt{-24}]$ & $h(-384)=\color{red}8\color{black}$ & \scalebox{1}{\input{volcanoes/line.tikz}}$\times 4$\\  & & $\Z[8\sqrt{-24}]$ & $h(-1536)=\color{red}16\color{black}$ & \scalebox{1}{\input{volcanoes/line.tikz}}$\times 8$\\\hline
     $53$ & $\Q(\sqrt{-1227})$ & $\Z[\sqrt{-1227}]$ & $h(-1227)=\color{red}4\color{black}$ & \scalebox{1}{\input{volcanoes/line.tikz}}$\times 2$\\\hline
     $55$ & $\Q(\sqrt{-1011})$ & $\Z[\sqrt{-1011}]$ & $h(-1011)=\color{red}12\color{black}$ & \scalebox{1}{\input{volcanoes/line.tikz}}$\times 6$\\\hline
     $58$ & $\Q(\sqrt{-168})$ & $\Z[\sqrt{-168}]$ & $h(-168)=\color{red}4\color{black}$ & \scalebox{1}{\input{volcanoes/line.tikz}}$\times 2$\\ & & $\Z[2\sqrt{-168}]$ & $h(-672)=\color{red}8\color{black}$ & \scalebox{1}{\input{volcanoes/line.tikz}}$\times 4$\\\hline
     $59$ & $\Q(\sqrt{-555})$ & $\Z[\sqrt{-555}]$ & $h(-555)=\color{red}4\color{black}$ & \scalebox{1}{\input{volcanoes/line.tikz}}$\times 2$
     \end{tabular}
\caption{Distribution of double-point volcanoes in $\G_3(\F_{1009})$.}
\label{graph:1009:2}
\end{figure}

The next cordilleras have volcanoes in an X-shape. They are pictured in Figure~\ref{graph:1009:3}.

\begin{figure}[H]
\centering
\begin{tabular}{c|c|c|c|l}
    Cordillera & Associated & Belt & Vertex & Volcano\\
    (Trace) & field & (Order) & count & shape\\\hline
     $2$ & $\Q(\sqrt{-7})$ & $\Z[\sqrt{-7}]$ & $\left(1+(3-(-7/3))\left(\frac{3^1-1}{3-1}\right)\right)h(-7) =\color{red}5\color{black}$ & 
     \scalebox{1}{\input{volcanoes/cross.tikz}} \\ & & $\Z[2\sqrt{-7}]$ & $\left(1+(3-(-28/3))\left(\frac{3^1-1}{3-1}\right)\right)h(-28)=\color{red}5\color{black}$ & 
     \scalebox{1}{\input{volcanoes/cross.tikz}} \\ & & $\Z[4\sqrt{-7}]$ & $\left(1+(3-(-112/3))\left(\frac{3^1-1}{3-1}\right)\right)h(-112)=\color{red}10\color{black}$ & 
     \scalebox{1}{\input{volcanoes/cross.tikz}}$\times 2$\\     & & $\Z[8\sqrt{-7}]$ & $\left(1+(3-(-448/3))\left(\frac{3^1-1}{3-1}\right)\right)h(-448)=\color{red}20\color{black}$ & 
     \scalebox{1}{\input{volcanoes/cross.tikz}}$\times 4$\\\hline
     $25$ & $\Q(\sqrt{-379})$ & $\Z[\sqrt{-379}]$ & $\left(1+(3-(-379/3))\left(\frac{3^1-1}{3-1}\right)\right)h(-379)=\color{red}15\color{black}$ & \scalebox{1}{\input{volcanoes/cross.tikz}}$\times 3$\\\hline
     $29$ & $\Q(\sqrt{-355})$ & $\Z[\sqrt{-355}]$ & $\left(1+(3-(-355/3))\left(\frac{3^1-1}{3-1}\right)\right)h(-355)=\color{red}20\color{black}$ & \scalebox{1}{\input{volcanoes/cross.tikz}}$\times 4$\\\hline
     $52$ & $\Q(\sqrt{-148})$ & $\Z[\sqrt{-148}]$ & $\left(1+(3-(-148/3))\left(\frac{3^1-1}{3-1}\right)\right)h(-148)=\color{red}10\color{black}$ & \scalebox{1}{\input{volcanoes/cross.tikz}}$\times 2$
\end{tabular}
\caption{Distribution of X-shaped volcanoes in $\G_3(\F_{1009})$.}
\label{graph:1009:3}
\end{figure}

Finally, the remaining vertices form larger more diverse volcanoes that are described in Figure~\ref{graph:1009:4}.

\begin{figure}[H]
\centering
\begin{tabular}{c|c|c|c|l}
    Cordillera & Associated & Belt & Vertex & Volcano\\
    (Trace) & field & (Order) & count & shape\\\hline
     $7$ & $\Q(\sqrt{-443})$ & $\Z[\sqrt{-443}]$ & $3^1\cdot h(-443)=\color{red}15\color{black}$ & 
     \scalebox{0.5}{\input{volcanoes/l=3/1009/7.tikz}} \\\hline
     $11$ & $\Q(\sqrt{-435})$ & $\Z[\sqrt{-435}]$ & $(1+3)h(-435)=\color{red}16\color{black}$ & 
     \scalebox{0.5}{\input{volcanoes/l=3/1009/11.tikz}}$\times 2$\\\hline
     $16$ & $\Q(\sqrt{-420})$ & $\Z[\sqrt{-420}]$ & $(1+3)h(-420)=\color{red}32\color{black}$ & 
     \scalebox{0.5}{\input{volcanoes/l=3/1009/16.tikz}}$\times 4$\\\hline
     $20$ & $\Q(\sqrt{-404})$ & $\Z[\sqrt{-404}]$ & $3^1\cdot h(-404)=\color{red}42\color{black}$ & 
     \scalebox{0.5}{\input{volcanoes/l=3/1009/20.tikz}} \\\hline
     $34$ & $\Q(\sqrt{-20})$ & $\Z[\sqrt{-20}]$ & $3^1\cdot  h(-20)=\color{red}6\color{black}$ & 
     \scalebox{0.5}{\input{volcanoes/l=3/1009/34.tikz}} \\ & & $\Z[2\sqrt{-20}]$ & $3^1\cdot h(-80)=\color{red}12\color{black}$ & 
     \scalebox{0.5}{\input{volcanoes/l=3/1009/34b.tikz}} \\ & & $\Z[4\sqrt{-20}]$ & $3^1\cdot h(-320)=\color{red}24\color{black}$ & 
     \scalebox{0.5}{\input{volcanoes/l=3/1009/34b.tikz}}$\times 2$\\\hline
     $38$ & $\Q(\sqrt{-8})$ & $\Z[\sqrt{-8}]$ & $3^2\cdot h(-8)=\color{red}9\color{black}$ & 
     \scalebox{0.5}{\input{volcanoes/l=3/1009/38a.tikz}} \\ & & $\Z[2\sqrt{-8}]$ & $3^2\cdot h(-32)=\color{red}18\color{black}$ & 
     \scalebox{0.5}{\input{volcanoes/l=3/1009/38b.tikz}} \\\hline
     $47$ & $\Q(\sqrt{-203})$ & $\Z[\sqrt{-203}]$ & $3^1\cdot h(-203)=\color{red}12\color{black}$ & 
     \scalebox{0.5}{\input{volcanoes/l=3/1009/47.tikz}} \\\hline
     $61$ & $\Q(\sqrt{-35})$ & $\Z[\sqrt{-35}]$ & $3^1\cdot h(-35)=\color{red}6\color{black}$ & 
     \scalebox{0.5}{\input{volcanoes/l=3/1009/61.tikz}} 
\end{tabular}
\caption{Distribution of larger volcanoes in $\G_3(\F_{1009})$.}
\label{graph:1009:4}
\end{figure}

This completes our extensive description of $\G_3(\F_{1009})$. We may now check that all vertices are accounted for by adding all red numbers. We obtain a total of 
\begin{itemize}
    \item $10$ supersingular $j$-invariants;
    \item $14+16+7=37$ non-regular vertices for $0$;
    \item $3+31+10=44$ non-regular vertices for $1728$;
    \item $211$ solo regular vertices;
    \item $430$ vertices in regular duos;
    \item $85$ vertices in X-shaped volcanoes; 
    \item and
    $192$ vertices in larger volcanoes.
\end{itemize}

In total $10+37+44+211+430+85+192=1009=p$ vertices!

\section{The full isogeny graph for \texorpdfstring{$p=1009$ and $\ell=3$}{p=1009 and l=3}}
\label{app:constellation}
\begin{figure}[H]
    \centering
    \fbox{\scalebox{0.45}{\input{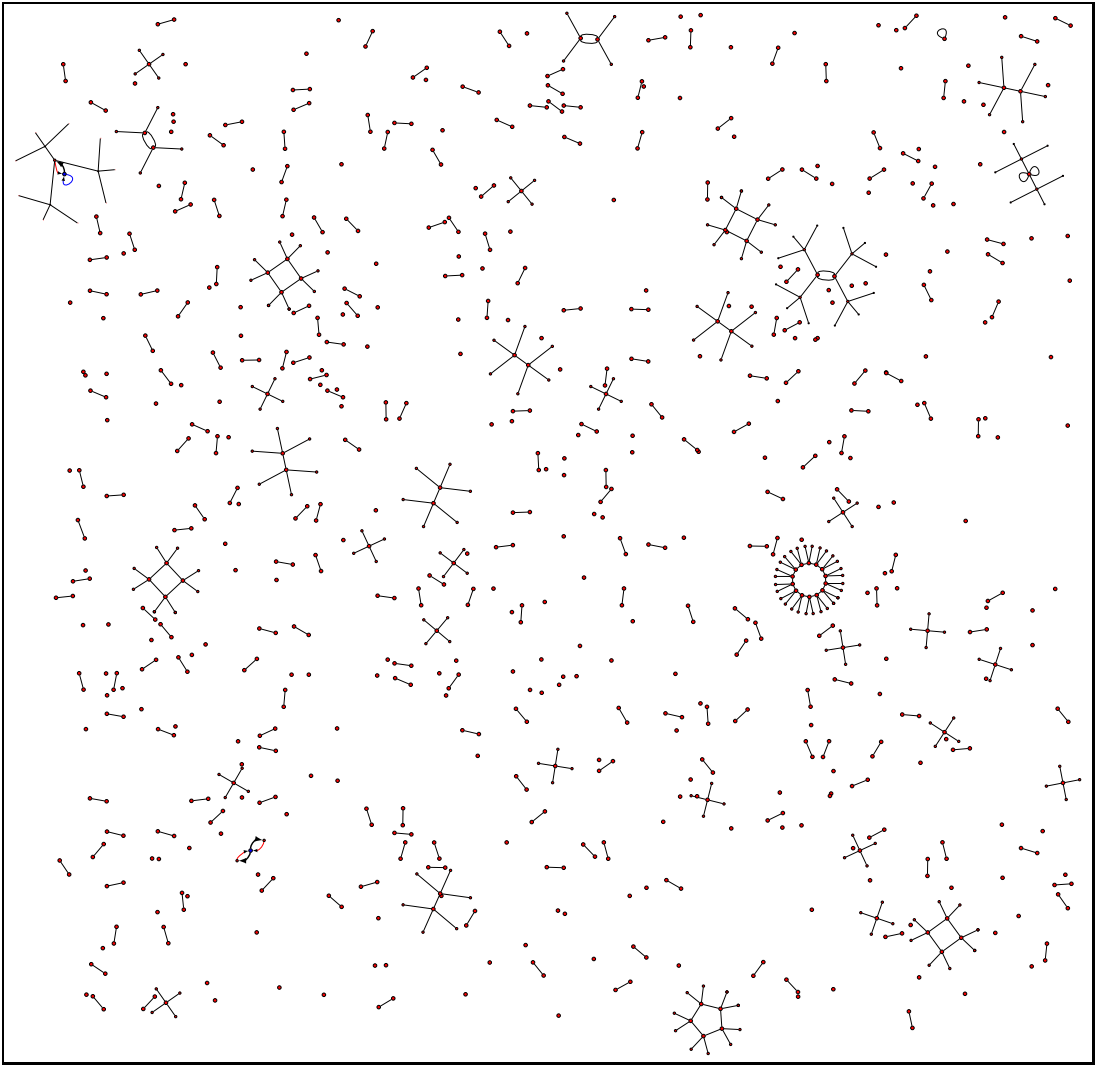}}}
    \caption{The volcano park $\mathcal{G}_{3}(1009)$}
    \label{fig:const}
\end{figure}

\begin{landscape}
\section{Examples of cordilleras with different values of \texorpdfstring{$\ell$}{l}}
\label{app:tables}
\begin{figure}[!htb]
\minipage{0.3\textwidth}
\begin{center}
\begin{tabular}{c|c|c|c}
    $m$ & $h(\Or_0)$ & $\ord_m$ & $\mathcal{B}_{t,m}$ \\\hline
     $1$ & $12$ & $2$ & \scalebox{0.5}{\input{volcanoes/l=2/7321/99.tikz}} $\times 6$ \\ \hline
     $3$ & $6$ & $2$ & \scalebox{0.5}{\input{volcanoes/l=2/7321/99.tikz}}$\times 3$\\\hline
     $5$ & $2$ & $2$ & \scalebox{0.5}{\input{volcanoes/l=2/7321/99.tikz}} \\\hline
     $3\cdot5$ & $1$ & $1$ & \scalebox{1}{\input{volcanoes/l=2/7321/679.tikz}}  
\end{tabular}
\end{center}
\caption{The $22$-cordillera in $\G_2(\F_{7321})$.}
\label{tab:2}\endminipage\hspace{50pt}
\minipage{0.3\textwidth}
\begin{center}
\begin{tabular}{c|c|c|c}
    $m$ & $h(\Or_0)$ & $\ord_m$ & $\mathcal{B}_{t,m}$ \\\hline
     $1$ & $24$ & $12$ & \scalebox{0.75}{\input{volcanoes/l=3/7321/5352.tikz}}$\times 2$ \\ \hline
     $2$ & $12$ & $6$ & \scalebox{0.75}{\input{volcanoes/l=3/7321/3660.tikz}}$\times 2$\\\hline
     $2^2$ & $6$ & $6$ & \scalebox{0.75}{\input{volcanoes/l=3/7321/3660.tikz}} \\\hline
     $5$ & $4$ & $4$ & \scalebox{0.75}{\input{volcanoes/l=3/7321/2525.tikz}} \\\hline
     $2\cdot5$ & $2$ & $2$ & \scalebox{0.75}{\input{volcanoes/l=3/7321/1490.tikz}} \\\hline
     $2^2\cdot5$ & $1$ & $1$ & \scalebox{1}{\input{volcanoes/l=3/7321/679.tikz}} 
\end{tabular}
\end{center}
\caption{The $22$-cordillera in $\G_3(\F_{7321})$.}
\label{tab:3}
\endminipage\hspace{50pt}
\minipage{0.3\textwidth}%
\begin{center}
\begin{tabular}{c|c|c|c}
    $m$ & $h(\Or_0)$ & $\ord_m$ & $\mathcal{B}_{t,m}$ \\\hline
     $1$ & $8$ & $1$ & \scalebox{1}{\input{volcanoes/l=5/7321/99.tikz}} $\times 8$ \\ \hline
     $2$ & $4$ & $1$ & \scalebox{1}{\input{volcanoes/l=5/7321/99.tikz}} $\times 4$ \\\hline
     $2^2$ & $2$ & $1$ & \scalebox{1}{\input{volcanoes/l=5/7321/99.tikz}}$\times 2$\\\hline
     $3$ & $4$ & $1$ & \scalebox{1}{\input{volcanoes/l=5/7321/99.tikz}}$\times 4$\\\hline
     $2\cdot3$ & $2$ & $1$ & \scalebox{1}{\input{volcanoes/l=5/7321/99.tikz}}$\times 2$\\\hline
     $2^2\cdot3$ & $1$ & $1$ & \scalebox{1}{\input{volcanoes/l=5/7321/99.tikz}} 
\end{tabular}
\end{center}
\caption{The $22$-cordillera in $\G_5(\F_{7321})$.}
\label{tab:5}
\endminipage
\end{figure}
\end{landscape}



\bibliographystyle{plain}
\bibliography{EC}


\end{document}

%% file: sample.tikzstyles
\tikzstyle{const dot}=[fill=red, draw = black, shape=circle,inner sep=0pt,minimum size =3pt]
\tikzstyle{red dot}=[fill=red, draw=black, shape=circle,inner sep=0pt,minimum size=2pt]
\tikzstyle{red dot 0}=[fill=red, draw=black, shape=circle,inner sep=0pt,minimum size=0pt]
\tikzstyle{red dot 1}=[fill=red, draw=black, shape=circle,inner sep=0pt,minimum size=1pt]
\tikzstyle{red dot 2}=[fill=red, draw=black, shape=circle,inner sep=0pt,minimum size=2pt]
\tikzstyle{red dot 3}=[fill=red, draw=black, shape=circle,inner sep=0pt,minimum size=3pt]
\tikzstyle{blue dot 3}=[fill=blue, draw=black, shape=circle,inner sep=0pt,minimum size=3pt]
\tikzstyle{red dot big}=[fill=red, draw = black, shape=circle,inner sep=0pt,minimum size =10pt]
\tikzstyle{green dot}=[fill=green, draw=black, shape=circle]
\tikzstyle{blue dot}=[fill=blue, draw=black, shape=circle]
\tikzstyle{test}=[fill=black, draw=black]
\tikzstyle{test 1}=[fill=black, draw=black, shape=circle,inner sep=0pt,minimum size = 1pt]

\tikzstyle{blue edge}=[-, draw=blue]
\tikzstyle{black arr}=[->,>=latex]
\tikzstyle{blue arr}=[->, draw=blue,>=latex]
\tikzstyle{blue arr short3}=[->, draw=blue, shorten >=3pt,>=latex]
\tikzstyle{blue arr short1}=[->, draw=blue, shorten >=1pt,>=latex]
\tikzstyle{black arr short3}=[->, draw=black, shorten >=3pt,>=latex]
\tikzstyle{black arr short1}=[->, draw=black, shorten >=1pt,>=latex]
\tikzstyle{green arr}=[->, draw=green,>=latex]
\tikzstyle{red arr}=[->, draw=red,>=latex]

\tikzstyle{65}=[-, draw={rgb,255:red,0; green,0; blue,0}]
\tikzstyle{64}=[-, draw={rgb,255:red,0; green,0; blue,0}]
\tikzstyle{63}=[-, draw={rgb,255:red,0; green,0; blue,0}]
\tikzstyle{62}=[-, draw={rgb,255:red,0; green,0; blue,0}]
\tikzstyle{61}=[-, draw={rgb,255:red,0; green,0; blue,0}]
\tikzstyle{60}=[-, draw={rgb,255:red,0; green,0; blue,0}]
\tikzstyle{59}=[-, draw={rgb,255:red,0; green,0; blue,0}]
\tikzstyle{58}=[-, draw={rgb,255:red,0; green,0; blue,0}]
\tikzstyle{57}=[-, draw={rgb,255:red,0; green,0; blue,0}]
\tikzstyle{56}=[-, draw={rgb,255:red,0; green,0; blue,0}]
\tikzstyle{55}=[-, draw={rgb,255:red,0; green,0; blue,0}]
\tikzstyle{54}=[-, draw={rgb,255:red,0; green,0; blue,0}]
\tikzstyle{53}=[-, draw={rgb,255:red,0; green,0; blue,0}]
\tikzstyle{52}=[-, draw={rgb,255:red,0; green,0; blue,0}]
\tikzstyle{51}=[-, draw={rgb,255:red,0; green,0; blue,0}]
\tikzstyle{50}=[-, draw={rgb,255:red,0; green,0; blue,0}]
\tikzstyle{49}=[-, draw={rgb,255:red,0; green,0; blue,0}]
\tikzstyle{48}=[-, draw={rgb,255:red,0; green,0; blue,0}]
\tikzstyle{47}=[-, draw={rgb,255:red,0; green,0; blue,0}]
\tikzstyle{46}=[-, draw={rgb,255:red,0; green,0; blue,0}]
\tikzstyle{45}=[-, draw={rgb,255:red,0; green,0; blue,0}]
\tikzstyle{44}=[-, draw={rgb,255:red,0; green,0; blue,0}]
\tikzstyle{43}=[-, draw={rgb,255:red,0; green,0; blue,0}]
\tikzstyle{42}=[-, draw={rgb,255:red,0; green,0; blue,0}]
\tikzstyle{41}=[-, draw={rgb,255:red,0; green,0; blue,0}]
\tikzstyle{40}=[-, draw={rgb,255:red,0; green,0; blue,0}]
\tikzstyle{39}=[-, draw={rgb,255:red,0; green,0; blue,0}]
\tikzstyle{38}=[-, draw={rgb,255:red,0; green,0; blue,0}]
\tikzstyle{37}=[-, draw={rgb,255:red,0; green,0; blue,0}]
\tikzstyle{36}=[-, draw={rgb,255:red,0; green,0; blue,0}]
\tikzstyle{35}=[-, draw={rgb,255:red,0; green,0; blue,0}]
\tikzstyle{34}=[-, draw={rgb,255:red,0; green,0; blue,0}]
\tikzstyle{33}=[-, draw={rgb,255:red,0; green,0; blue,0}]
\tikzstyle{32}=[-, draw={rgb,255:red,0; green,0; blue,0}]
\tikzstyle{31}=[-, draw={rgb,255:red,0; green,0; blue,0}]
\tikzstyle{30}=[-, draw={rgb,255:red,0; green,0; blue,0}]
\tikzstyle{29}=[-, draw={rgb,255:red,0; green,0; blue,0}]
\tikzstyle{28}=[-, draw={rgb,255:red,0; green,0; blue,0}]
\tikzstyle{27}=[-, draw={rgb,255:red,0; green,0; blue,0}]
\tikzstyle{26}=[-, draw={rgb,255:red,0; green,0; blue,0}]
\tikzstyle{25}=[-, draw={rgb,255:red,0; green,0; blue,0}]
\tikzstyle{24}=[-, draw={rgb,255:red,0; green,0; blue,0}]
\tikzstyle{23}=[-, draw={rgb,255:red,0; green,0; blue,0}]
\tikzstyle{22}=[-, draw={rgb,255:red,0; green,0; blue,0}]
\tikzstyle{21}=[-, draw={rgb,255:red,0; green,0; blue,0}]
\tikzstyle{20}=[-, draw={rgb,255:red,0; green,0; blue,0}]
\tikzstyle{19}=[-, draw={rgb,255:red,0; green,0; blue,0}]
\tikzstyle{18}=[-, draw={rgb,255:red,0; green,0; blue,0}]
\tikzstyle{17}=[-, draw={rgb,255:red,0; green,0; blue,0}]
\tikzstyle{16}=[-, draw={rgb,255:red,0; green,0; blue,0}]
\tikzstyle{15}=[-, draw={rgb,255:red,0; green,0; blue,0}]
\tikzstyle{14}=[-, draw={rgb,255:red,0; green,0; blue,0}]
\tikzstyle{13}=[-, draw={rgb,255:red,0; green,0; blue,0}]
\tikzstyle{12}=[-, draw={rgb,255:red,0; green,0; blue,0}]
\tikzstyle{11}=[-, draw={rgb,255:red,0; green,0; blue,0}]
\tikzstyle{10}=[-, draw={rgb,255:red,0; green,0; blue,0}]
\tikzstyle{9}=[-, draw={rgb,255:red,0; green,0; blue,0}]
\tikzstyle{8}=[-, draw={rgb,255:red,0; green,0; blue,0}]
\tikzstyle{7}=[-, draw={rgb,255:red,0; green,0; blue,0}]
\tikzstyle{5}=[-, draw={rgb,255:red,0; green,0; blue,0}]
\tikzstyle{4}=[-, draw={rgb,255:red,0; green,0; blue,0}]
\tikzstyle{2}=[-, draw={rgb,255:red,0; green,0; blue,0}]
\tikzstyle{1}=[-, draw={rgb,255:red,0; green,0; blue,0}]
\tikzstyle{0}=[-, draw={rgb,255:red,0; green,0; blue,0}]

%% file: volcanoes/l=3/craters.tikz
\begin{tikzpicture}
	\begin{pgfonlayer}{nodelayer}
		\node [style=red dot 3] (0) at (-5.5, 0) {};
		\node [style=red dot 3] (1) at (-4.5, 0) {};
		\node [style=red dot 3] (2) at (-3.5, 0) {};
		\node [style=red dot 3] (3) at (-2.5, -0.5) {};
		\node [style=red dot 3] (4) at (-2.5, 0.5) {};
		\node [style=red dot 3] (5) at (-1.5, 0.5) {};
		\node [style=red dot 3] (6) at (-1.5, -0.5) {};
		\node [style=red dot 3] (7) at (-0.5, 0.5) {};
		\node [style=red dot 3] (8) at (-0.5, -0.5) {};
		\node [style=red dot 3] (9) at (0.35, 0) {};
		\node [style=red dot 3] (10) at (1, 0.5) {};
		\node [style=red dot 3] (11) at (1, -0.5) {};
		\node [style=red dot 3] (12) at (2, 0.5) {};
		\node [style=red dot 3] (13) at (2, -0.475) {};
		\node [style=red dot 3] (14) at (2.75, 0.25) {};
		\node [style=red dot 3] (15) at (2.75, -0.25) {};
		\node [style=red dot 3] (16) at (3.25, -0.5) {};
		\node [style=red dot 3] (17) at (3.25, 0.5) {};
	\end{pgfonlayer}
	\begin{pgfonlayer}{edgelayer}
		\draw [draw=black,in=135, out=45, loop] (1) to ();
		\draw [draw=black,in=135, out=45, loop] (2) to ();
		\draw [draw=black,in=-45, out=-135, loop] (2) to ();
		\draw [draw=black] (4) to (3);
		\draw [draw=black,bend right, looseness=1.25] (5) to (6);
		\draw [draw=black] (7) to (8);
		\draw [draw=black] (8) to (9);
		\draw [draw=black] (9) to (7);
		\draw [draw=black] (11) to (10);
		\draw [draw=black] (10) to (12);
		\draw [draw=black] (12) to (13);
		\draw [draw=black] (13) to (11);
		\draw [draw=black,bend left, looseness=1.25] (5) to (6);
		\draw [draw=black] (14) to (17);
		\draw [draw=black] (14) to (15);
		\draw [draw=black] (15) to (16);
		\draw [draw=black,dashed,bend left=75, looseness=2.25] (17) to (16);
	\end{pgfonlayer}
\end{tikzpicture}

%% file: fig/example0_3.tikz
\begin{tikzpicture}
\begin{pgfonlayer}{nodelayer}
\node [style=red dot 3] (0) at (90:1) {};
\node [style=red dot 2] (1) at (270:1) {};
\node [] (2) at (90:2) {\footnotesize{\textcolor{blue}{$1$}}};
\node [] (4) at (180:0.75) {\footnotesize{\textcolor{red}{$1$}}};
\node [] (3) at (0:0.75) {\footnotesize{\textcolor{black}{$3$}}};
\node [] (5) at (270:2) {$\ell=3$};
\node [] (8) at (180:3) {---------};
\node [] (6) at (161:3.2) {\footnotesize{level $0$}};
\node [] (7) at (197:3.2) {\footnotesize{level $1$}};
\node [] (5) at (0:1.1) {\phantom{a}};
\end{pgfonlayer}
\begin{pgfonlayer}{edgelayer}

\path (0) edge [style =  blue arr,out=135,in=45,looseness=15] (0);
\draw [style=black arr,in=60, out=-60,looseness=1,very thick] (0) to (1);
\draw [style= red arr,in=-120, out=120,looseness=1] (1) to (0);
\end{pgfonlayer}
\end{tikzpicture}

%% file: fig/example0_1.tikz
\begin{tikzpicture}
\begin{pgfonlayer}{nodelayer}
\node [style=red dot 3] (0) at (90:1) {};
\node [style=red dot 2] (1) at (-150:2) {};
\node [style=red dot 2] (10) at (-30:2) {};
\node [] (2) at (90:2) {\footnotesize{\textcolor{blue}{$2$}}};
\node [] (4) at (150:1.5) {\footnotesize{\textcolor{red}{$1$}}};
\node [] (4) at (0:1) {\footnotesize{\textcolor{red}{$1$}}};
\node [] (3) at (180:1) {\footnotesize{\textcolor{black}{$3$}}};
\node [] (3) at (30:1.5) {\footnotesize{\textcolor{black}{$3$}}};
\node [] (5) at (270:2) {$\ell\equiv1\text{ mod } 3$};
\node [] (5) at (270:1.1) {$\cdots$};
\node [] (5) at (-15:3) {$\}\frac{\ell-1}{3}$};
\end{pgfonlayer}
\begin{pgfonlayer}{edgelayer}
\path (0) edge [style =  blue arr,out=135,in=45,looseness=15] (0);
\draw [style=black arr,in=30, out=-100,looseness=0.5,very thick] (0) to (1);
\draw [style= red arr,in=-150, out=90,looseness=1] (1) to (0);
\draw [style=black arr,in=90, out=-30,looseness=1,very thick] (0) to (10);
\draw [style= red arr,in=-80, out=150,looseness=0.5] (10) to (0);
\end{pgfonlayer}
\end{tikzpicture}

%% file: fig/example0_2.tikz
\begin{tikzpicture}
\begin{pgfonlayer}{nodelayer}
\node [style=red dot 3] (0) at (90:1) {};
\node [style=red dot 2] (1) at (-150:2) {};
\node [style=red dot 2] (10) at (-30:2) {};
\node [] (4) at (150:1.5) {\footnotesize{\textcolor{red}{$1$}}};
\node [] (4) at (0:1) {\footnotesize{\textcolor{red}{$1$}}};
\node [] (3) at (180:1) {\footnotesize{\textcolor{black}{$3$}}};
\node [] (3) at (30:1.5) {\footnotesize{\textcolor{black}{$3$}}};
\node [] (5) at (270:2) {$\ell\equiv2\text{ mod } 3$};
\node [] (5) at (270:1.1) {$\cdots$};
\node [] (5) at (-15:3) {$\}\frac{\ell+1}{3}$};
\end{pgfonlayer}
\begin{pgfonlayer}{edgelayer}
\draw [style=black arr,in=30, out=-100,looseness=0.5,very thick] (0) to (1);
\draw [style= red arr,in=-150, out=90,looseness=1] (1) to (0);
\draw [style=black arr,in=90, out=-30,looseness=1,very thick] (0) to (10);
\draw [style= red arr,in=-80, out=150,looseness=0.5] (10) to (0);
\end{pgfonlayer}
\end{tikzpicture}

%% file: fig/exemple1728_2.tikz
\begin{tikzpicture}
\begin{pgfonlayer}{nodelayer}
\node [style=red dot 3] (0) at (90:1) {};
\node [style=red dot 2] (1) at (270:1) {};
\node [] (2) at (90:2) {\footnotesize{\textcolor{blue}{$1$}}};
\node [] (4) at (180:0.75) {\footnotesize{\textcolor{red}{$1$}}};
\node [] (3) at (0:0.75) {\footnotesize{\textcolor{black}{$2$}}};
\node [] (5) at (270:2) {$\ell=2$};
\node [] (8) at (180:3) {---------};
\node [] (6) at (161:3.2) {\footnotesize{level $0$}};
\node [] (7) at (197:3.2) {\footnotesize{level $1$}};
\node [] (5) at (0:1.1) {\phantom{a}};
\end{pgfonlayer}
\begin{pgfonlayer}{edgelayer}

\path (0) edge [style =  blue arr,out=135,in=45,looseness=15] (0);
\draw [style=black arr,in=60, out=-60,looseness=1,very thick] (0) to (1);
\draw [style= red arr,in=-120, out=120,looseness=1] (1) to (0);
\end{pgfonlayer}
\end{tikzpicture}

%% file: fig/example1728_1.tikz
\begin{tikzpicture}
\begin{pgfonlayer}{nodelayer}
\node [style=red dot 3] (0) at (90:1) {};
\node [style=red dot 2] (1) at (-150:2) {};
\node [style=red dot 2] (10) at (-30:2) {};
\node [] (2) at (90:2) {\footnotesize{\textcolor{blue}{$2$}}};
\node [] (4) at (150:1.5) {\footnotesize{\textcolor{red}{$1$}}};
\node [] (4) at (0:1) {\footnotesize{\textcolor{red}{$1$}}};
\node [] (3) at (180:1) {\footnotesize{\textcolor{black}{$2$}}};
\node [] (3) at (30:1.5) {\footnotesize{\textcolor{black}{$2$}}};
\node [] (5) at (270:2) {$\ell\equiv1\text{ mod } 4$};
\node [] (5) at (270:1.1) {$\cdots$};
\node [] (5) at (-15:3) {$\}\frac{\ell-1}{2}$};
\end{pgfonlayer}
\begin{pgfonlayer}{edgelayer}
\path (0) edge [style =  blue arr,out=135,in=45,looseness=15] (0);
\draw [style=black arr,in=30, out=-100,looseness=0.5,very thick] (0) to (1);
\draw [style= red arr,in=-150, out=90,looseness=1] (1) to (0);
\draw [style=black arr,in=90, out=-30,looseness=1,very thick] (0) to (10);
\draw [style= red arr,in=-80, out=150,looseness=0.5] (10) to (0);
\end{pgfonlayer}
\end{tikzpicture}

%% file: fig/exemple1728_3.tikz
\begin{tikzpicture}
\begin{pgfonlayer}{nodelayer}
\node [style=red dot 3] (0) at (90:1) {};
\node [style=red dot 2] (1) at (-150:2) {};
\node [style=red dot 2] (10) at (-30:2) {};
\node [] (4) at (150:1.5) {\footnotesize{\textcolor{red}{$1$}}};
\node [] (4) at (0:1) {\footnotesize{\textcolor{red}{$1$}}};
\node [] (3) at (180:1) {\footnotesize{\textcolor{black}{$2$}}};
\node [] (3) at (30:1.5) {\footnotesize{\textcolor{black}{$2$}}};
\node [] (5) at (270:2) {$\ell\equiv3\text{ mod } 4$};
\node [] (5) at (270:1.1) {$\cdots$};
\node [] (5) at (-15:3) {$\}\frac{\ell+1}{2}$};
\end{pgfonlayer}
\begin{pgfonlayer}{edgelayer}
\draw [style=black arr,in=30, out=-100,looseness=0.5,very thick] (0) to (1);
\draw [style= red arr,in=-150, out=90,looseness=1] (1) to (0);
\draw [style=black arr,in=90, out=-30,looseness=1,very thick] (0) to (10);
\draw [style= red arr,in=-80, out=150,looseness=0.5] (10) to (0);
\end{pgfonlayer}
\end{tikzpicture}

%% file: volcanoes/l=2/1009/30.tikz
\begin{tikzpicture}
\begin{pgfonlayer}{nodelayer}
\node [style=red dot 3] (30) at (90:1) {};
\node [style=red dot 3] (655) at (270:1) {};
\node [style=red dot 2] (915) at (90:2) {};
\node [style=red dot 2] (754) at (270:2) {};
\node [style=red dot 1] (439) at (45:3) {};
\node [style=red dot 1] (476) at (135:3) {};
\node [style=red dot 1] (494) at (225:3) {};
\node [style=red dot 1] (925) at (315:3) {};
\node [style=red dot 0] (467) at (22.5:4) {};
\node [style=red dot 0] (789) at (67.5:4) {};
\node [style=red dot 0] (532) at (112.5:4) {};
\node [style=red dot 0] (712) at (157.5:4) {};
\node [style=red dot 0] (286) at (202.5:4) {};
\node [style=red dot 0] (876) at (247.5:4) {};
\node [style=red dot 0] (715) at (292.5:4) {};
\node [style=red dot 0] (978) at (337.5:4) {};
\node [style=red dot 0] (626) at (11.25:5) {};
\node [style=red dot 0] (954) at (33.75:5) {};
\node [style=red dot 0] (106) at (56.25:5) {};
\node [style=red dot 0] (427) at (78.75:5) {};
\node [style=red dot 0] (504) at (101.25:5) {};
\node [style=red dot 0] (533) at (123.75:5) {};
\node [style=red dot 0] (230) at (146.25:5) {};
\node [style=red dot 0] (274) at (168.75:5) {};
\node [style=red dot 0] (100) at (191.25:5) {};
\node [style=red dot 0] (263) at (213.75:5) {};
\node [style=red dot 0] (304) at (236.25:5) {};
\node [style=red dot 0] (438) at (258.75:5) {};
\node [style=red dot 0] (257) at (281.25:5) {};
\node [style=red dot 0] (624) at (303.75:5) {};
\node [style=red dot 0] (377) at (326.25:5) {};
\node [style=red dot 0] (942) at (348.75:5) {};
\end{pgfonlayer}
\begin{pgfonlayer}{edgelayer}
\draw [style= 14,in=0, out=0,looseness=0.5] (655) to (30);
\draw [style= 14,in=180, out=180,looseness=0.5] (655) to (30);
\draw [style= 14] (915) to (30);
\draw [style= 14] (754) to (655);
\draw [style= 14] (467) to (439);
\draw [style= 14] (789) to (439);
\draw [style= 14] (915) to (439);
\draw [style= 14] (532) to (476);
\draw [style= 14] (712) to (476);
\draw [style= 14] (915) to (476);
\draw [style= 14] (533) to (532);
\draw [style= 14] (712) to (230);
\draw [style= 14] (712) to (274);
\draw [style= 14] (532) to (504);
\draw [style= 14] (626) to (467);
\draw [style= 14] (954) to (467);
\draw [style= 14] (789) to (106);
\draw [style= 14] (789) to (427);
\draw [style= 14] (925) to (754);
\draw [style= 14] (754) to (494);
\draw [style= 14] (876) to (494);
\draw [style= 14] (978) to (925);
\draw [style= 14] (925) to (715);
\draw [style= 14] (978) to (377);
\draw [style= 14] (978) to (942);
\draw [style= 14] (715) to (257);
\draw [style= 14] (715) to (624);
\draw [style= 14] (494) to (286);
\draw [style= 14] (876) to (304);
\draw [style= 14] (876) to (438);
\draw [style= 14] (286) to (100);
\draw [style= 14] (286) to (263);
\end{pgfonlayer}
\end{tikzpicture}

%% file: volcanoes/l=3/3121/783.tikz
\begin{tikzpicture}
\begin{pgfonlayer}{nodelayer}
\node [style=red dot 3] (2930) at (90:1) {};
\node [style=red dot 3] (379) at (180:1) {};
\node [style=red dot 3] (1983) at (270:1) {};
\node [style=red dot 3] (264) at (360:1) {};
\node [style=red dot 2] (1861) at (72:2) {};
\node [style=red dot 2] (2852) at (108:2) {};
\node [style=red dot 2] (783) at (162:2) {};
\node [style=red dot 2] (1941) at (198:2) {};
\node [style=red dot 2] (1711) at (252:2) {};
\node [style=red dot 2] (2311) at (288:2) {};
\node [style=red dot 2] (2223) at (342:2) {};
\node [style=red dot 2] (2617) at (378:2) {};
\node [style=red dot 2] (2223) at (342:2) {};
\node [style=red dot 2] (2617) at (378:2) {};
\node [style=red dot 1] (118) at (60:3) {};
\node [style=red dot 1] (413) at (72:3) {};
\node [style=red dot 1] (2153) at (84:3) {};
\node [style=red dot 1] (630) at (96:3) {};
\node [style=red dot 1] (1166) at (108:3) {};
\node [style=red dot 1] (2669) at (120:3) {};
\node [style=red dot 1] (533) at (150:3) {};
\node [style=red dot 1] (1829) at (162:3) {};
\node [style=red dot 1] (2629) at (174:3) {};
\node [style=red dot 1] (1421) at (186:3) {};
\node [style=red dot 1] (1531) at (198:3) {};
\node [style=red dot 1] (2122) at (210:3) {};
\node [style=red dot 1] (817) at (240:3) {};
\node [style=red dot 1] (1171) at (252:3) {};
\node [style=red dot 1] (1286) at (264:3) {};
\node [style=red dot 1] (1910) at (276:3) {};
\node [style=red dot 1] (2019) at (288:3) {};
\node [style=red dot 1] (2923) at (300:3) {};
\node [style=red dot 1] (1273) at (330:3) {};
\node [style=red dot 1] (1337) at (342:3) {};
\node [style=red dot 1] (1397) at (354:3) {};
\node [style=red dot 1] (1154) at (366:3) {};
\node [style=red dot 1] (1389) at (378:3) {};
\node [style=red dot 1] (2231) at (390:3) {};
\node [style=red dot 1] (1273) at (330:3) {};
\node [style=red dot 1] (1337) at (342:3) {};
\node [style=red dot 1] (1397) at (354:3) {};
\node [style=red dot 1] (1154) at (366:3) {};
\node [style=red dot 1] (1389) at (378:3) {};
\node [style=red dot 1] (2231) at (390:3) {};
\end{pgfonlayer}
\begin{pgfonlayer}{edgelayer}
\draw [style=  46] (1861) to (118);
\draw [style=  46] (2153) to (1861);
\draw [style=  46] (2930) to (1861);
\draw [style=  46] (1861) to (413);
\draw [style=  46] (1983) to (264);
\draw [style=  46] (2223) to (264);
\draw [style=  46] (2617) to (264);
\draw [style=  46] (2930) to (264);
\draw [style=  46] (783) to (379);
\draw [style=  46] (1941) to (379);
\draw [style=  46] (1983) to (379);
\draw [style=  46] (2930) to (379);
\draw [style=  46] (2930) to (2852);
\draw [style=  46] (2852) to (630);
\draw [style=  46] (2852) to (1166);
\draw [style=  46] (2852) to (2669);
\draw [style=  46] (1829) to (783);
\draw [style=  46] (2629) to (783);
\draw [style=  46] (2122) to (1941);
\draw [style=  46] (2311) to (1983);
\draw [style=  46] (1983) to (1711);
\draw [style=  46] (2923) to (2311);
\draw [style=  46] (2311) to (1910);
\draw [style=  46] (2311) to (2019);
\draw [style=  46] (1711) to (817);
\draw [style=  46] (1711) to (1171);
\draw [style=  46] (1711) to (1286);
\draw [style=  46] (2617) to (1154);
\draw [style=  46] (2617) to (1389);
\draw [style=  46] (2617) to (2231);
\draw [style=  46] (2223) to (1273);
\draw [style=  46] (2223) to (1337);
\draw [style=  46] (2223) to (1397);
\draw [style=  46] (1941) to (1421);
\draw [style=  46] (1941) to (1531);
\draw [style=  46] (783) to (533);
\end{pgfonlayer}
\end{tikzpicture}

%% file: volcanoes/l=3/3121/1.tikz
\begin{tikzpicture}
\begin{pgfonlayer}{nodelayer}
\node [style=red dot 3] (1) at (90:1) {};
\node [style=red dot 3] (2940) at (120:1) {};
\node [style=red dot 3] (1352) at (150:1) {};
\node [style=red dot 3] (1972) at (180:1) {};
\node [style=red dot 3] (1831) at (210:1) {};
\node [style=red dot 3] (2150) at (240:1) {};
\node [style=red dot 3] (1265) at (270:1) {};
\node [style=red dot 3] (540) at (300:1) {};
\node [style=red dot 3] (336) at (330:1) {};
\node [style=red dot 3] (862) at (360:1) {};
\node [style=red dot 3] (811) at (390:1) {};
\node [style=red dot 3] (2500) at (420:1) {};
\node [style=red dot 2] (1095) at (83.0769:2) {};
\node [style=red dot 2] (1302) at (96.9231:2) {};
\node [style=red dot 2] (1951) at (113.077:2) {};
\node [style=red dot 2] (2841) at (126.923:2) {};
\node [style=red dot 2] (2024) at (143.077:2) {};
\node [style=red dot 2] (2329) at (156.923:2) {};
\node [style=red dot 2] (1240) at (173.077:2) {};
\node [style=red dot 2] (2694) at (186.923:2) {};
\node [style=red dot 2] (1757) at (203.077:2) {};
\node [style=red dot 2] (1930) at (216.923:2) {};
\node [style=red dot 2] (2446) at (233.077:2) {};
\node [style=red dot 2] (2969) at (246.923:2) {};
\node [style=red dot 2] (473) at (263.077:2) {};
\node [style=red dot 2] (1310) at (276.923:2) {};
\node [style=red dot 2] (577) at (293.077:2) {};
\node [style=red dot 2] (2516) at (306.923:2) {};
\node [style=red dot 2] (1501) at (323.077:2) {};
\node [style=red dot 2] (1665) at (336.923:2) {};
\node [style=red dot 2] (272) at (353.077:2) {};
\node [style=red dot 2] (2759) at (366.923:2) {};
\node [style=red dot 2] (411) at (383.077:2) {};
\node [style=red dot 2] (2647) at (396.923:2) {};
\node [style=red dot 2] (541) at (413.077:2) {};
\node [style=red dot 2] (2391) at (426.923:2) {};
\node [style=red dot 2] (541) at (413.077:2) {};
\node [style=red dot 2] (2391) at (426.923:2) {};
\end{pgfonlayer}
\begin{pgfonlayer}{edgelayer}
\draw [style=  64] (1095) to (1);
\draw [style=  64] (1302) to (1);
\draw [style=  64] (2500) to (1);
\draw [style=  64] (2940) to (1);
\draw [style=  64] (1972) to (1352);
\draw [style=  64] (2024) to (1352);
\draw [style=  64] (2329) to (1352);
\draw [style=  64] (2940) to (1352);
\draw [style=  64] (2940) to (1951);
\draw [style=  64] (2940) to (2841);
\draw [style=  64] (2694) to (1972);
\draw [style=  64] (1972) to (1240);
\draw [style=  64] (1930) to (1831);
\draw [style=  64] (1972) to (1831);
\draw [style=  64] (2150) to (1831);
\draw [style=  64] (1831) to (1757);
\draw [style=  64] (2446) to (2150);
\draw [style=  64] (2969) to (2150);
\draw [style=  64] (1310) to (1265);
\draw [style=  64] (2150) to (1265);
\draw [style=  64] (1265) to (473);
\draw [style=  64] (577) to (540);
\draw [style=  64] (1265) to (540);
\draw [style=  64] (2516) to (540);
\draw [style=  64] (540) to (336);
\draw [style=  64] (862) to (336);
\draw [style=  64] (1501) to (336);
\draw [style=  64] (1665) to (336);
\draw [style=  64] (2759) to (862);
\draw [style=  64] (862) to (272);
\draw [style=  64] (862) to (811);
\draw [style=  64] (2500) to (811);
\draw [style=  64] (2647) to (811);
\draw [style=  64] (811) to (411);
\draw [style=  64] (2500) to (541);
\draw [style=  64] (2500) to (2391);
\end{pgfonlayer}
\end{tikzpicture}

%% file: volcanoes/l=2/square/ex.tikz
\begin{tikzpicture}
\begin{pgfonlayer}{nodelayer}
\node [style=red dot 3] (30) at (180:1) {};
\node [style=red dot 3] (655) at (0:1) {};
\node [style=red dot 2] (915) at (180:2) {};
\node [style=red dot 2] (754) at (0:2) {};
\end{pgfonlayer}
\begin{pgfonlayer}{edgelayer}
\draw [style= 14,in=90, out=90,looseness=0.5] (655) to (30);
\draw [style= 14,in=270, out=270,looseness=0.5] (655) to (30);
\draw [style= 14] (915) to (30);
\draw [style= 14] (754) to (655);
\end{pgfonlayer}
\end{tikzpicture}

%% file: volcanoes/l=3/103/8.tikz
\begin{tikzpicture}
\begin{pgfonlayer}{nodelayer}
\node [style=red dot 3] (29) at (90:1) {};
\node [style=red dot 3] (70) at (150:1) {};
\node [style=red dot 3] (43) at (210:1) {};
\node [style=red dot 3] (5) at (270:1) {};
\node [style=red dot 3] (60) at (330:1) {};
\node [style=red dot 3] (32) at (390:1) {};
\node [style=red dot 2] (4) at (90:2) {};
\node [style=red dot 2] (44) at (150:2) {};
\node [style=red dot 2] (86) at (210:2) {};
\node [style=red dot 2] (85) at (270:2) {};
\node [style=red dot 2] (63) at (330:2) {};
\node [style=red dot 2] (62) at (390:2) {};
\node [style=red dot 2] (62) at (390:2) {};
\end{pgfonlayer}
\begin{pgfonlayer}{edgelayer}
\draw [ ] (29) to (4);
\draw [ ] (32) to (29);
\draw [ ] (70) to (29);
\draw [ ] (60) to (32);
\draw [ ] (62) to (32);
\draw [ ] (70) to (43);
\draw [ ] (86) to (43);
\draw [ ] (70) to (44);
\draw [ ] (43) to (5);
\draw [ ] (60) to (5);
\draw [ ] (85) to (5);
\draw [ ] (63) to (60);
\end{pgfonlayer}
\end{tikzpicture}

%% file: volcanoes/l=3/1009/19.tikz
\begin{tikzpicture}
\begin{pgfonlayer}{nodelayer}
\node [] (-1) at (180:1) {19};
\node [style=red dot 3] (45) at (90:1) {};
\node [style=red dot 3] (154) at (270:1) {};
\end{pgfonlayer}
\begin{pgfonlayer}{edgelayer}
\draw [style=  19] (154) to (45);
\end{pgfonlayer}
\end{tikzpicture}
\begin{tikzpicture}
\begin{pgfonlayer}{nodelayer}

\node [style=red dot 3] (130) at (90:1) {};
\node [style=red dot 3] (202) at (270:1) {};
\end{pgfonlayer}
\begin{pgfonlayer}{edgelayer}
\draw [] (202) to (130);
\end{pgfonlayer}
\end{tikzpicture}
\begin{tikzpicture}
\begin{pgfonlayer}{nodelayer}
\node [style=red dot 3] (142) at (90:1) {};
\node [style=red dot 3] (503) at (270:1) {};
\end{pgfonlayer}
\begin{pgfonlayer}{edgelayer}
\draw [] (503) to (142);
\end{pgfonlayer}
\end{tikzpicture}
\begin{tikzpicture}
\begin{pgfonlayer}{nodelayer}
\node [style=red dot 3] (190) at (90:1) {};
\node [style=red dot 3] (761) at (270:1) {};
\end{pgfonlayer}
\begin{pgfonlayer}{edgelayer}
\draw [] (761) to (190);
\end{pgfonlayer}
\end{tikzpicture}
\begin{tikzpicture}
\begin{pgfonlayer}{nodelayer}
\node [style=red dot 3] (200) at (90:1) {};
\node [style=red dot 3] (610) at (270:1) {};
\end{pgfonlayer}
\begin{pgfonlayer}{edgelayer}
\draw [] (610) to (200);
\end{pgfonlayer}
\end{tikzpicture}
\begin{tikzpicture}
\begin{pgfonlayer}{nodelayer}
\node [style=red dot 3] (493) at (90:1) {};
\node [style=red dot 3] (728) at (270:1) {};
\end{pgfonlayer}
\begin{pgfonlayer}{edgelayer}
\draw [] (728) to (493);
\end{pgfonlayer}
\end{tikzpicture}
\begin{tikzpicture}
\begin{pgfonlayer}{nodelayer}
\node [style=red dot 3] (579) at (90:1) {};
\node [style=red dot 3] (703) at (270:1) {};
\end{pgfonlayer}
\begin{pgfonlayer}{edgelayer}
\draw [] (703) to (579);
\end{pgfonlayer}
\end{tikzpicture}
\begin{tikzpicture}
\begin{pgfonlayer}{nodelayer}
\node [style=red dot 3] (735) at (90:1) {};
\node [style=red dot 3] (867) at (270:1) {};
\end{pgfonlayer}
\begin{pgfonlayer}{edgelayer}
\draw [] (867) to (735);
\end{pgfonlayer}
\end{tikzpicture}

%% file: volcanoes/l=3/1009/43.tikz
\begin{tikzpicture}
\begin{pgfonlayer}{nodelayer}
\node [] (-1) at (180:3.5) {43};
\node [style=blue dot 3] (0) at (90:1) {};
\node [style=red dot 2] (611) at (270:1) {};
\node [style=red dot 1] (110) at (150:2) {};
\node [style=red dot 1] (433) at (270:2) {};
\node [style=red dot 1] (652) at (30:2) {};
\node [style=red dot 0] (705) at (110:3) {};
\node [style=red dot 0] (760) at (150:3) {};
\node [style=red dot 0] (944) at (190:3) {};
\node [style=red dot 0] (65) at (230:3) {};
\node [style=red dot 0] (153) at (270:3) {};
\node [style=red dot 0] (366) at (-50:3) {};
\node [style=red dot 0] (143) at (-10:3) {};
\node [style=red dot 0] (217) at (30:3) {};
\node [style=red dot 0] (410) at (70:3) {};
\node [] (4) at (150:0.65) {\footnotesize{\textcolor{red}{$1$}}};
\node [] (3) at (30:0.65) {\footnotesize{\textcolor{black}{$3$}}};
\end{pgfonlayer}
\begin{pgfonlayer}{edgelayer}
\path (0) edge [style =  blue arr,out=135,in=45,looseness=15] (0);
\draw [style=black arr,in=60, out=-60,looseness=1,very thick] (0) to (611);
\draw [style= red arr,in=-120, out=120,looseness=1] (611) to (0);
\draw [style=  43] (652) to (611);
\draw [style=  43] (611) to (110);
\draw [style=  43] (705) to (110);
\draw [style=  43] (760) to (110);
\draw [style=  43] (944) to (110);
\draw [style=  43] (611) to (433);
\draw [style=  43] (652) to (143);
\draw [style=  43] (652) to (217);
\draw [style=  43] (652) to (410);
\draw [style=  43] (433) to (65);
\draw [style=  43] (433) to (153);
\draw [style=  43] (433) to (366);
\end{pgfonlayer}
\end{tikzpicture}

%% file: volcanoes/l=3/1009/62.tikz
\begin{tikzpicture}
\begin{pgfonlayer}{nodelayer}
\node [] (-1) at (180:1) {62};
\node [style=red dot 3] (523) at (90:0) {};
\end{pgfonlayer}
\begin{pgfonlayer}{edgelayer}
\path (523) edge [style =  62,out=315,in=225,looseness=15] node[above] {} (523);
\end{pgfonlayer}
\end{tikzpicture}
\begin{tikzpicture}
\begin{pgfonlayer}{nodelayer}
\node [style=red dot 3] (97) at (90:1) {};
\node [style=red dot 3] (424) at (270:1) {};
\end{pgfonlayer}
\begin{pgfonlayer}{edgelayer}
\draw [style=  62] (424) to (97);
\end{pgfonlayer}
\end{tikzpicture}
\begin{tikzpicture}
\begin{pgfonlayer}{nodelayer}
\node [style=red dot 3] (264) at (90:1) {};
\node [style=red dot 3] (627) at (270:1) {};
\end{pgfonlayer}
\begin{pgfonlayer}{edgelayer}
\draw [style=  62] (627) to (264);
\end{pgfonlayer}
\end{tikzpicture}
\begin{tikzpicture}
\begin{pgfonlayer}{nodelayer}
\node [style=red dot 3] (445) at (90:1) {};
\node [style=red dot 3] (929) at (270:1) {};
\end{pgfonlayer}
\begin{pgfonlayer}{edgelayer}
\draw [style=  62] (929) to (445);
\end{pgfonlayer}
\end{tikzpicture}

%% file: volcanoes/l=3/1009/30.tikz
\begin{tikzpicture}
\begin{pgfonlayer}{nodelayer}
\node [] (-1) at (-1,0) {30};
\node [style=red dot 3] (1) at (0,-1.5) {};
\node [style=red dot 3] (2) at (0,-0.5) {};
\node [style=red dot 3] (1) at (0,0.5) {};
\node [style=red dot 3] (2) at (0,1.5) {};
\node [style=red dot 3] (1) at (1,-1.5) {};
\node [style=red dot 3] (2) at (1,-0.5) {};
\node [style=red dot 3] (1) at (1,0.5) {};
\node [style=red dot 3] (2) at (1,1.5) {};
\node [style=red dot 3] (1) at (2,-1.5) {};
\node [style=red dot 3] (2) at (2,-0.5) {};
\node [style=red dot 3] (1) at (2,0.5) {};
\node [style=red dot 3] (2) at (2,1.5) {};
\node [style=red dot 3] (1) at (3,-1.5) {};
\node [style=red dot 3] (2) at (3,-0.5) {};
\node [style=red dot 3] (1) at (3,0.5) {};
\node [style=red dot 3] (2) at (3,1.5) {};
\node [style=red dot 3] (1) at (4,-1.5) {};
\node [style=red dot 3] (2) at (4,-0.5) {};
\node [style=red dot 3] (1) at (4,0.5) {};
\node [style=red dot 3] (2) at (4,1.5) {};
\node [style=red dot 3] (1) at (5,-1.5) {};
\node [style=red dot 3] (2) at (5,-0.5) {};
\node [style=red dot 3] (1) at (5,0.5) {};
\node [style=red dot 3] (2) at (5,1.5) {};
\node [style=red dot 3] (1) at (6,-1.5) {};
\node [style=red dot 3] (2) at (6,-0.5) {};
\node [style=red dot 3] (1) at (6,0.5) {};
\node [style=red dot 3] (2) at (6,1.5) {};
\node [style=red dot 3] (1) at (7,-1.5) {};
\node [style=red dot 3] (2) at (7,-0.5) {};
\node [style=red dot 3] (1) at (7,0.5) {};
\end{pgfonlayer}
\end{tikzpicture}

%% file: volcanoes/l=3/1009/56.tikz
\begin{tikzpicture}
\begin{pgfonlayer}{nodelayer}
\node [] (-1) at (180:2.5) {56};
\node [style=blue dot 3] (0) at (0:0) {};
\node [style=red dot 2] (1) at (150:2) {};
\node [style=red dot 2] (10) at (-30:2) {};
\node [] (4) at (180:1.5) {\footnotesize{\textcolor{red}{$1$}}};
\node [] (4) at (0:1.5) {\footnotesize{\textcolor{red}{$1$}}};
\node [] (3) at (110:1) {\footnotesize{\textcolor{black}{$2$}}};
\node [] (3) at (290:1) {\footnotesize{\textcolor{black}{$2$}}};
\end{pgfonlayer}
\begin{pgfonlayer}{edgelayer}
\draw [style=black arr,in=0, out=120,looseness=1,very thick] (0) to (1);
\draw [style= red arr,in=180, out=270,looseness=1] (1) to (0);
\draw [style=black arr,in=180, out=-60,looseness=1,very thick] (0) to (10);
\draw [style= red arr,in=0, out=110,looseness=1] (10) to (0);
\end{pgfonlayer}
\end{tikzpicture}
\begin{tikzpicture}
\begin{pgfonlayer}{nodelayer}
\node [style=red dot 3] (1) at (90:0) {};
\node [style=red dot 2] (135) at (-45:1) {};
\node [style=red dot 2] (453) at (45:1) {};
\node [style=red dot 2] (515) at (135:1) {};
\node [style=red dot 2] (933) at (225:1) {};
\end{pgfonlayer}
\begin{pgfonlayer}{edgelayer}
\draw [] (135) to (1);
\draw [] (453) to (1);
\draw [] (515) to (1);
\draw [] (933) to (1);
\end{pgfonlayer}
\end{tikzpicture}
\begin{tikzpicture}
\begin{pgfonlayer}{nodelayer}
\node [style=red dot 3] (247) at (90:0) {};
\node [style=red dot 2] (524) at (-45:1) {};
\node [style=red dot 2] (686) at (45:1) {};
\node [style=red dot 2] (891) at (135:1) {};
\node [style=red dot 2] (917) at (225:1) {};
\end{pgfonlayer}
\begin{pgfonlayer}{edgelayer}
\draw (524) to (247);
\draw (686) to (247);
\draw (891) to (247);
\draw (917) to (247);
\end{pgfonlayer}
\end{tikzpicture}

%% file: volcanoes/point.tikz
\begin{tikzpicture}
\begin{pgfonlayer}{nodelayer}
\node [] (0) at (90:0.5) {};
\node [] (0) at (270:0.5) {};
\node [style=red dot 3] (29) at (0:0) {};
\end{pgfonlayer}
\end{tikzpicture}

%% file: volcanoes/line.tikz
\begin{tikzpicture}
\begin{pgfonlayer}{nodelayer}
\node [] (0) at (90:0.27) {};
\node [] (0) at (270:0.2) {};
\node [style=red dot 3] (45) at (0:1) {};
\node [style=red dot 3] (154) at (180:1) {};
\end{pgfonlayer}
\begin{pgfonlayer}{edgelayer}
\draw [] (154) to (45);
\end{pgfonlayer}
\end{tikzpicture}

%% file: volcanoes/cross.tikz
\begin{tikzpicture}
\begin{pgfonlayer}{nodelayer}
\node [] (0) at (90:1.2) {};
\node [] (0) at (270:1.2) {};
\node [style=red dot 3] (362) at (90:0) {};
\node [style=red dot 2] (54) at (-45:1) {};
\node [style=red dot 2] (134) at (45:1) {};
\node [style=red dot 2] (140) at (135:1) {};
\node [style=red dot 2] (261) at (225:1) {};
\end{pgfonlayer}
\begin{pgfonlayer}{edgelayer}
\draw [] (362) to (54);
\draw [] (362) to (134);
\draw [] (362) to (140);
\draw [] (362) to (261);
\end{pgfonlayer}
\end{tikzpicture}

%% file: volcanoes/l=3/1009/7.tikz
\begin{tikzpicture}
\begin{pgfonlayer}{nodelayer}
\node [] (0) at (90:2.2) {};
\node [] (0) at (270:2.2) {};
\node [style=red dot 3] (749) at (90:1) {};
\node [style=red dot 3] (450) at (162:1) {};
\node [style=red dot 3] (299) at (234:1) {};
\node [style=red dot 3] (290) at (306:1) {};
\node [style=red dot 3] (431) at (378:1) {};
\node [style=red dot 2] (12) at (75:2) {};
\node [style=red dot 2] (147) at (105:2) {};
\node [style=red dot 2] (197) at (147:2) {};
\node [style=red dot 2] (847) at (177:2) {};
\node [style=red dot 2] (71) at (219:2) {};
\node [style=red dot 2] (596) at (249:2) {};
\node [style=red dot 2] (277) at (291:2) {};
\node [style=red dot 2] (980) at (321:2) {};
\node [style=red dot 2] (319) at (363:2) {};
\node [style=red dot 2] (756) at (393:2) {};
\node [style=red dot 2] (319) at (363:2) {};
\node [style=red dot 2] (756) at (393:2) {};
\end{pgfonlayer}
\begin{pgfonlayer}{edgelayer}
\draw [style=  7] (749) to (12);
\draw [style=  7] (749) to (147);
\draw [style=  7] (749) to (431);
\draw [style=  7] (756) to (431);
\draw [style=  7] (749) to (450);
\draw [style=  7] (847) to (450);
\draw [style=  7] (450) to (197);
\draw [style=  7] (450) to (299);
\draw [style=  7] (596) to (299);
\draw [style=  7] (299) to (71);
\draw [style=  7] (299) to (290);
\draw [style=  7] (431) to (290);
\draw [style=  7] (980) to (290);
\draw [style=  7] (290) to (277);
\draw [style=  7] (431) to (319);
\end{pgfonlayer}
\end{tikzpicture}

%% file: volcanoes/l=3/1009/11.tikz
\begin{tikzpicture}
\begin{pgfonlayer}{nodelayer}
\node [] (0) at (90:2.2) {};
\node [] (0) at (270:2.2) {};
\node [style=red dot 3] (990) at (90:1) {};
\node [style=red dot 3] (527) at (270:1) {};
\node [style=red dot 2] (34) at (30:2) {};
\node [style=red dot 2] (255) at (90:2) {};
\node [style=red dot 2] (819) at (150:2) {};
\node [style=red dot 2] (381) at (210:2) {};
\node [style=red dot 2] (768) at (270:2) {};
\node [style=red dot 2] (773) at (330:2) {};
\end{pgfonlayer}
\begin{pgfonlayer}{edgelayer}
\draw [style=  11] (990) to (34);
\draw [style=  11] (990) to (255);
\draw [style=  11] (768) to (527);
\draw [style=  11] (773) to (527);
\draw [style=  11] (990) to (527);
\draw [style=  11] (990) to (819);
\draw [style=  11] (527) to (381);
\end{pgfonlayer}
\end{tikzpicture}

%% file: volcanoes/l=3/1009/16.tikz
\begin{tikzpicture}
\begin{pgfonlayer}{nodelayer}
\node [] (0) at (90:2.2) {};
\node [] (0) at (270:2.2) {};
\node [style=red dot 3] (985) at (90:1) {};
\node [style=red dot 3] (878) at (270:1) {};
\node [style=red dot 2] (18) at (30:2) {};
\node [style=red dot 2] (253) at (90:2) {};
\node [style=red dot 2] (926) at (150:2) {};
\node [style=red dot 2] (675) at (210:2) {};
\node [style=red dot 2] (681) at (270:2) {};
\node [style=red dot 2] (758) at (330:2) {};
\end{pgfonlayer}
\begin{pgfonlayer}{edgelayer}
\draw [style=  16] (985) to (18);
\draw [style=  16] (985) to (253);
\draw [style=  16] (985) to (878);
\draw [style=  16] (985) to (926);
\draw [style=  16] (878) to (675);
\draw [style=  16] (878) to (681);
\draw [style=  16] (878) to (758);
\end{pgfonlayer}
\end{tikzpicture}

%% file: volcanoes/l=3/1009/20.tikz
\begin{tikzpicture}
\begin{pgfonlayer}{nodelayer}
\node [] (0) at (90:2.2) {};
\node [] (0) at (270:2.2) {};
\node [style=red dot 3] (502) at (90:1) {};
\node [style=red dot 3] (586) at (115.714:1) {};
\node [style=red dot 3] (700) at (141.429:1) {};
\node [style=red dot 3] (490) at (167.143:1) {};
\node [style=red dot 3] (394) at (192.857:1) {};
\node [style=red dot 3] (613) at (218.571:1) {};
\node [style=red dot 3] (393) at (244.286:1) {};
\node [style=red dot 3] (741) at (270:1) {};
\node [style=red dot 3] (778) at (295.714:1) {};
\node [style=red dot 3] (397) at (321.429:1) {};
\node [style=red dot 3] (996) at (347.143:1) {};
\node [style=red dot 3] (456) at (372.857:1) {};
\node [style=red dot 3] (205) at (398.571:1) {};
\node [style=red dot 3] (540) at (424.286:1) {};
\node [style=red dot 2] (36) at (84:2) {};
\node [style=red dot 2] (422) at (96:2) {};
\node [style=red dot 2] (306) at (109.714:2) {};
\node [style=red dot 2] (808) at (121.714:2) {};
\node [style=red dot 2] (323) at (135.429:2) {};
\node [style=red dot 2] (898) at (147.429:2) {};
\node [style=red dot 2] (412) at (161.143:2) {};
\node [style=red dot 2] (865) at (173.143:2) {};
\node [style=red dot 2] (118) at (186.857:2) {};
\node [style=red dot 2] (292) at (198.857:2) {};
\node [style=red dot 2] (367) at (212.571:2) {};
\node [style=red dot 2] (895) at (224.571:2) {};
\node [style=red dot 2] (434) at (238.286:2) {};
\node [style=red dot 2] (855) at (250.286:2) {};
\node [style=red dot 2] (311) at (264:2) {};
\node [style=red dot 2] (484) at (276:2) {};
\node [style=red dot 2] (776) at (289.714:2) {};
\node [style=red dot 2] (848) at (301.714:2) {};
\node [style=red dot 2] (83) at (315.429:2) {};
\node [style=red dot 2] (843) at (327.429:2) {};
\node [style=red dot 2] (772) at (341.143:2) {};
\node [style=red dot 2] (833) at (353.143:2) {};
\node [style=red dot 2] (333) at (366.857:2) {};
\node [style=red dot 2] (976) at (378.857:2) {};
\node [style=red dot 2] (221) at (392.571:2) {};
\node [style=red dot 2] (281) at (404.571:2) {};
\node [style=red dot 2] (109) at (418.286:2) {};
\node [style=red dot 2] (324) at (430.286:2) {};
\node [style=red dot 2] (109) at (418.286:2) {};
\node [style=red dot 2] (324) at (430.286:2) {};
\end{pgfonlayer}
\begin{pgfonlayer}{edgelayer}
\draw [style=  20] (502) to (36);
\draw [style=  20] (540) to (502);
\draw [style=  20] (586) to (502);
\draw [style=  20] (502) to (422);
\draw [style=  20] (700) to (586);
\draw [style=  20] (808) to (586);
\draw [style=  20] (586) to (306);
\draw [style=  20] (898) to (700);
\draw [style=  20] (700) to (323);
\draw [style=  20] (700) to (490);
\draw [style=  20] (865) to (490);
\draw [style=  20] (490) to (394);
\draw [style=  20] (613) to (394);
\draw [style=  20] (490) to (412);
\draw [style=  20] (394) to (118);
\draw [style=  20] (394) to (292);
\draw [style=  20] (895) to (613);
\draw [style=  20] (613) to (367);
\draw [style=  20] (434) to (393);
\draw [style=  20] (613) to (393);
\draw [style=  20] (741) to (393);
\draw [style=  20] (855) to (393);
\draw [style=  20] (778) to (741);
\draw [style=  20] (741) to (311);
\draw [style=  20] (741) to (484);
\draw [style=  20] (848) to (778);
\draw [style=  20] (778) to (397);
\draw [style=  20] (843) to (397);
\draw [style=  20] (996) to (397);
\draw [style=  20] (778) to (776);
\draw [style=  20] (397) to (83);
\draw [style=  20] (976) to (456);
\draw [style=  20] (996) to (456);
\draw [style=  20] (996) to (772);
\draw [style=  20] (996) to (833);
\draw [style=  20] (221) to (205);
\draw [style=  20] (281) to (205);
\draw [style=  20] (456) to (205);
\draw [style=  20] (540) to (205);
\draw [style=  20] (456) to (333);
\draw [style=  20] (540) to (109);
\draw [style=  20] (540) to (324);
\end{pgfonlayer}
\end{tikzpicture}

%% file: volcanoes/l=3/1009/34.tikz
\begin{tikzpicture}
\begin{pgfonlayer}{nodelayer}
\node [] (0) at (90:2.2) {};
\node [] (0) at (270:2.2) {};
\node [style=red dot 3] (309) at (90:1) {};
\node [style=red dot 3] (423) at (270:1) {};
\node [style=red dot 2] (287) at (45:2) {};
\node [style=red dot 2] (886) at (135:2) {};
\node [style=red dot 2] (429) at (225:2) {};
\node [style=red dot 2] (782) at (315:2) {};
\end{pgfonlayer}
\begin{pgfonlayer}{edgelayer}
\draw [style= 34] (309) to (287);
\draw [style= 34,in=0, out=0,looseness=0.5] (423) to (309);
\draw [style= 34,in=180, out=180,looseness=0.5] (423) to (309);
\draw [style= 34] (886) to (309);
\draw [style= 34] (429) to (423);
\draw [style= 34] (782) to (423);
\end{pgfonlayer}
\end{tikzpicture}

%% file: volcanoes/l=3/1009/34b.tikz
\begin{tikzpicture}
\begin{pgfonlayer}{nodelayer}
\node [] (0) at (90:2.2) {};
\node [] (0) at (270:2.2) {};
\node [style=red dot 3] (437) at (90:1) {};
\node [style=red dot 3] (884) at (180:1) {};
\node [style=red dot 3] (117) at (270:1) {};
\node [style=red dot 3] (254) at (360:1) {};
\node [style=red dot 2] (16) at (72:2) {};
\node [style=red dot 2] (75) at (108:2) {};
\node [style=red dot 2] (241) at (162:2) {};
\node [style=red dot 2] (872) at (198:2) {};
\node [style=red dot 2] (329) at (252:2) {};
\node [style=red dot 2] (617) at (288:2) {};
\node [style=red dot 2] (296) at (342:2) {};
\node [style=red dot 2] (320) at (378:2) {};
\node [style=red dot 2] (296) at (342:2) {};
\node [style=red dot 2] (320) at (378:2) {};
\end{pgfonlayer}
\begin{pgfonlayer}{edgelayer}
\draw [style=  34] (437) to (16);
\draw [style=  34] (884) to (437);
\draw [style=  34] (437) to (75);
\draw [style=  34] (296) to (254);
\draw [style=  34] (320) to (254);
\draw [style=  34] (437) to (254);
\draw [style=  34] (254) to (117);
\draw [style=  34] (329) to (117);
\draw [style=  34] (617) to (117);
\draw [style=  34] (884) to (117);
\draw [style=  34] (884) to (241);
\draw [style=  34] (884) to (872);
\end{pgfonlayer}
\end{tikzpicture}

%% file: volcanoes/l=3/1009/38a.tikz
\begin{tikzpicture}
\begin{pgfonlayer}{nodelayer}
\node [] (0) at (90:2.2) {};
\node [] (0) at (270:2.2) {};
\node [style=red dot 3] (937) at (90:0) {};
\node [style=red dot 2] (13) at (0:1) {};
\node [style=red dot 2] (870) at (180:1) {};
\node [style=red dot 1] (93) at (-60:2) {};
\node [style=red dot 1] (506) at (0:2) {};
\node [style=red dot 1] (997) at (60:2) {};
\node [style=red dot 1] (345) at (120:2) {};
\node [style=red dot 1] (382) at (180:2) {};
\node [style=red dot 1] (385) at (240:2) {};
\end{pgfonlayer}
\begin{pgfonlayer}{edgelayer}
\draw [style=  38] (93) to (13);
\draw [style=  38] (506) to (13);
\draw [style=  38] (937) to (13);
\draw [style=  38] (997) to (13);
\path (937) edge [style =  38,out=315,in=225,looseness=15] node[above] {} (937);
\path (937) edge [style =  38,out=135,in=45,looseness=15] node[above] {} (937);
\draw [style=  38] (937) to (870);
\draw [style=  38] (870) to (345);
\draw [style=  38] (870) to (382);
\draw [style=  38] (870) to (385);
\end{pgfonlayer}
\end{tikzpicture}

%% file: volcanoes/l=3/1009/38b.tikz
\begin{tikzpicture}
\begin{pgfonlayer}{nodelayer}
\node [] (0) at (90:3.2) {};
\node [] (0) at (270:3.2) {};
\node [style=red dot 3] (173) at (90:1) {};
\node [style=red dot 3] (780) at (270:1) {};
\node [style=red dot 2] (653) at (45:2) {};
\node [style=red dot 2] (813) at (135:2) {};
\node [style=red dot 2] (42) at (225:2) {};
\node [style=red dot 2] (880) at (315:2) {};
\node [style=red dot 1] (17) at (15:3) {};
\node [style=red dot 1] (750) at (45:3) {};
\node [style=red dot 1] (955) at (75:3) {};
\node [style=red dot 1] (187) at (105:3) {};
\node [style=red dot 1] (662) at (135:3) {};
\node [style=red dot 1] (934) at (165:3) {};
\node [style=red dot 1] (67) at (195:3) {};
\node [style=red dot 1] (706) at (225:3) {};
\node [style=red dot 1] (900) at (255:3) {};
\node [style=red dot 1] (464) at (285:3) {};
\node [style=red dot 1] (680) at (315:3) {};
\node [style=red dot 1] (781) at (345:3) {};
\end{pgfonlayer}
\begin{pgfonlayer}{edgelayer}
\draw [style= 38] (653) to (17);
\draw [style= 38] (750) to (653);
\draw [style= 38] (955) to (653);
\draw [style= 38] (653) to (173);
\draw [style= 38,in=0, out=0,looseness=0.5] (780) to (173);
\draw [style= 38,in=180, out=180,looseness=0.5] (780) to (173);
\draw [style= 38] (813) to (173);
\draw [style= 38] (880) to (780);
\draw [style= 38] (934) to (813);
\draw [style= 38] (813) to (187);
\draw [style= 38] (813) to (662);
\draw [style= 38] (67) to (42);
\draw [style= 38] (706) to (42);
\draw [style= 38] (780) to (42);
\draw [style= 38] (900) to (42);
\draw [style= 38] (880) to (464);
\draw [style= 38] (880) to (680);
\draw [style= 38] (880) to (781);
\end{pgfonlayer}
\end{tikzpicture}

%% file: volcanoes/l=3/1009/47.tikz
\begin{tikzpicture}
\begin{pgfonlayer}{nodelayer}
\node [] (0) at (90:2.2) {};
\node [] (0) at (270:2.2) {};
\node [style=red dot 3] (24) at (90:1) {};
\node [style=red dot 3] (906) at (180:1) {};
\node [style=red dot 3] (677) at (270:1) {};
\node [style=red dot 3] (171) at (360:1) {};
\node [style=red dot 2] (416) at (72:2) {};
\node [style=red dot 2] (775) at (108:2) {};
\node [style=red dot 2] (298) at (162:2) {};
\node [style=red dot 2] (308) at (198:2) {};
\node [style=red dot 2] (671) at (252:2) {};
\node [style=red dot 2] (716) at (288:2) {};
\node [style=red dot 2] (690) at (342:2) {};
\node [style=red dot 2] (799) at (378:2) {};
\node [style=red dot 2] (690) at (342:2) {};
\node [style=red dot 2] (799) at (378:2) {};
\end{pgfonlayer}
\begin{pgfonlayer}{edgelayer}
\draw [style=  47] (171) to (24);
\draw [style=  47] (416) to (24);
\draw [style=  47] (775) to (24);
\draw [style=  47] (906) to (24);
\draw [style=  47] (677) to (171);
\draw [style=  47] (690) to (171);
\draw [style=  47] (799) to (171);
\draw [style=  47] (906) to (298);
\draw [style=  47] (906) to (308);
\draw [style=  47] (716) to (677);
\draw [style=  47] (906) to (677);
\draw [style=  47] (677) to (671);
\end{pgfonlayer}
\end{tikzpicture}

%% file: volcanoes/l=3/1009/61.tikz
\begin{tikzpicture}
\begin{pgfonlayer}{nodelayer}
\node [] (0) at (90:2.2) {};
\node [] (0) at (270:2.2) {};
\node [style=red dot 3] (73) at (90:1) {};
\node [style=red dot 3] (344) at (270:1) {};
\node [style=red dot 2] (713) at (45:2) {};
\node [style=red dot 2] (846) at (135:2) {};
\node [style=red dot 2] (101) at (225:2) {};
\node [style=red dot 2] (869) at (315:2) {};
\end{pgfonlayer}
\begin{pgfonlayer}{edgelayer}
\draw [style= 61,in=0, out=0,looseness=0.5] (344) to (73);
\draw [style= 61,in=180, out=180,looseness=0.5] (344) to (73);
\draw [style= 61] (713) to (73);
\draw [style= 61] (846) to (73);
\draw [style= 61] (869) to (344);
\draw [style= 61] (344) to (101);
\end{pgfonlayer}
\end{tikzpicture}

%% file: volcanoes/l=2/7321/99.tikz
\begin{tikzpicture}
\begin{pgfonlayer}{nodelayer}
\node [] (0) at (90:3) {};
\node [] (0) at (270:3) {};
\node [style=red dot 3] (2820) at (90:1) {};
\node [style=red dot 3] (3656) at (270:1) {};
\node [style=red dot 2] (99) at (45:2) {};
\node [style=red dot 2] (5071) at (135:2) {};
\node [style=red dot 2] (921) at (225:2) {};
\node [style=red dot 2] (2083) at (315:2) {};
\node [style=red dot 1] (2791) at (22.5:3) {};
\node [style=red dot 1] (4236) at (67.5:3) {};
\node [style=red dot 1] (1071) at (112.5:3) {};
\node [style=red dot 1] (3039) at (157.5:3) {};
\node [style=red dot 1] (4311) at (202.5:3) {};
\node [style=red dot 1] (5197) at (247.5:3) {};
\node [style=red dot 1] (1623) at (292.5:3) {};
\node [style=red dot 1] (5116) at (337.5:3) {};
\end{pgfonlayer}
\begin{pgfonlayer}{edgelayer}
\draw [style=  22] (2791) to (99);
\draw [style=  22] (2820) to (99);
\draw [style=  22] (4236) to (99);
\draw [style=  22] (3656) to (2820);
\draw [style=  22] (5071) to (2820);
\draw [style=  22] (5071) to (1071);
\draw [style=  22] (5071) to (3039);
\draw [style=  22] (3656) to (921);
\draw [style=  22] (4311) to (921);
\draw [style=  22] (5197) to (921);
\draw [style=  22] (3656) to (2083);
\draw [style=  22] (5116) to (2083);
\draw [style=  22] (2083) to (1623);
\end{pgfonlayer}
\end{tikzpicture}

%% file: volcanoes/l=2/7321/679.tikz
\begin{tikzpicture}
\begin{pgfonlayer}{nodelayer}
\node [] (0) at (90:2) {};
\node [] (0) at (180:1.5) {};
\node [] (0) at (270:2) {};
\node [style=red dot 3] (679) at (90:0) {};
\node [style=red dot 2] (1490) at (0:1) {};
\node [style=red dot 2] (5854) at (180:1) {};
\node [style=red dot 1] (5175) at (-45:2) {};
\node [style=red dot 1] (5385) at (45:2) {};
\node [style=red dot 1] (2525) at (135:2) {};
\node [style=red dot 1] (6008) at (225:2) {};
\end{pgfonlayer}
\begin{pgfonlayer}{edgelayer}
\path (679) edge [style =  22,out=315,in=225,looseness=15] node[above] {} (679);
\draw [style=  22] (1490) to (679);
\draw [style=  22] (5854) to (679);
\draw [style=  22] (5175) to (1490);
\draw [style=  22] (5385) to (1490);
\draw [style=  22] (6008) to (5854);
\draw [style=  22] (5854) to (2525);
\end{pgfonlayer}
\end{tikzpicture}

%% file: volcanoes/l=3/7321/5352.tikz
\begin{tikzpicture}
\begin{pgfonlayer}{nodelayer}
\node [] (0) at (90:2.25) {};
\node [] (0) at (270:2.25) {};
\node [style=red dot 3] (4916) at (90:1) {};
\node [style=red dot 3] (3038) at (120:1) {};
\node [style=red dot 3] (394) at (150:1) {};
\node [style=red dot 3] (6317) at (180:1) {};
\node [style=red dot 3] (1906) at (210:1) {};
\node [style=red dot 3] (2155) at (240:1) {};
\node [style=red dot 3] (4036) at (270:1) {};
\node [style=red dot 3] (6829) at (300:1) {};
\node [style=red dot 3] (4119) at (330:1) {};
\node [style=red dot 3] (655) at (360:1) {};
\node [style=red dot 3] (638) at (390:1) {};
\node [style=red dot 3] (2540) at (420:1) {};
\node [style=red dot 2] (324) at (83.0769:2) {};
\node [style=red dot 2] (1623) at (96.9231:2) {};
\node [style=red dot 2] (1224) at (113.077:2) {};
\node [style=red dot 2] (5546) at (126.923:2) {};
\node [style=red dot 2] (5760) at (143.077:2) {};
\node [style=red dot 2] (6490) at (156.923:2) {};
\node [style=red dot 2] (4236) at (173.077:2) {};
\node [style=red dot 2] (6794) at (186.923:2) {};
\node [style=red dot 2] (2051) at (203.077:2) {};
\node [style=red dot 2] (6043) at (216.923:2) {};
\node [style=red dot 2] (682) at (233.077:2) {};
\node [style=red dot 2] (1161) at (246.923:2) {};
\node [style=red dot 2] (723) at (263.077:2) {};
\node [style=red dot 2] (5116) at (276.923:2) {};
\node [style=red dot 2] (3122) at (293.077:2) {};
\node [style=red dot 2] (4245) at (306.923:2) {};
\node [style=red dot 2] (3877) at (323.077:2) {};
\node [style=red dot 2] (5284) at (336.923:2) {};
\node [style=red dot 2] (2791) at (353.077:2) {};
\node [style=red dot 2] (5932) at (366.923:2) {};
\node [style=red dot 2] (5301) at (383.077:2) {};
\node [style=red dot 2] (6116) at (396.923:2) {};
\node [style=red dot 2] (6226) at (413.077:2) {};
\node [style=red dot 2] (6376) at (426.923:2) {};
\node [style=red dot 2] (6226) at (413.077:2) {};
\node [style=red dot 2] (6376) at (426.923:2) {};
\end{pgfonlayer}
\begin{pgfonlayer}{edgelayer}
\draw [style=  22] (4916) to (324);
\draw [style=  22] (4916) to (1623);
\draw [style=  22] (4916) to (2540);
\draw [style=  22] (6226) to (2540);
\draw [style=  22] (6376) to (2540);
\draw [style=  22] (4916) to (3038);
\draw [style=  22] (5546) to (3038);
\draw [style=  22] (3038) to (394);
\draw [style=  22] (5760) to (394);
\draw [style=  22] (6317) to (394);
\draw [style=  22] (6490) to (394);
\draw [style=  22] (3038) to (1224);
\draw [style=  22] (6794) to (6317);
\draw [style=  22] (2051) to (1906);
\draw [style=  22] (2155) to (1906);
\draw [style=  22] (6043) to (1906);
\draw [style=  22] (6317) to (1906);
\draw [style=  22] (6317) to (4236);
\draw [style=  22] (4036) to (2155);
\draw [style=  22] (2155) to (682);
\draw [style=  22] (2155) to (1161);
\draw [style=  22] (5116) to (4036);
\draw [style=  22] (6829) to (4036);
\draw [style=  22] (4036) to (723);
\draw [style=  22] (6829) to (3122);
\draw [style=  22] (5284) to (4119);
\draw [style=  22] (6829) to (4119);
\draw [style=  22] (6829) to (4245);
\draw [style=  22] (2791) to (655);
\draw [style=  22] (4119) to (655);
\draw [style=  22] (5932) to (655);
\draw [style=  22] (4119) to (3877);
\draw [style=  22] (655) to (638);
\draw [style=  22] (2540) to (638);
\draw [style=  22] (5301) to (638);
\draw [style=  22] (6116) to (638);
\end{pgfonlayer}
\end{tikzpicture}

%% file: volcanoes/l=3/7321/3660.tikz
\begin{tikzpicture}
\begin{pgfonlayer}{nodelayer}
\node [] (0) at (90:2.25) {};
\node [] (0) at (270:2.25) {};
\node [style=red dot 3] (3660) at (90:1) {};
\node [style=red dot 3] (6585) at (150:1) {};
\node [style=red dot 3] (2695) at (210:1) {};
\node [style=red dot 3] (1106) at (270:1) {};
\node [style=red dot 3] (484) at (330:1) {};
\node [style=red dot 3] (245) at (390:1) {};
\node [style=red dot 2] (99) at (77.1429:2) {};
\node [style=red dot 2] (4760) at (102.857:2) {};
\node [style=red dot 2] (1092) at (137.143:2) {};
\node [style=red dot 2] (3940) at (162.857:2) {};
\node [style=red dot 2] (316) at (197.143:2) {};
\node [style=red dot 2] (2060) at (222.857:2) {};
\node [style=red dot 2] (2083) at (257.143:2) {};
\node [style=red dot 2] (6073) at (282.857:2) {};
\node [style=red dot 2] (584) at (317.143:2) {};
\node [style=red dot 2] (1753) at (342.857:2) {};
\node [style=red dot 2] (1733) at (377.143:2) {};
\node [style=red dot 2] (6230) at (402.857:2) {};
\node [style=red dot 2] (1733) at (377.143:2) {};
\node [style=red dot 2] (6230) at (402.857:2) {};
\end{pgfonlayer}
\begin{pgfonlayer}{edgelayer}
\draw [style=  22] (3660) to (99);
\draw [style=  22] (4760) to (3660);
\draw [style=  22] (6585) to (3660);
\draw [style=  22] (484) to (245);
\draw [style=  22] (1733) to (245);
\draw [style=  22] (3660) to (245);
\draw [style=  22] (6230) to (245);
\draw [style=  22] (6585) to (1092);
\draw [style=  22] (6585) to (2695);
\draw [style=  22] (6585) to (3940);
\draw [style=  22] (2695) to (316);
\draw [style=  22] (2083) to (1106);
\draw [style=  22] (2695) to (1106);
\draw [style=  22] (6073) to (1106);
\draw [style=  22] (2695) to (2060);
\draw [style=  22] (584) to (484);
\draw [style=  22] (1106) to (484);
\draw [style=  22] (1753) to (484);
\end{pgfonlayer}
\end{tikzpicture}

%% file: volcanoes/l=3/7321/2525.tikz
\begin{tikzpicture}
\begin{pgfonlayer}{nodelayer}
\node [] (0) at (90:2.25) {};
\node [] (0) at (270:2.25) {};
\node [style=red dot 3] (5175) at (90:1) {};
\node [style=red dot 3] (6008) at (180:1) {};
\node [style=red dot 3] (5385) at (270:1) {};
\node [style=red dot 3] (2525) at (360:1) {};
\node [style=red dot 2] (1031) at (72:2) {};
\node [style=red dot 2] (3050) at (108:2) {};
\node [style=red dot 2] (4123) at (162:2) {};
\node [style=red dot 2] (6462) at (198:2) {};
\node [style=red dot 2] (2134) at (252:2) {};
\node [style=red dot 2] (4881) at (288:2) {};
\node [style=red dot 2] (1127) at (342:2) {};
\node [style=red dot 2] (2225) at (378:2) {};
\node [style=red dot 2] (1127) at (342:2) {};
\node [style=red dot 2] (2225) at (378:2) {};
\end{pgfonlayer}
\begin{pgfonlayer}{edgelayer}
\draw [style=  22] (5175) to (1031);
\draw [style=  22] (6008) to (5175);
\draw [style=  22] (5175) to (2525);
\draw [style=  22] (5385) to (2525);
\draw [style=  22] (5175) to (3050);
\draw [style=  22] (6462) to (6008);
\draw [style=  22] (6008) to (4123);
\draw [style=  22] (6008) to (5385);
\draw [style=  22] (5385) to (2134);
\draw [style=  22] (5385) to (4881);
\draw [style=  22] (2525) to (1127);
\draw [style=  22] (2525) to (2225);
\end{pgfonlayer}
\end{tikzpicture}

%% file: volcanoes/l=3/7321/1490.tikz
\begin{tikzpicture}
\begin{pgfonlayer}{nodelayer}
\node [] (0) at (90:2.25) {};
\node [] (0) at (270:2.25) {};
\node [style=red dot 3] (1490) at (90:1) {};
\node [style=red dot 3] (5854) at (270:1) {};
\node [style=red dot 2] (2321) at (45:2) {};
\node [style=red dot 2] (2659) at (135:2) {};
\node [style=red dot 2] (1622) at (225:2) {};
\node [style=red dot 2] (4214) at (315:2) {};
\end{pgfonlayer}
\begin{pgfonlayer}{edgelayer}
\draw [style= 22] (2321) to (1490);
\draw [style= 22] (2659) to (1490);
\draw [style= 22,in=0, out=0,looseness=0.5] (5854) to (1490);
\draw [style= 22,in=180, out=180,looseness=0.5] (5854) to (1490);
\draw [style= 22] (5854) to (1622);
\draw [style= 22] (5854) to (4214);
\end{pgfonlayer}
\end{tikzpicture}

%% file: volcanoes/l=3/7321/679.tikz
\begin{tikzpicture}
\begin{pgfonlayer}{nodelayer}
\node [] (0) at (90:1.25) {};
\node [] (0) at (270:1.25) {};
\node [style=red dot 3] (679) at (90:0) {};
\node [style=red dot 2] (3724) at (0:1) {};
\node [style=red dot 2] (4611) at (180:1) {};
\end{pgfonlayer}
\begin{pgfonlayer}{edgelayer}
\path (679) edge [style =  22,out=315,in=225,looseness=15] node[above] {} (679);
\path (679) edge [style =  22,out=135,in=45,looseness=15] node[above] {} (679);
\draw [style=  22] (3724) to (679);
\draw [style=  22] (4611) to (679);
\end{pgfonlayer}
\end{tikzpicture}

%% file: volcanoes/l=5/7321/99.tikz
\begin{tikzpicture}
\begin{pgfonlayer}{nodelayer}
\node [] (0) at (90:1.25) {};
\node [] (0) at (270:1.25) {};
\node [style=red dot 3] (1622) at (90:0) {};
\node [style=red dot 2] (99) at (-60:1) {};
\node [style=red dot 2] (584) at (0:1) {};
\node [style=red dot 2] (1006) at (60:1) {};
\node [style=red dot 2] (2060) at (120:1) {};
\node [style=red dot 2] (5430) at (180:1) {};
\node [style=red dot 2] (7165) at (240:1) {};
\end{pgfonlayer}
\begin{pgfonlayer}{edgelayer}
\draw [style=  22] (1622) to (99);
\draw [style=  22] (2060) to (1622);
\draw [style=  22] (5430) to (1622);
\draw [style=  22] (7165) to (1622);
\draw [style=  22] (1622) to (584);
\draw [style=  22] (1622) to (1006);
\end{pgfonlayer}
\end{tikzpicture}